\documentclass{amsart}%
\usepackage{amsfonts, amsmath, amssymb}%
\usepackage{graphicx}
\usepackage{pstricks}

\input xy
\xyoption{all}

\newtheorem{theorem}{Theorem}
\theoremstyle{definition}

\newtheorem{corollary}{Corollary}
\newtheorem{definition}{Definition}

\newtheorem{proposition}{Proposition}
\newtheorem{remark}{Remark}
\newtheorem{question}{Question}
\numberwithin{equation}{section}
\newcommand{\brak}[1]{\langle #1\rangle}

\DeclareMathOperator{\Hom}{Hom}

\DeclareMathOperator{\id}{Id}
\DeclareMathOperator{\rk}{rk}
\def\co{\colon\thinspace}

\begin{document}

\title{The $sl(2)$ foam cohomology via a TQFT}
\author{Carmen Caprau}

\address{Department of Mathematics, California State University, Fresno, CA 93740 USA}
\email{ccaprau@csufresno.edu}
\urladdr{}

\date{}
\thanks{The author was supported in part by NSF grant DMS 0906401}

\begin{abstract}
We construct a cohomology theory for oriented links using singular cobordisms and a special type of $2$-dimensional Topological Quantum Field Theory (TQFT), categorifying the quantum $sl(2)$ invariant. In particular, we give a description of the universal dot-free $sl(2)$ foam cohomology for links via a $2$-dimensional TQFT.

\end{abstract}
\maketitle

\section{Introduction}

The author constructed in~\cite{CC1} the universal $sl(2)$ tangle cohomology via dotted foams modulo a finite set of local relations, using Bar-Natan's~\cite{BN1} approach to local Khovanov homology and Khovanov's work in~\cite{Kh2}. We refer to this as the \textit{ universal $sl(2)$ foam cohomology}. The construction starts at the topological picture made of resolutions associated to an oriented tangle diagram, called \textit{webs}, and of dotted seamed cobordisms between webs, called \textit{foams}. A web is a disjoint union of piecewise oriented 1-manifolds containing $2k$ bivalent vertices ($k \geq 0$), so that for any two adjacent vertices, one is a sink and the other a source. A foam is a piecewise oriented cobordism between such webs, and might contain some dots (as Khovanov's foams in~\cite{Kh2}). To switch from the geometric world to an algebraic one, one applies a `tautological' functor and arrives at a cochain complex of modules and module homomorphisms whose cohomology is a bigraded tangle invariant. Restricting to the case of links and considering the graded Euler characteristic of this invariant, one recovers the quantum $sl(2)$ polynomial. Much as Bar-Natan~\cite{BN1} did in his  approach to Khovanov's homology for tangles, the invariance was proved at the level of the topological picture by considering the set of foams modulo local relations. 

The universal $sl(2)$ foam cohomology corresponds (in a certain way) to the Frobenius algebra structure defined on the ring $\mathcal{A} =R[ X]/(X^2 -hX -a)$, for which the counit and comultiplication maps are given in the basis $\{1, X\}$ by 
\begin{eqnarray}\label{Frobenius structure}
\begin{cases} \epsilon(1) = 0\\ \epsilon(X) = 1, \end{cases} \quad \begin{cases}
 \Delta(1) = 1 \otimes X + X \otimes 1-h 1 \otimes 1\\ 
 \Delta(X) = X \otimes X + a 1 \otimes 1.\end{cases}
 \end{eqnarray}
 The ground ring is $R =  \mathbb{Z}[i][a, h]$, where $a$ and $h$ formal parameters and $i$ is the primitive fourth root of unity. Imposing $a = h = 0$ we arrive at an isomorphic version of the original Khovanov homology~\cite{BN0, Kh1}. Moreover, letting $h = 0$ we obtain Bar-Natan's theory~\cite{BN1}, while setting $a = 1, h = 0$, Lee's theory~\cite{L} is recovered. 

In \cite{CC2} the author provided the tools that lead to efficient computations of the dotted foam  cohomology groups, and also showed that if $2^{-1}$ exists in the ground ring, then one can work with a purely topological version of the foam theory, where no dots are needed on foams. 

 The advantage of the $sl(2)$ foam cohomology versus the original Khovanov homology and Bar-Natan's  work in~\cite{BN1} is the well-defined functorial property with respect to tangle/link cobordisms relative to boundaries. In particular, it fixes the functoriality of the original Khovanov invariant (for details we refer the reader to~\cite{CC0, CC1}). We would like to remark that the same result was obtained by Clark, Morrison and Walker in \cite{CMW} (they used the term `disoriented' cobordisms for the 2-cobordisms which the author called foams). 
 
The original Khovanov homology uses a $2$-dimensional Topological Quantum Field Theory (TQFT) corresponding to a certain Frobenius algebra. It is then worthwhile to ask if one can obtain the $sl(2)$ foam cohomology by applying some type of a TQFT rather than a tautological functor, and then what the defining algebra structure of this TQFT is.

The purpose of this paper is to answer the above questions, thus to provide an algebraic framework for the universal $sl(2)$ foam cohomology for links using a special type of $2$-dimensional TQFT defined on foams (in doing so, we will consider the dot-free foam theory~\cite[Section 4]{CC2}). The first step in achieving this goal was made by the author in~\cite{CC3}, were she considered a particular type of foams, called \textit{singular 2-cobordisms}, and showed that the category $\textbf{Sing-2Cob}$ of singular 2-cobordisms admits a completely algebraic description as the free symmetric monoidal category on a, what the author called, \textit{twin Frobenius algebra}. A twin Frobenius algebra $(C, W, z, z^*)$ consists of a commutative Frobenius algebra $C,$ a symmetric Frobenius algebra $W,$ and an algebra homomorphism $z \co C \to W$ with dual $z^* \co W \to C,$ satisfying some additional conditions. The author also introduced in~\cite{CC3} a special type of 2-dimensional TQFT, so-called \textit{twin} TQFT, defined on singular 2-cobordisms and showed that it is equivalent to a twin Frobenius algebra in a symmetric monoidal category. 

The category $\textbf{Sing-2Cob}$ of singular 2-cobordisms, as considered in~\cite{CC3}, has as objects clockwise oriented circles and \textit{bi-webs}; a bi-web is a closed web with exactly two bivalent vertices. In this paper we enhance the category of singular cobordisms by allowing counterclockwise oriented circles as well, and thus working in a more general setup. We abuse of notation and denote the enhanced category by $\textbf{Sing-2Cob}$ as well. The algebraic structure of the enhanced category is coined by the term: \textit{enhanced twin Frobenius algebra}. Then we use this category to complete the second step in achieving the goal of constructing a cohomology theory for oriented links---via singular cobordisms and a certain type of 2-dimensional TQFT defined on them---which is isomorphic to the universal dot-free $sl(2)$ foam cohomology. We chose to work with the dot-free version of the foam theory as we are interested in a purely topological construction; the `downside' of this is the need of $2^{-1}$ in the ground ring.

It is clear that in searching for an answer to the above mentioned goal, one wishes to keep the well-defined functoriality of the $sl(2)$ foam cohomology. This is our reason for considering the general setup containing both clockwise and counterclockwise oriented circles, although the method described in this paper can be employed when working with simpler singular 2-cobordisms whose boundaries are bi-webs and circles with only one type of orientation. However, the author conjectures that by allowing both orientations for a circle, the resulting cohomology via a TQFT is properly functorial, while in the simpler case the functorial property holds only up to a sign.

The possible TQFT describing the $sl(2)$ foam cohomology should satisfy the local relations used in the construction of this theory. In~\cite[Example 1 of Section 5]{CC3} the author gave an example of a twin TQFT satisfying the $sl(2)$ foam local relations, but it failed short to satisfy the laws of a twin Frobenius algebra (to be specific, the ``genus-one condition'' of a twin Frobenius algebra holds with a minus sign). This tells us that, although the enhanced category $\textbf{Sing-2Cob}$ is governed by an enhanced twin Frobenius algebra, the algebraic structure that underlines the $sl(2)$ foam cohomology must be a minor modification of that of an enhanced twin Frobenius algebra. The key to finding the appropriate algebraic structure lies within the results of the foam theory: Let us recall that the local relations that one imposes on the set of foams imply a few other local relations, among whom there are the ``curtain identities" (CI-1) and (CI-2) (depicted here in Section~\ref{sec:univKh}). A particular form of the curtain identities are those given in equations~\eqref{imposed-relation1} and~\eqref{imposed-relation2}, hence we need to mod out the morphisms of the category $\textbf{Sing-2Cob}$ by these latter relations. We show that the new category, let us denote it by $\textbf{eSing-2Cob}$, admits an algebraic description as the free symmetric monoidal category on an \textit{identical twin Frobenius algebra} (we thought this is an appropriate name for the corresponding algebraic structure). An identical twin Frobenius algebra $(C, W, z_1, z_1^*,z_2, z_2^*)$ consists of two commutative Frobenius algebras $C$ and $W$,  and two algebra isomorphisms $z_1, z_2 \co C \to W$ with duals $z_1^*, z_2^* \co W \to C$ and inverse isomorphisms $iz_1^*, -iz_2^*$, respectively, where $i^2 = -1$. The 2-dimensional TQFT of interest to us is defined on the category $\textbf{eSing-2Cob}$, and is equivalent to an identical twin Frobenius algebra; hence we call it an \textit{identical twin TQFT}.

Here is a brief plan of the paper. Section~\ref{sec:univKh} provides a review of the construction giving rise to the universal dot-free $sl(2)$ foam cohomology. In Section~\ref{sec:twin Frobenius extensions} we set up the ground for the paper and introduce the category $\textbf{Sing-2Cob}$ of singular 2-cobordisms and its related category, $\textbf{eSing-2Cob}$. We also describe the algebraic structures of the two categories. Section~\ref{sec:TQFT} provides the (identical twin) TQFT that we are concerned with here, and shows that it satisfies the local relations used in the foam theory without dots. The most important part of the paper is Section~\ref{sec:link cohomology}, where we construct a cohomology for oriented links, isomorphic to the universal dot-free $sl(2)$ foam cohomology. The `topology to algebra' functor used in this construction is the TQFT of Section~\ref{sec:TQFT}. 


\section{Review of the universal dot-free $sl(2)$ foam cohomology}\label{sec:univKh}

 Let us briefly recall the construction of the universal $sl(2)$ foam cohomology introduced in~\cite{CC1}. For the purpose of this paper, we consider here the purely topological version of the foam theory in which no dots are present on foams (for more details we refer to~\cite[Section 4]{CC2}). The construction in~\cite{CC0, CC1, CC2} was given for the general case of tangles, but here we restrict our attention to links. 
 
We mentioned in the introduction that the ring $\mathcal{A} =R[ X]/(X^2 -hX -a)$, endowed with the Frobenius algebra structure defined by the counit and comultiplication maps given in~\eqref{Frobenius structure},  plays a key role in the universal $sl(2)$ foam cohomology. In the dot-free version of this theory, the ground ring is $R =  \mathbb{Z}[\frac{1}{2}, i][a, h]$, and it is graded by setting $\mbox{deg}(a) = 4, \,\, \mbox{deg}(h) = 2 \,\, \text{and} \,\, \mbox{deg} (1) = \mbox{deg}(i) = 0$. The ring $\mathcal{A} = \brak{1, X}_R$ is also graded by letting $\mbox{deg}(1) = -1$ and $\mbox{deg}(X) = -1$.

Given an $n$-crossing diagram $D$ representing an oriented link $L$, we form an $n$-dimensional cube $\mathcal{U}$ of resolutions (called `webs') and cobordisms between them (called `foams'). This cube is then flattened to a Bar-Natan type~\cite{BN1} \textit{formal complex} in the additive category $\textbf{Foam}_{/\tilde{\ell}}$ whose objects are formally graded webs and whose morphisms are formal linear combinations of foams, modulo certain local relations $\tilde{\ell}$. Some further degree and height shifts are applied, which depend only on the number of positive and negative crossings in $D$, to arrive at a formal complex $[D]$ living in the category \textit{Kof}: = Kom(Mat($\textbf{Foams}_{/\tilde{\ell}}$)) of complexes of formal direct sums of objects in $\textbf{Foams}_{/\tilde{\ell}}$, and which is an up-to-homotopy invariant.

The reason for working with foams (which are seamed or disoriented cobordisms) rather than ordinary cobordisms is that in the $sl(2)$ foam cohomology one categorifies the \textit{oriented} version of the quantum $sl(2)$ polynomial of $L$. The state summation for this polynomial is given by the formula $ P_2(L) = P_2(D) := \sum_{\Gamma} \pm q^{\alpha(\Gamma)}\brak{\Gamma}$, where the sum is over all resolutions $\Gamma$ of $D,$ and the exponents $\alpha(\Gamma)$ and the $\pm$ sign are determined by the relations:  
\begin{eqnarray} \label{eq:polyn_inv}
\raisebox{-8pt}{\includegraphics[height=0.3in]{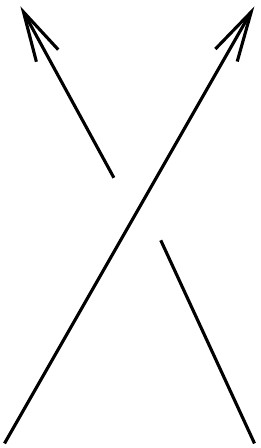}} \,= \,q\,\,\,\,\,\raisebox{-8pt} {\includegraphics[height=0.3in]{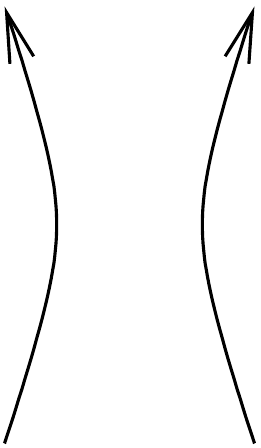}} - q^2\,\,\,\,\,\raisebox{-8pt} {\includegraphics[height=0.3in]{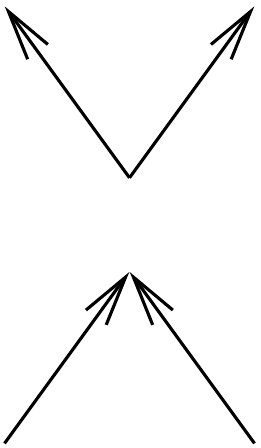}}\hspace{2cm}
\raisebox{-8pt}{\includegraphics[height=0.3in]{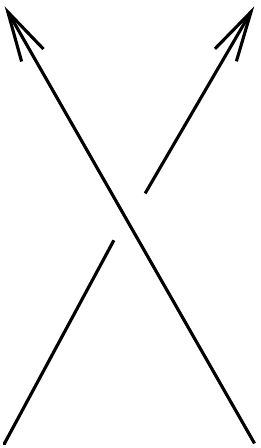}} \,= \,q^{-1}\,\raisebox{-8pt} {\includegraphics[height=0.3in]{orienres.pdf}} - q^{-2}\, \raisebox{-8pt} {\includegraphics[height=0.3in]{singres.pdf}}.
\end{eqnarray}
The \textit{bracket polynomial} $\brak{\Gamma}$ associated to a web $\Gamma$ is an element of the Laurent polynomial ring $\mathbb{Z}[q, q^{-1}],$ and it is evaluated via the skein relations:
\begin{eqnarray} \label{skein_circle}
\brak{\raisebox{-3pt}{\includegraphics[height=0.15in]{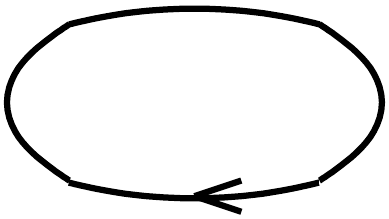}} \bigcup \Gamma} = (q + q^{-1}) \brak{\Gamma} =\brak{\raisebox{-3pt}{\includegraphics[height=0.15in]{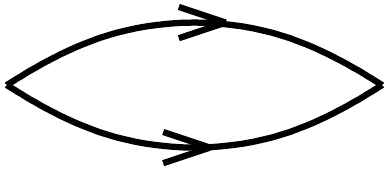}} \bigcup \Gamma},\\
\brak{\raisebox{-4pt}{\includegraphics[height=0.12in]{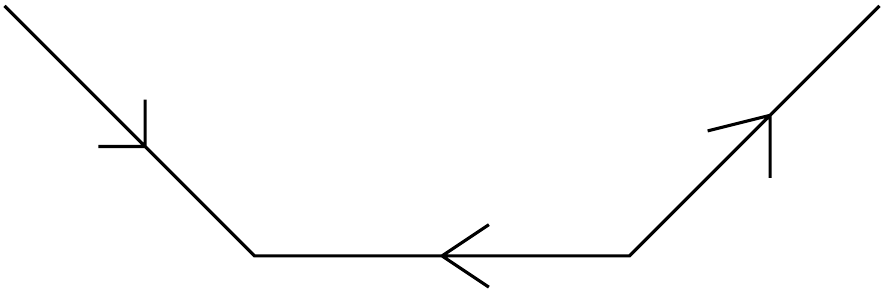}}} = 
\brak{\raisebox{-4pt}{\includegraphics[height=0.10in]{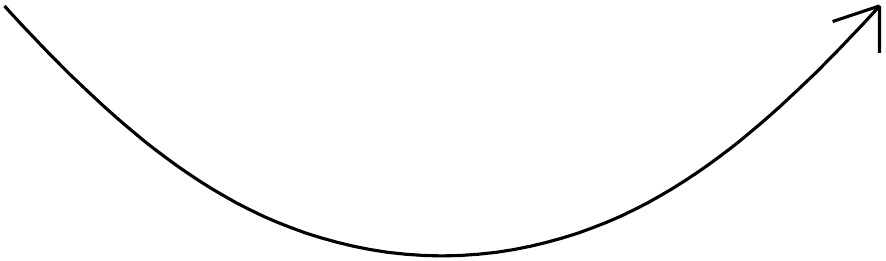}}}, \quad
\brak{\raisebox{-4pt}{\includegraphics[height=0.12in]{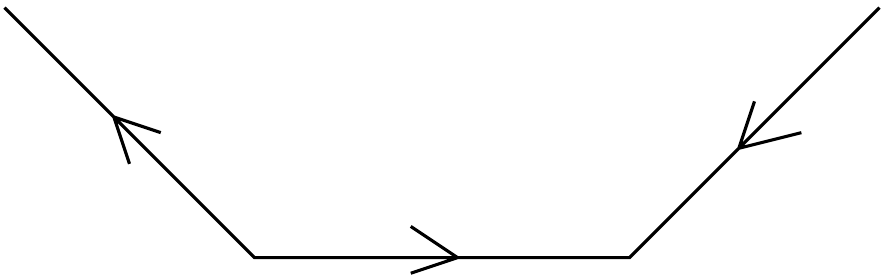}}} = 
\brak{\raisebox{-4pt}{\includegraphics[height=0.10in]{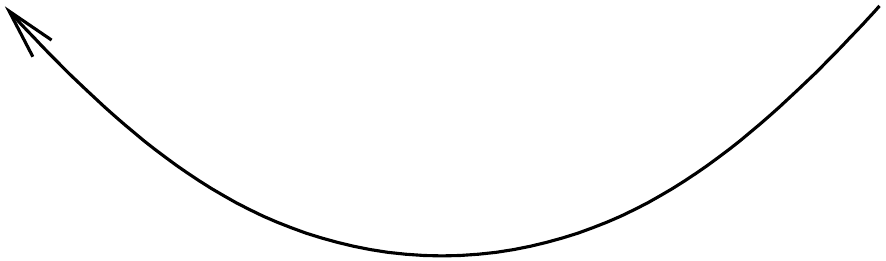}}}. \label{skein_rembv}
\end{eqnarray}
   
Let us now say a few words about our webs and foams. A diagram $\Gamma$ obtained by resolving each crossing in $D$ in either the oriented or disoriented fashion is a collection of disjoint circle graphs, called \textit{webs}. A web is a planar graph with bivalent vertices such that the two edges incident to a vertex are either both entering the vertex or both leaving the vertex. Webs without vertices are also allowed. For every \textit{singular resolution} \raisebox{-8pt} {\includegraphics[height=0.3in]{singres.pdf}} there is an ordering of the two edges that meet at a vertex. We say that the edge that `goes in' or `goes out' from the right, respectively, is the \textit{preferred edge} for the corresponding vertex. Two adjacent vertices of a web are called \textit{of the same type} if the edge they share is for both vertices either the preferred edge or the non-preferred one. Otherwise, the vertices are called \textit{of different type}.
The local relations depicted in~\eqref{skein_rembv} say that we can `remove' adjacent pairs of vertices of the same type. 
 
A \textit{foam} is a piecewise oriented cobordism between two webs $\Gamma_0$ and $\Gamma_1$, regarded up to boundary-preserving isotopy. We draw foams with their source at the top and their target at the bottom (notice that this is the opposite convention of that used in~\cite{CC0, CC1, CC2}), and we compose them by placing one on top the other. Foams have \textit{singular arcs} (and/or \textit{singular circles})---represented by red curves---where orientations disagree; that is, the two facets incident with a given singular arc have opposite orientation, and because of this, they induce the same orientation on that arc. For each singular arc of a foam, there is an ordering of the facets that are incident with it, in the sense that one of the facets is the \textit{preferred facet} for that singular arc. This ordering is induced by the ordering of the edges corresponding to bivalent vertices that the singular arc connects, in the following sense: the preferred facet of a singular arc contains in its boundary the preferred edges of the two bivalent vertices that the singular arc connects. In particular, a pair of vertices can be connected by a singular arc only if the above rule is satisfied. We indicate facets' ordering near a singular arc by using labels 1 and 2. 

It is necessary to repeat here the local relations that appear in the definition of the category $\textbf{Foam}_{/\tilde{\ell}}$: 
$$\xymatrix@R=2mm{
 (S) \ \ \ \ \ \ \ \ \ 
  \raisebox{-10pt}{\includegraphics[width=0.4in]{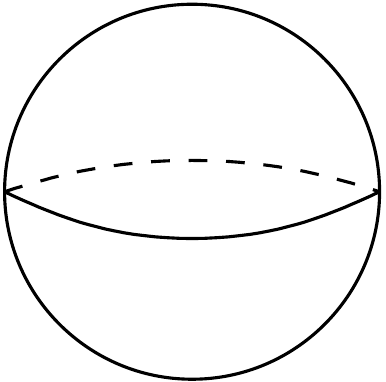}}=0 \hspace{3cm} \raisebox{-5pt}{\includegraphics[width=0.6in]{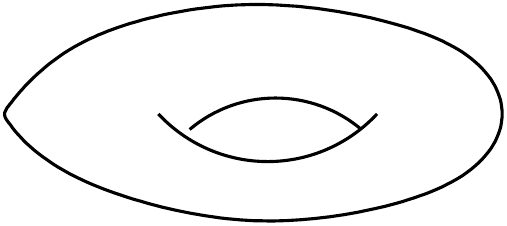}}=2 & (T) \\
\raisebox{-13pt}{\includegraphics[height=.4in]{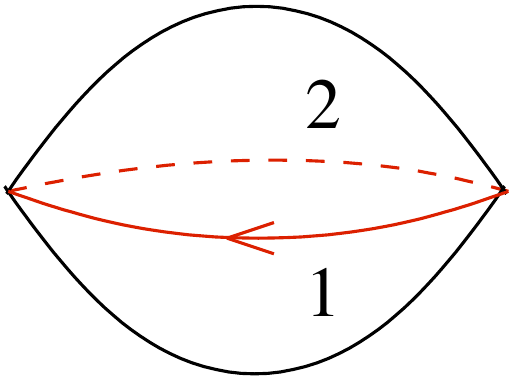}} = 0 =  \raisebox{-13pt}{\includegraphics[height=.4in]{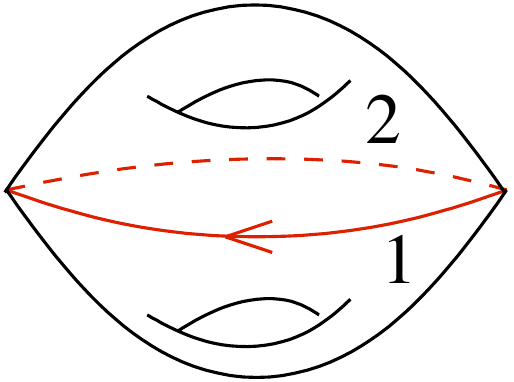}}\, \hspace{1.5cm} \raisebox{-13pt}{\includegraphics[height=.4in]{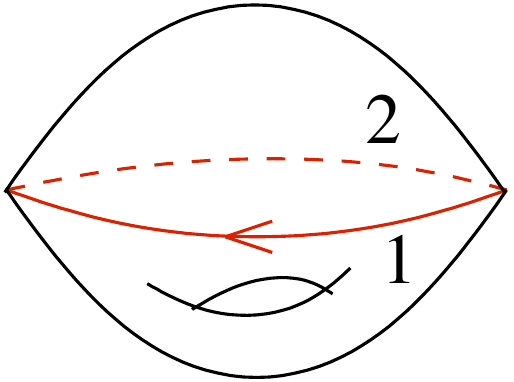}} = 2i = -\,\raisebox{-13pt}{\includegraphics[height=.4in]{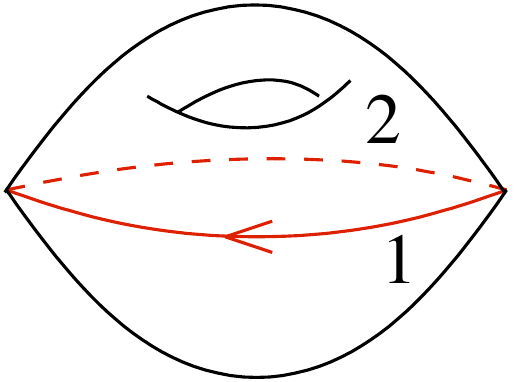}}  & (UFO)\\
\ (G2)\ \ \ \ \ \   
\raisebox{-8pt}{\includegraphics[height=.35in]{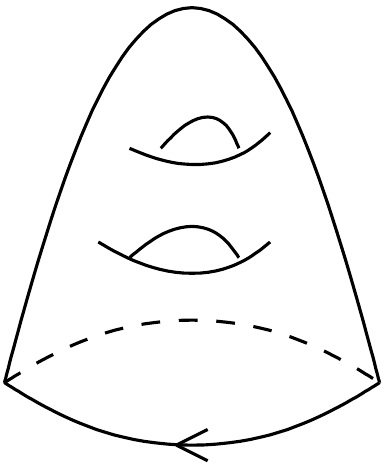}} = (h^2 + 4a)\, \raisebox{-8pt}{\includegraphics[height=.35in]{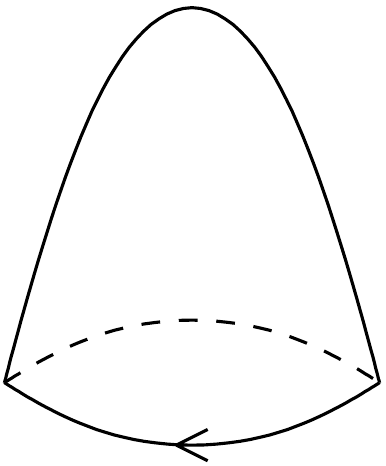}} \hspace{1cm}  \hspace{1cm} \raisebox{-22pt}{\includegraphics[height=.7in]{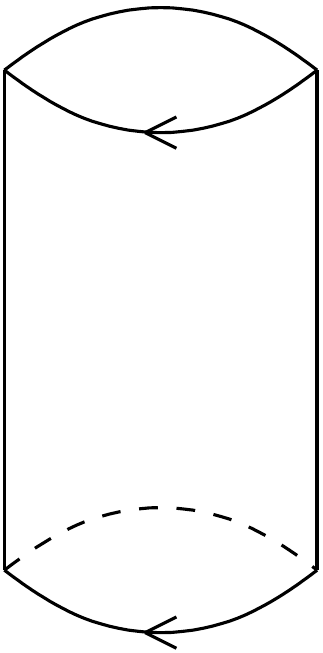}} = \displaystyle\frac{1}{2} \raisebox{-22pt}{\includegraphics[height=.7in]{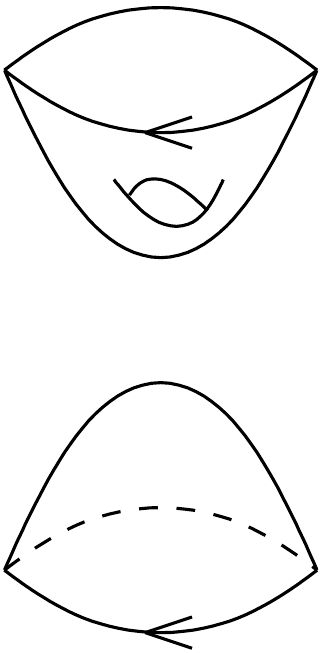}} + \frac{1}{2}\raisebox{-22pt}{\includegraphics[height=.7in]{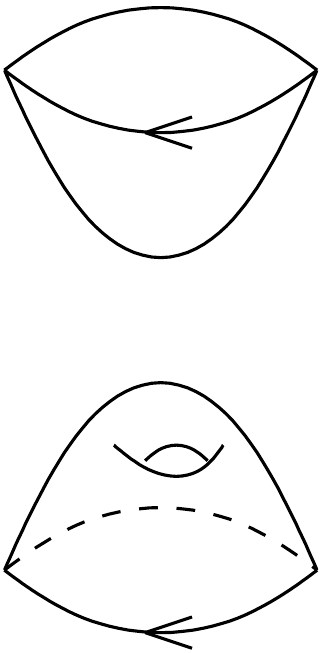}} &(SF)
}$$
To be precise, we mod out the set of morphisms of the category of foams by the local relations $\tilde{\ell}$ = (S, T, UFO, G2, SF). We remind the reader that the local relations $\ell =$ (2D, SF, S, UFO) used in the dotted $sl(2)$ foam cohomology are slightly different (see~\cite[Section 3.2]{CC1}).

For the purpose of this paper, it is important to recall that the imposed local relations of the foam theory (either dot-free or dotted theory) imply the following \textit{curtain identities}:
\begin{align}
\raisebox{-13pt}{\includegraphics[height =0.5in]{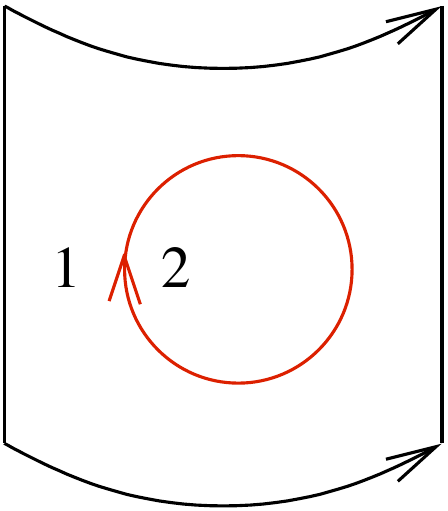}}&= i
\,\raisebox{-13pt}{\includegraphics[height=0.5in]{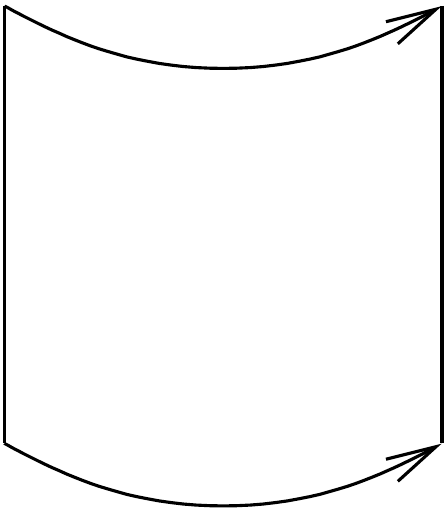}} \hspace{1.5cm}
\raisebox{-13pt}{\includegraphics[height =0.5in]{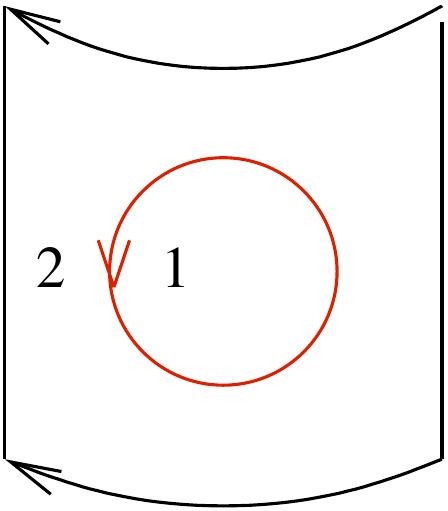}} =-i
\,\raisebox{-13pt}{\includegraphics[ height=0.5in]{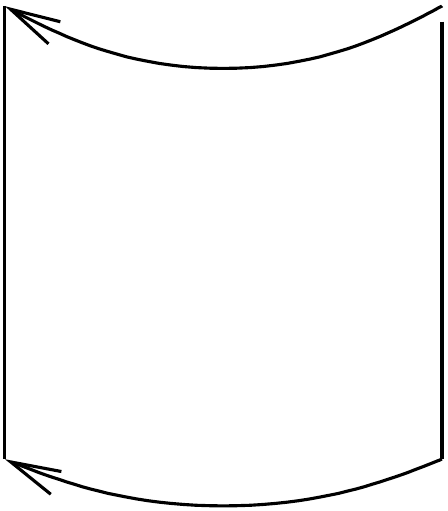}} & (CI-1) \\
\raisebox{-18pt}{\includegraphics[height=0.6in]{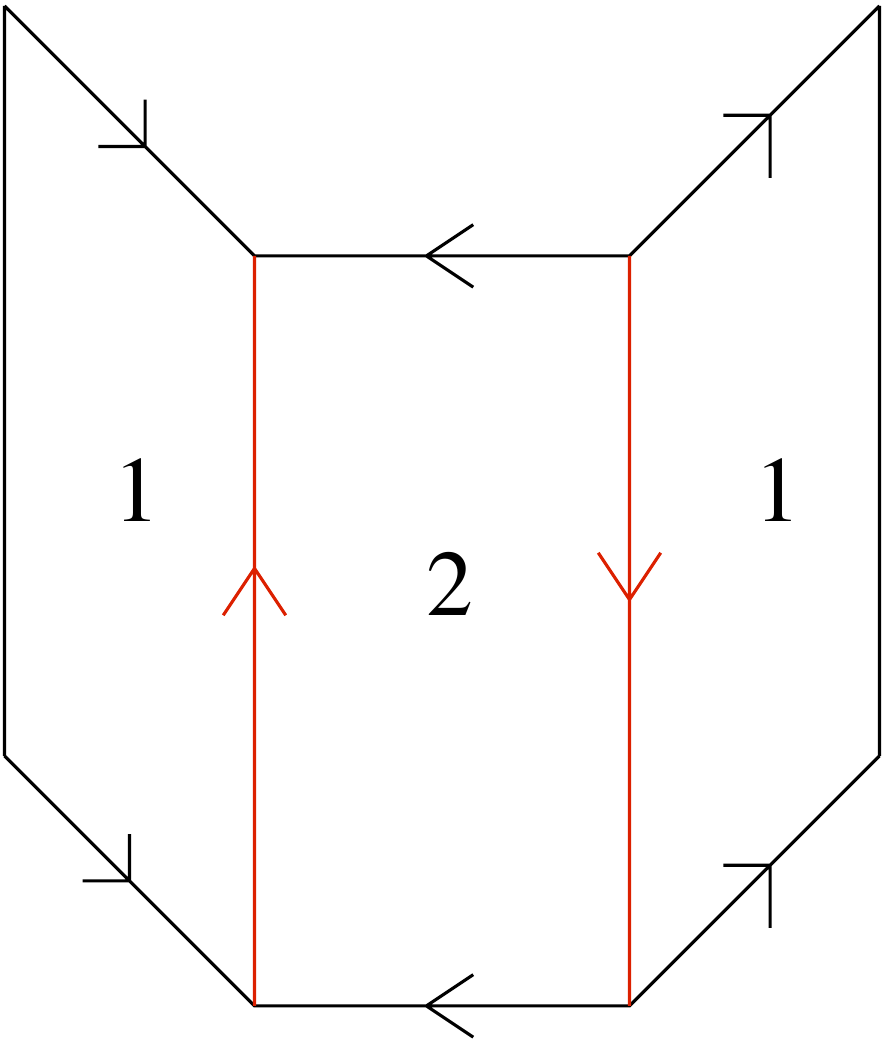}} &= -i \,
\raisebox{-18pt}{\includegraphics[height=0.6in]{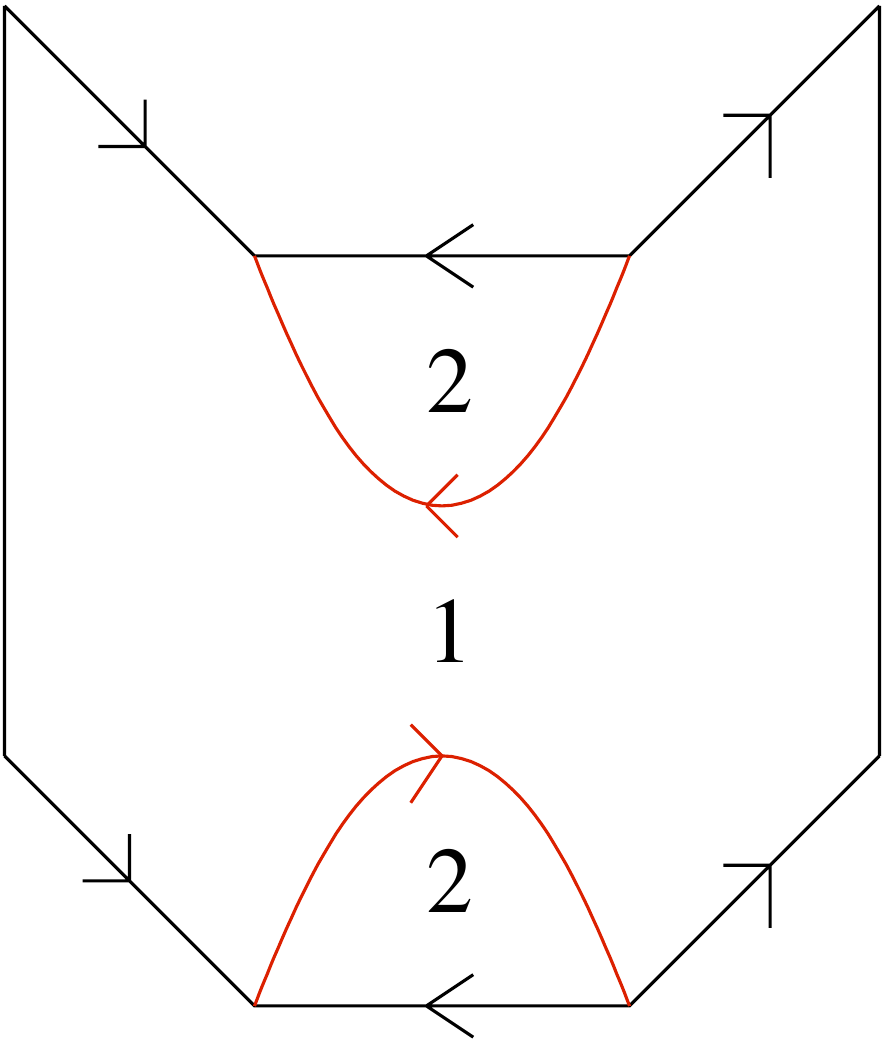}} \hspace{1cm}
\raisebox{-18pt}{\includegraphics[height=0.6in]{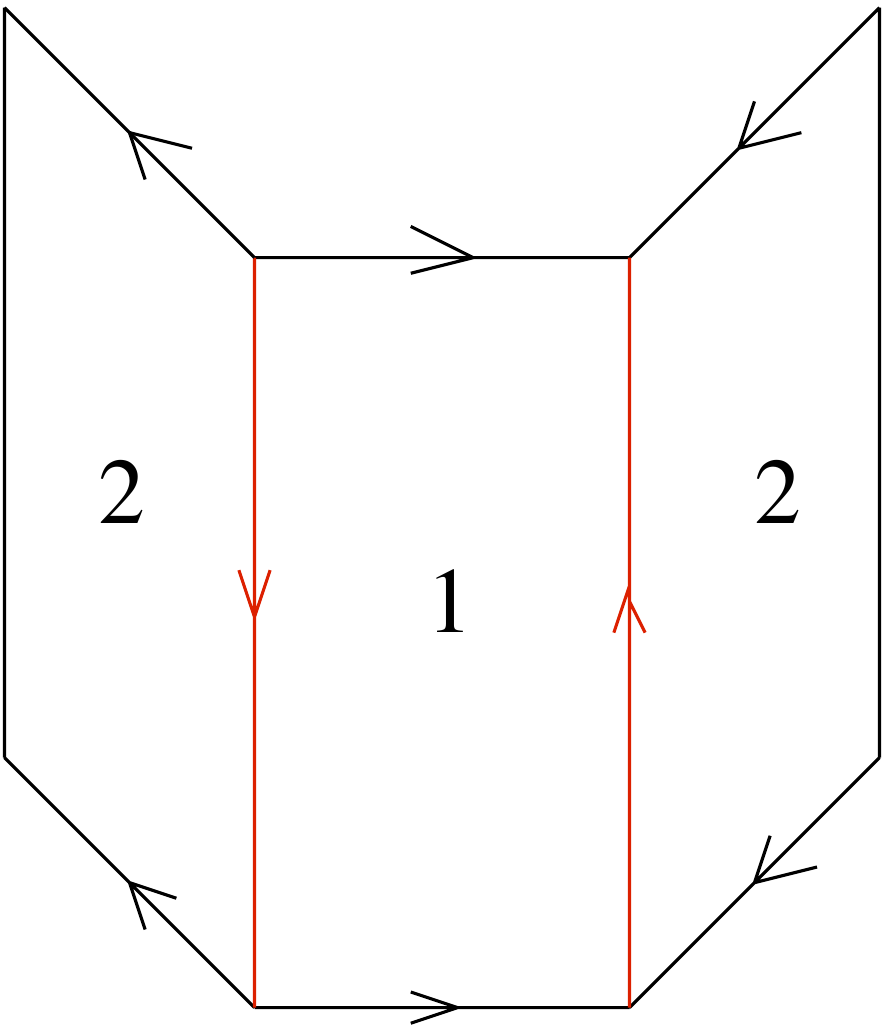}}= i\,
\raisebox{-18pt}{\includegraphics[height=0.6in]{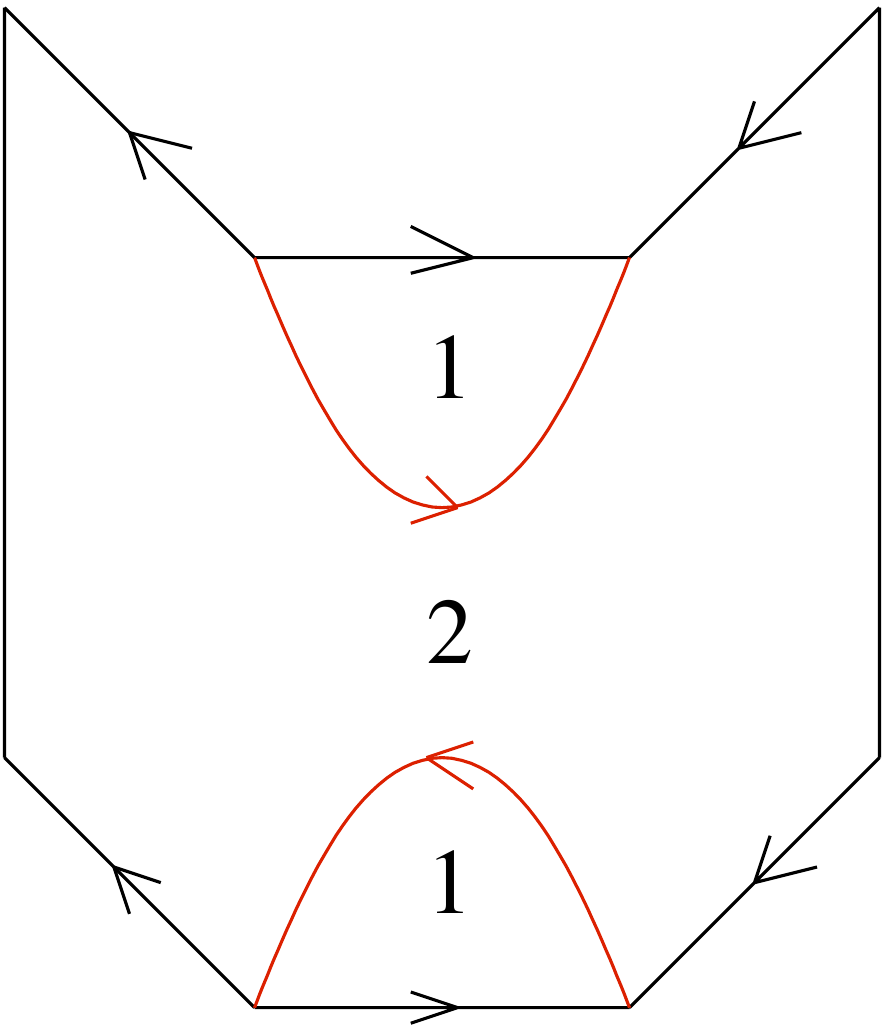}} & (CI-2)
\end{align}
as well as the the \textit{cutting-neck} relation (CN) depicted below:
$$\xymatrix@R=2mm{
 \raisebox{-22pt}{\includegraphics[width=0.3in, height=.8in]{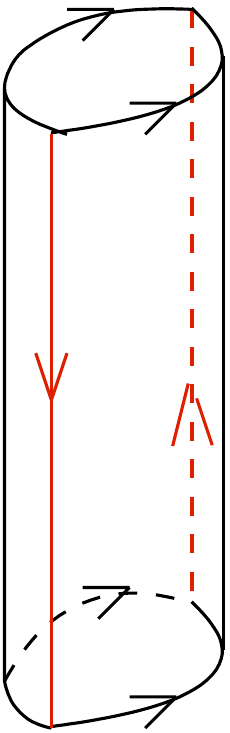}}= \displaystyle \frac{i}{2}
\raisebox{-25pt}{\includegraphics[width=0.35in, height=0.85in]{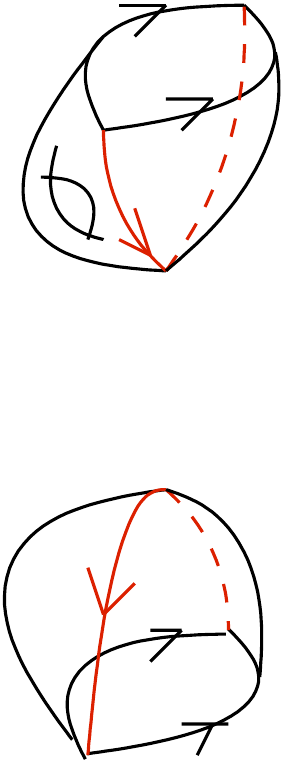}} + \frac{i}{2}
\raisebox{-25pt}{\includegraphics[width=0.35in, height=0.85in]{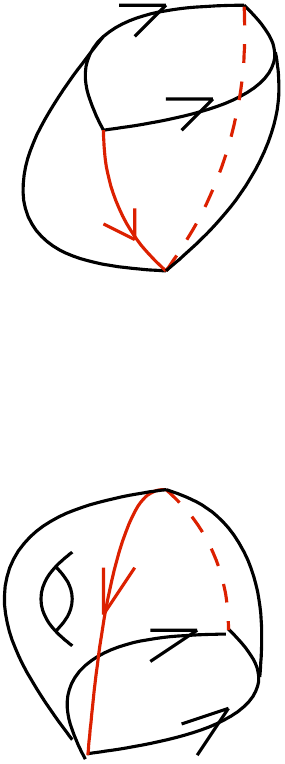}} & (CN)}$$
The following isomorphisms hold in $\textbf{Foam}_{/\tilde{\ell}}$ and are consequences of the ``curtain identities": 
\begin{align}\label{isomorphisms}
\raisebox{-25pt}{\includegraphics[height=.8in]{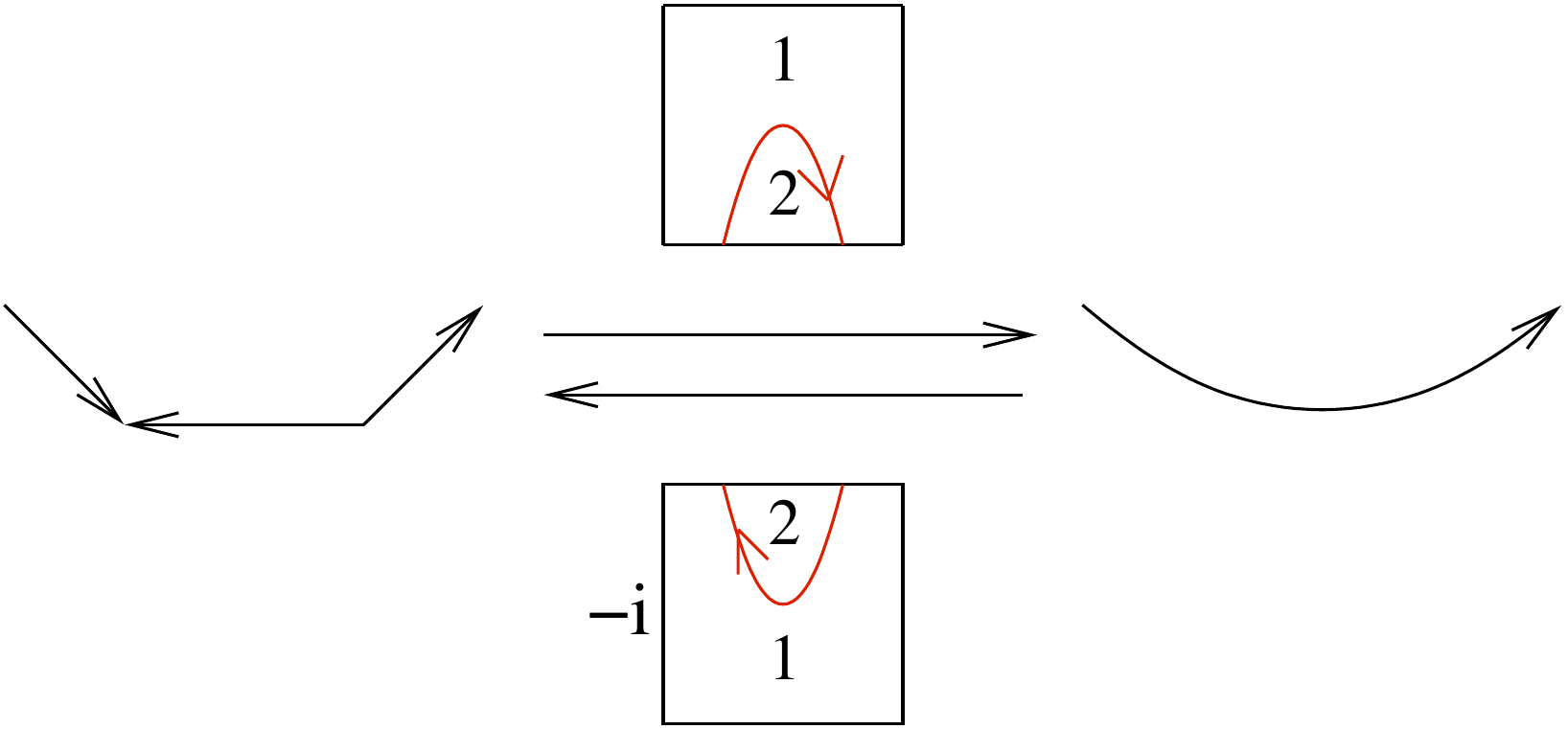}}
 \hspace{1.5cm}
\raisebox{-25pt}{\includegraphics[height=.8in]{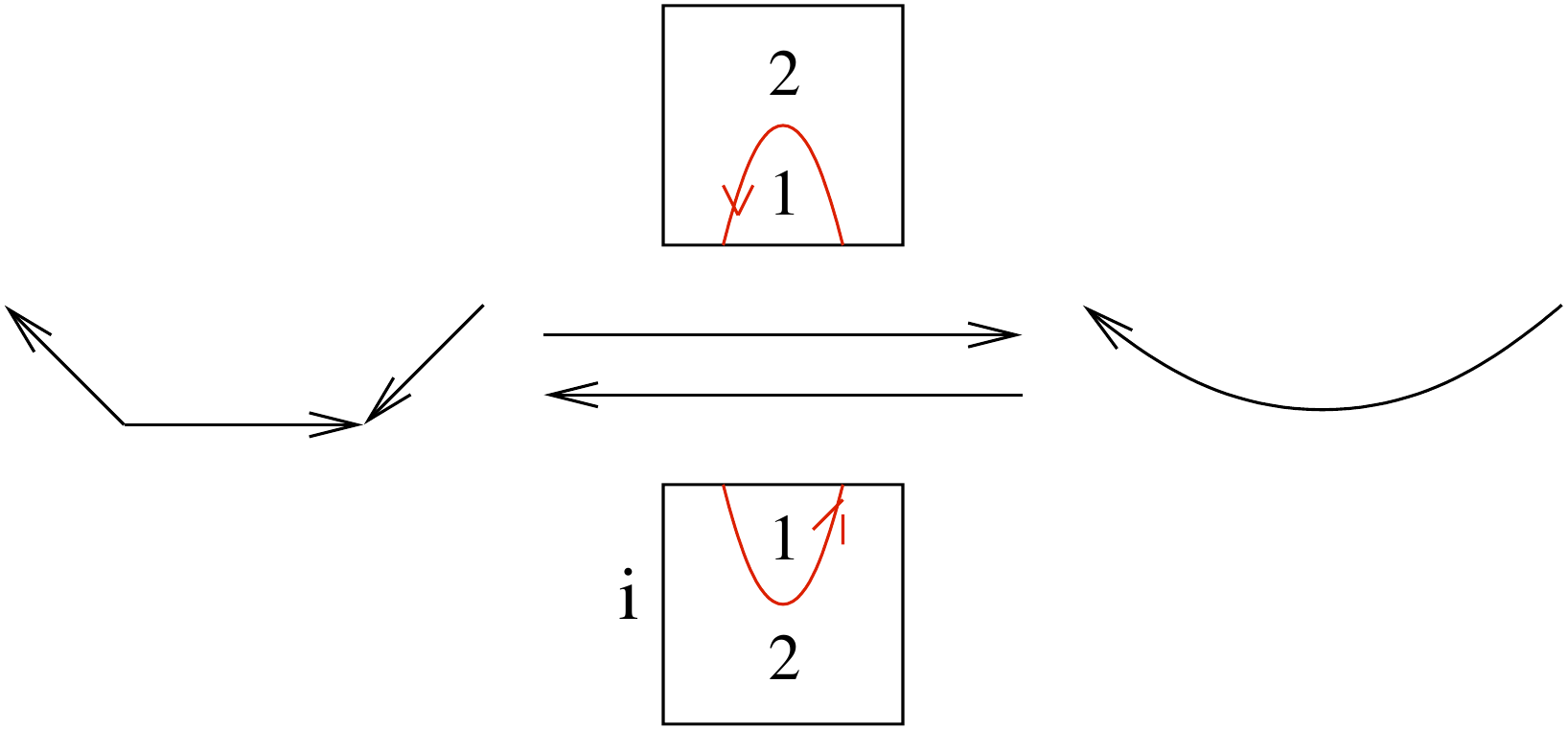}}\end{align}
We remark that the ``curtain identities" are the local relations imposed on the set of cobordisms in~\cite{CMW}. 

The category $\textbf{Foam}_{/\tilde{\ell}}$ is graded by defining $\deg(S) = -\chi(S)$ for every foam $S$ in $\textbf{Foam}_{/\tilde{\ell}},$ where $\chi$ is the Euler characteristic.  
 
 To obtain a computable invariant, one applies a  \textit{tautological} degree-preserving functor 
 $\mathcal{F} \co \textbf{Foams}_{/\tilde{\ell}} \to \textbf{R-Mod}$,
 which associates to a web $\Gamma$ the set $ \Hom_{\textbf{Foam}_{/\tilde{\ell}}}(\emptyset, \Gamma)$ of all morphisms from the empty 1-manifold $\emptyset$ to $\Gamma$, and which associates to a homomorphism between two webs the obvious homomorphism obtained by composition. The functor $\mathcal{F}$ mimics the web skein relations~\eqref{skein_circle} and~\eqref{skein_rembv}, and extends in a straightforward manner to the category \textit{Kof}. $\mathcal{F}([D])$ is an ordinary complex of graded $R$-modules whose cohomology is an up-to-homotopy bigraded invariant of $L$, and whose graded Euler characteristic is the quantum $sl(2)$ polynomial $P_2(L)$.


\section{Twin Frobenius algebras and singular 2-cobordisms}\label{sec:twin Frobenius extensions}

Let $\mathbf{C}$ be an arbitrary symmetric monoidal category with unit object $\textbf{1} \in \mathbf{C}.$  As examples of such a category, we are interested in the category $\textbf{Vect}_k$ of vector spaces over a field $k$ and $k$-linear maps, and in the category $\textbf{R-Mod}$ of $R$-modules and module homomorphisms, where $R$ a commutative ring; in particular, the category $\textbf{Ab}$ of abelian groups is of interest for us.  

In the next two subsections we generalize the concepts introduced in~\cite{CC3}.
\subsection{Enhanced twin Frobenius algebras}

Recall that a \textit{Frobenius algebra} in $\mathbf{C}$, $(C, m, \iota, \Delta, \epsilon)$, is an associative algebra $(C,m, \iota)$ with unit $\iota \co \textbf{1} \to C$ and multiplication $m \co C \otimes C \to C$ which is also a coassociative algebra $(C, \Delta, \epsilon)$ with counit $\epsilon \co C \to \textbf{1}$ and comultiplication $\Delta \co C \to C \otimes C$ satisfying 
$ (m \otimes \id_C) \circ (\id_C \otimes \Delta) = \Delta \circ m = (\id_C \otimes m) \circ (\Delta \otimes \id_C). $

A Frobenius algebra $(C, m, \iota, \Delta, \epsilon)$ is called \textit{commutative} if $m \circ \tau = m,$ and is called \textit{symmetric} if $\epsilon \circ m = \epsilon \circ m \circ \tau,$ where $\tau \co C \otimes C \to C \otimes C,\, a \otimes b \mapsto \ b \otimes a.$
A \textit{homomorhism of Frobenius algebras} $f \co C \to C'$ is a linear map which is both a homomorphism of unital algebras and counital coalgebras.

\begin{definition}
 An  \textit{enhanced twin Frobenius algebra} $\mathbf{eT}: = (C, W, z_1, z_1^*, z_2, z_2^*)$ in $\mathbf{C}$ consists of
\begin{itemize}
\item  a commutative Frobenius algebra $C = (C, m_C, \iota_C, \Delta_C, \epsilon_C)$,
\item a symmetric Frobenius algebra $W = (W, m_W, \iota_W, \Delta_W, \epsilon_W)$,
\item four morphisms $z_1, z_2 \co C \to W$ and $z_1^*, z_2^* \co W \to C$, 
\end{itemize}
such that $z_1$ and  $z_2$ are homomorphisms of algebra objects in $\mathbf{C},$ and such that the following hold for $k = 1,2$:
\begin{equation}
\epsilon_C \circ m_C \circ (\id_C \otimes z_{k}^*) = \epsilon_W \circ m_W \circ (z_{k} \otimes \id_W), \hspace{2.4cm} (\text{duality}) 
\end{equation}
\begin{equation}
m_W \circ (\id_W \otimes z_{k}) = m_W \circ \tau_{W, W} \circ (\id_W \otimes z_{k}), \hspace{0.8cm}  (\text{centrality condition})
\end{equation}
\begin{equation}
z_k \circ m_C \circ \Delta_C \circ z_k^* = m_W \circ \tau_{W,W} \circ \Delta_W. \hspace{1.8cm}  (\text{genus-one condition}) \end{equation}
\end{definition}
The first equality says that for each $k = 1,2$,  $z_k^*$ is the morphism dual to $z_k$ (implying that $z_k^*$ is a homomorphism of coalgebras in $\mathcal{C}$). If $\mathbf{C} = \textbf{Vect}_k,$ the second equality says that $z_k(C)$ is contained in the center of the algebra $W$, for each $k = 1,2$.

We remark that a \textit{twin Frobenius algebra}, as introduced by the author in~\cite{CC3}, is a set $(C, W, z, z^*)$ which mimics the laws of an enhanced twin Frobenius algebra, with the difference that a twin Frobenius algebra involves only one morphism $z \co C \to W$ with dual $z^* \co W \to C$.
 
\begin{definition} \label{def:homs in eT-Frob}
A \textit{homomorphism of enhanced twin Frobenius algebras}  is a map \[ (C, W, z_1, z_1^*, z_2, z_2^*) \stackrel{(\phi, \psi)}{\longrightarrow} (\overline{C}, \overline{W}_, \overline{z}_1, \overline{z}_1^*, \overline{z}_2, \overline{z}_2^* ) \]
 consisting of a pair $(\phi, \psi)$ of Frobenius algebra homomorphisms $\phi \co C \to \overline{C} $ and $\psi \co W \to \overline{W},$ such that $\overline{z}_{k} \circ \phi = \psi \circ z_{k}$ and $\overline{z}_{k}^*\circ \psi = \phi \circ z_{k}^*,$ for each $k = 1,2.$ 
 \end{definition}
 
 The category of enhanced twin Frobenius algebras in $\textbf{C}$ and their homomorphisms has a symmetric monoidal structure with respect to the tensor product of two  extended twin Frobenius algebras 
\[\mathbf{eT} = (C, W, z_1, z_1^*, z_2, z_2^*)\,\, \text{and} \,\,\mathbf{\overline{eT}} = (\overline{C}, \overline{W}, \overline{z}_1, \overline{z}_1^*, \overline{z}_2, \overline{z}_2^*),\]
which is defined as 
\[\mathbf{eT} \otimes \mathbf{\overline{eT}} = (C \otimes_R \overline{C}, W \otimes_R \overline{W}, z_1 \otimes \overline{z}_1, z_1^* \otimes \overline{z}_1^*, z_2 \otimes \overline{z}_2, z_2^* \otimes \overline{z}_2^*).\] The unit of the monoidal structure is $(\mathbf{1} ,\mathbf{1} ,\id_{\mathbf{1}}, \id_{\mathbf{1}},\id_{\mathbf{1}}, \id_{\mathbf{1}}).$ 

\subsection{Singular cobordisms and the category $\textbf{Sing-2Cob}$}

The category of singular $2$-cobordisms is an extension of the category $\textbf{2Cob}$ of $2$-dimensional cobordisms from oriented manifolds to piecewise oriented (but globally oriented) manifolds. Singular 2-cobordisms are a particular type of foams and may contain singular arcs and/or singular circles where orientations disagree. What we call here $\textbf{Sing-2Cob}$ is in fact a skeleton of the category of singular 2-cobordisms, and is an extension of the category with the same name introduced in~\cite{CC3}, by allowing counterclockwise oriented circles as well. 

We fix a specific \textit{bi-web} (a closed web with exactly two bivalent vertices) and we denote it by 1. We also fix a positively oriented circle and a negatively oriented circle, which we denote by $0_{+}$ and $0_{-}$, respectively.
\[ 1 = \raisebox{-3pt}{\includegraphics[height=0.15in]{singcircle.pdf}} \qquad 0_{+} = \raisebox{-3pt}{\includegraphics[height=0.15in]{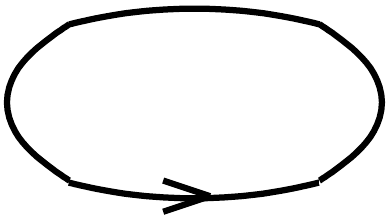}} \qquad 0_{-} = \raisebox{-3pt}{\includegraphics[height=0.15in]{circle.pdf}}  \]

An object of the category $\textbf{Sing-2Cob}$ is a finite sequence  $\textbf{n} = (n_1, n_2,\dots, n_k)$ where $n_j \in \{0_+, 0_-,1 \}$. Thus $\textbf{n}$ is a disjoint union of copies of the fixed bi-web and oriented circles. The length of the sequence, denoted by $\vert \textbf{n} \vert = k $, can be any nonnegative integer, and equals the number of disjoint connected components of the object $\textbf{n}.$ If $\vert \textbf{n} \vert = 0$, then $\textbf{n}$ is the empty 1-manifold.

A morphism $\Sigma \co \textbf{n} \to \textbf{m}$ in $\textbf{Sing-2Cob}$ is a cobordism with source $\textbf{n}$ and target $\textbf{m}$, considered up to equivalences. The boundary of $\Sigma$ is $\partial \Sigma = \overline{\textbf{n}} \cup \textbf{m}$, where $\overline{\textbf{n}}$ is $\textbf{n}$ with opposite orientation. Two singular cobordisms $\Sigma_1$ and $\Sigma_2$ are considered equivalent, and we write $\Sigma_1 \cong \Sigma_2$, if there exists an orientation-preserving diffeomorphism $\Sigma_1 \to \Sigma_2$ which restricts to the identity on the boundary. 

We read morphisms as cobordisms from top to bottom, by convention, and we compose them by placing one on top the other. The concatenation $\textbf{n} \amalg \textbf{m} := (n_1, n_2, \dots, n_{\vert \textbf{n} \vert}, m_1, m_2, \dots, m_{\vert \textbf{m} \vert})$ of sequences together with the free union of singular 2-cobordisms endow the category $\textbf{Sing-2Cob}$ with the structure of a symmetric monoidal category. 

The results in~\cite{CC3} can be easily extended to the new category $\textbf{Sing-2Cob}$, and one acquires that every connected singular 2-cobordism in $\textbf{Sing-2Cob}$ can be obtained by composing the following cobordisms:

\begin{equation} \label{eg:generators_circle} 
\raisebox{-13pt}{\includegraphics[height=0.4in]{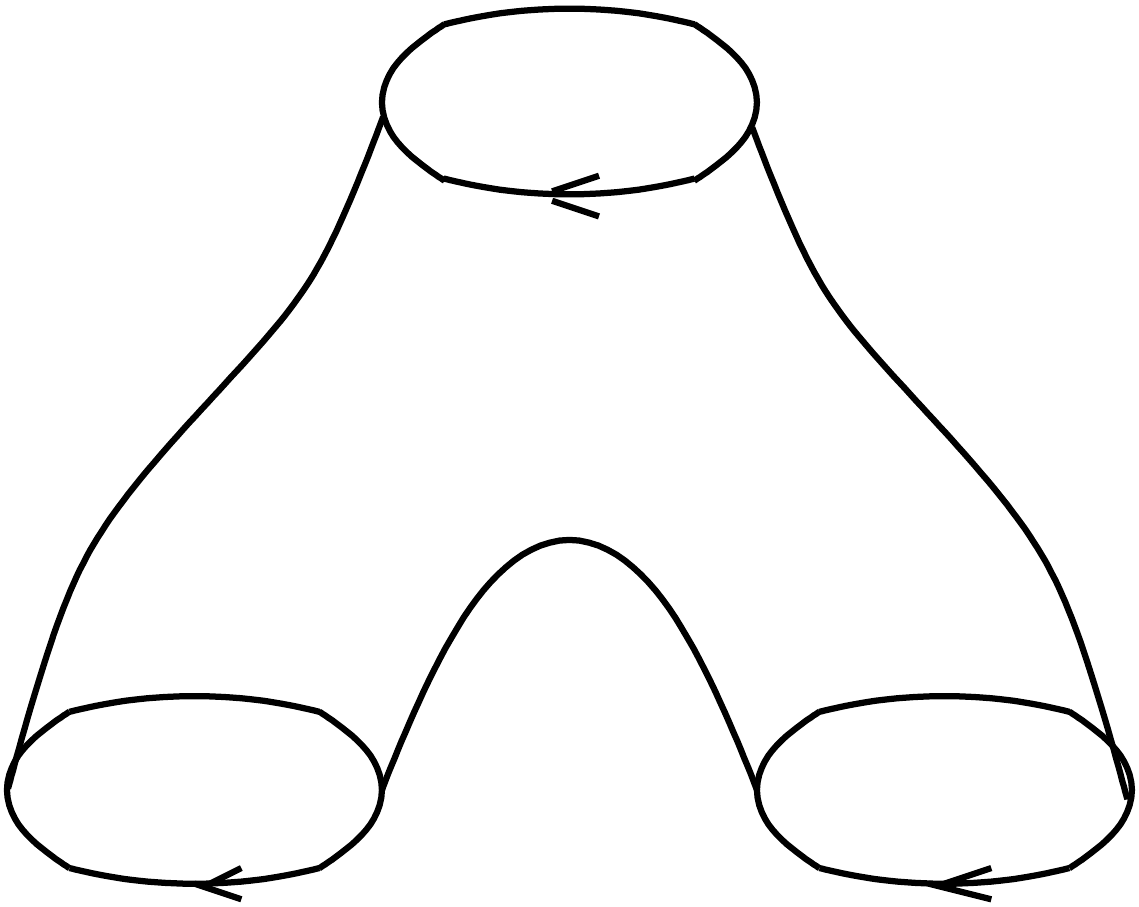}}  \quad  \raisebox{-13pt}{\includegraphics[height=0.4in]{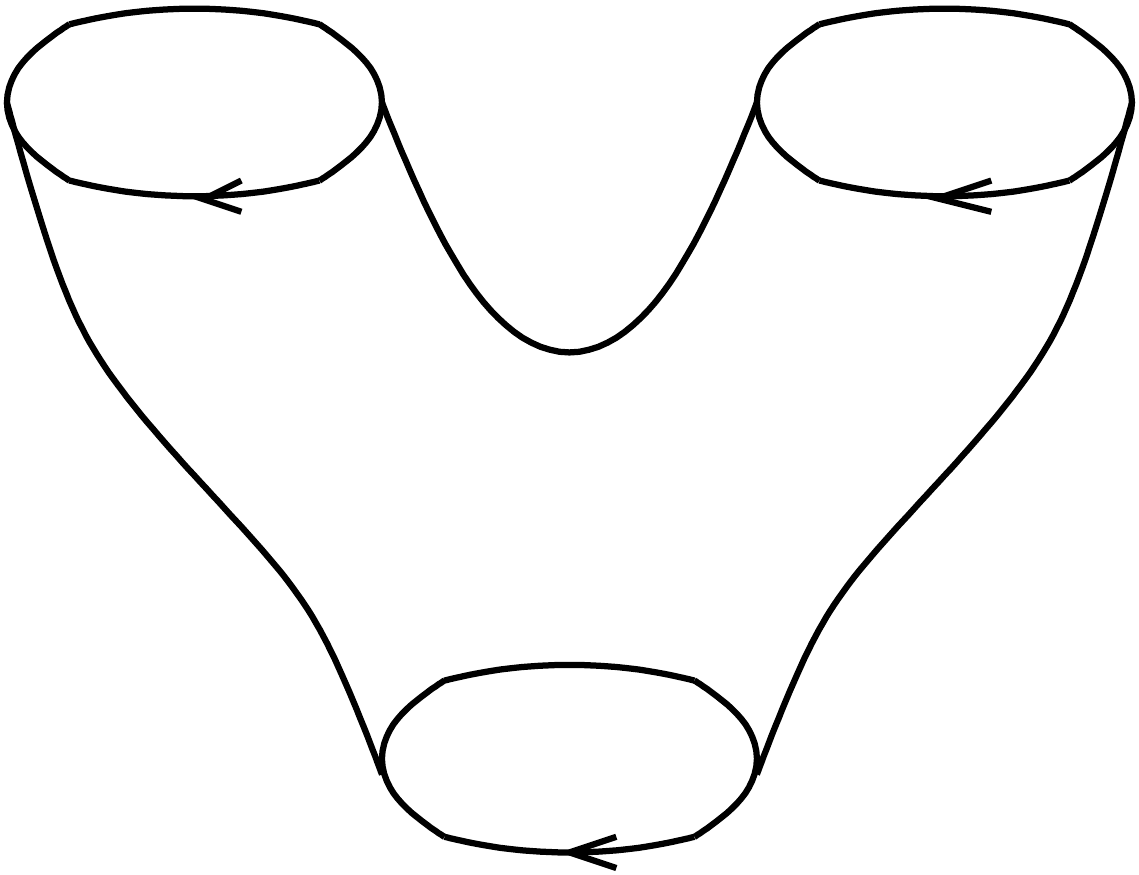}}  \quad  \raisebox{-5pt}{\includegraphics[height=.25in]{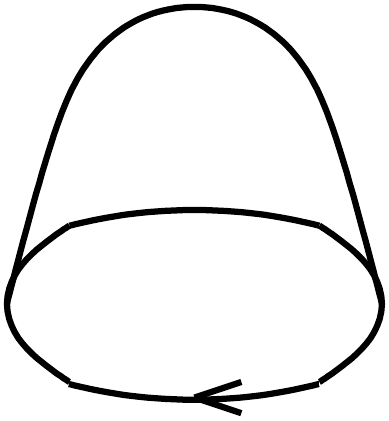}}   \quad \raisebox{-5pt}{\includegraphics[height=0.25in]{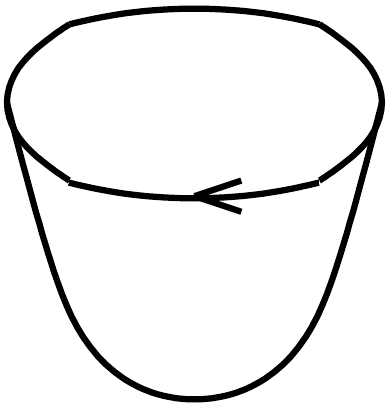}} 
\end{equation}
\begin{equation}  \label{eg:generators_web}
 \raisebox{-13pt}{\includegraphics[height=0.4in]{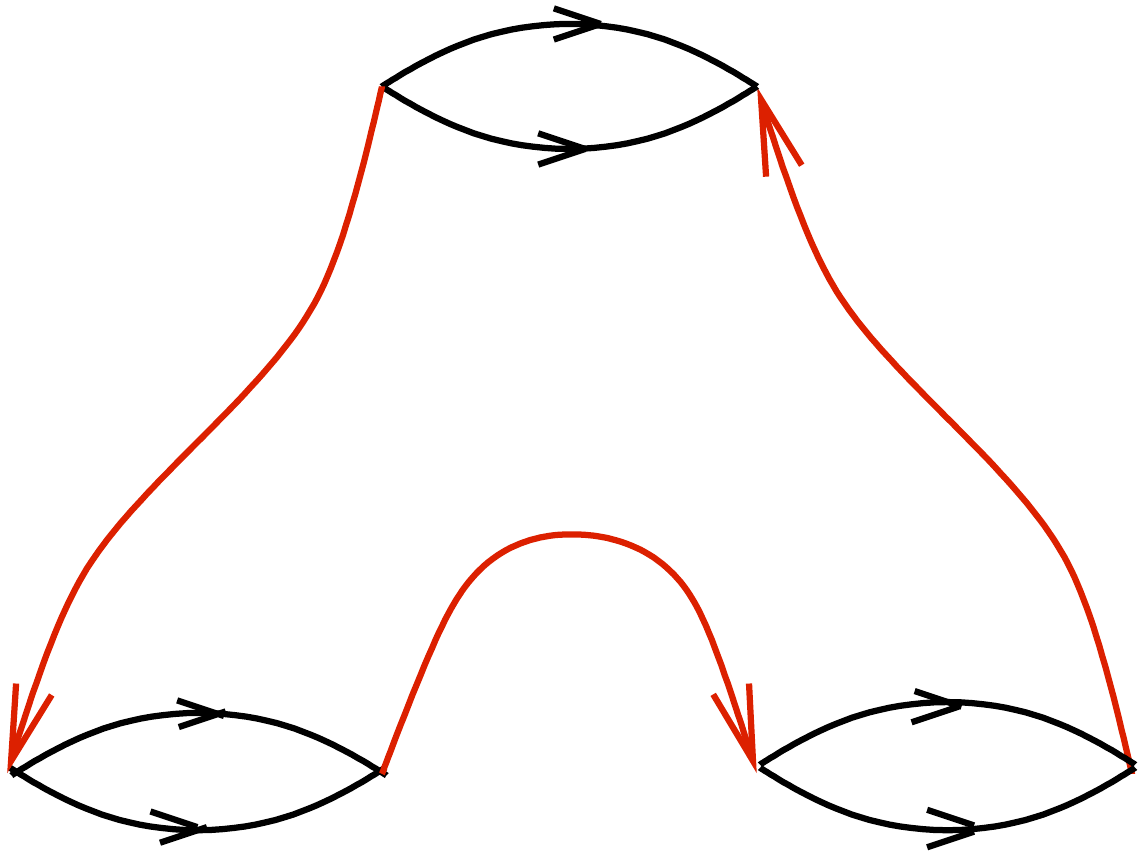}}  \quad  \raisebox{-13pt}{\includegraphics[height=0.4in]{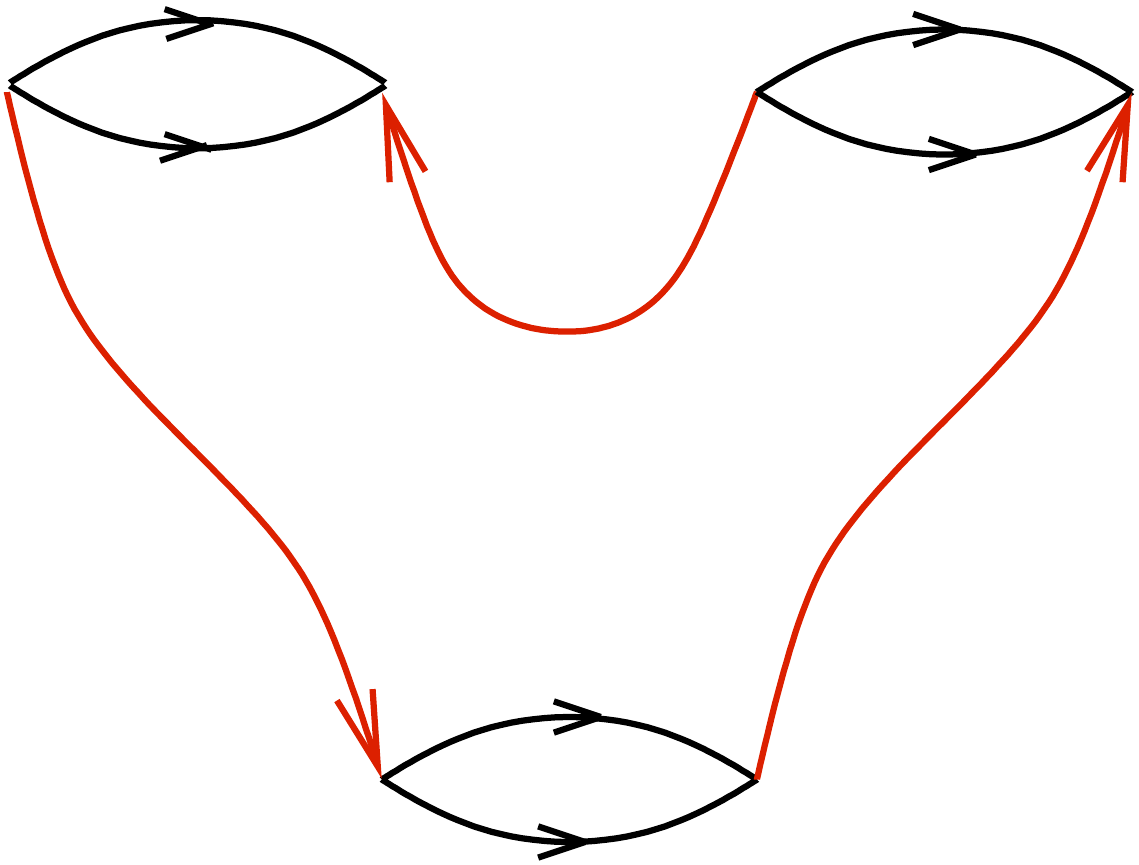}}  \quad  \raisebox{-5pt}{\includegraphics[height=.25in]{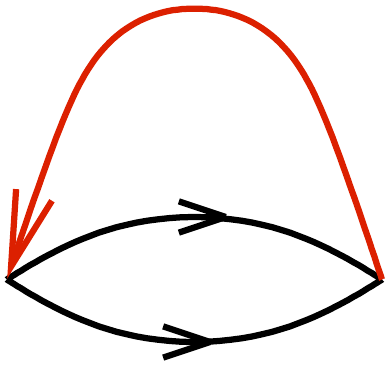}}  \quad \raisebox{-5pt}{\includegraphics[height=0.25in]{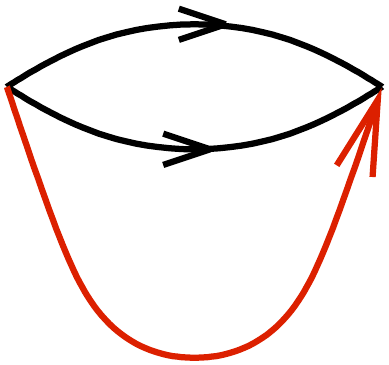}} 
 \end{equation} 
 \begin{equation} \label{eg:generators_web_circle_one}
 \raisebox{-18pt}{\includegraphics[height=0.4in]{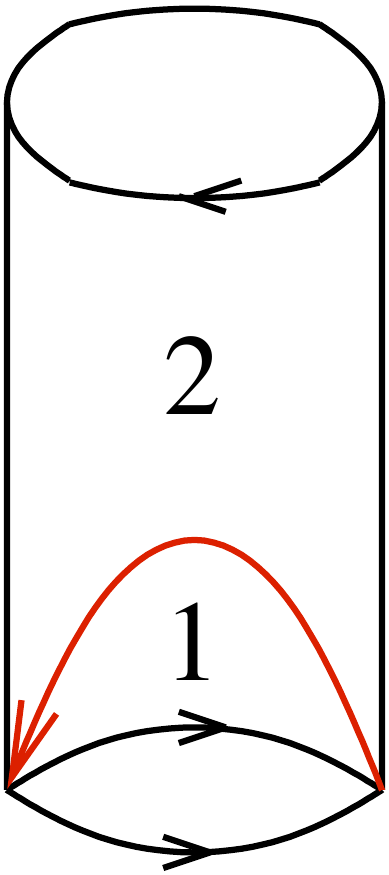}}   \hspace{0.5cm} \raisebox{-18pt}{\includegraphics[height=0.4in]{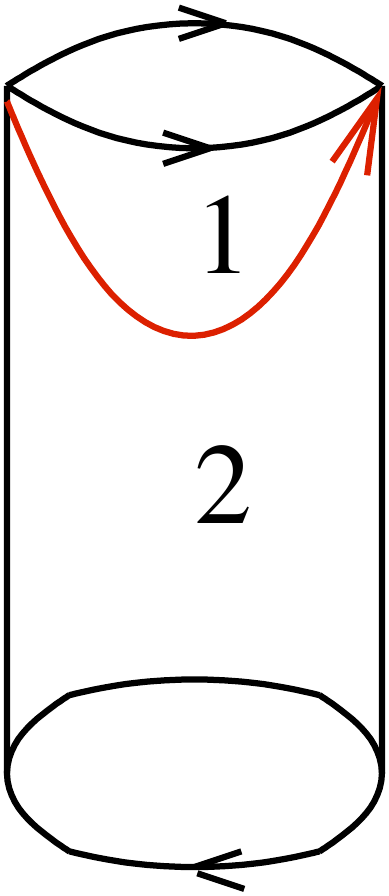}}  \hspace{0.5cm} 
 \raisebox{-18pt}{\includegraphics[height=0.4in]{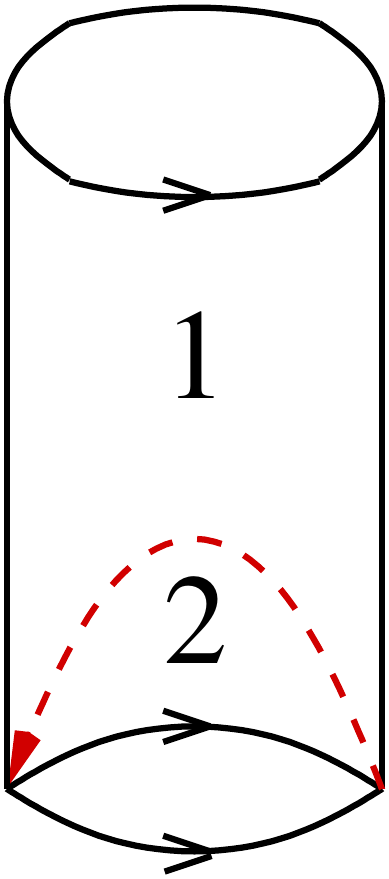}} \hspace{0.5cm} \raisebox{-18pt}{\includegraphics[height=0.4in]{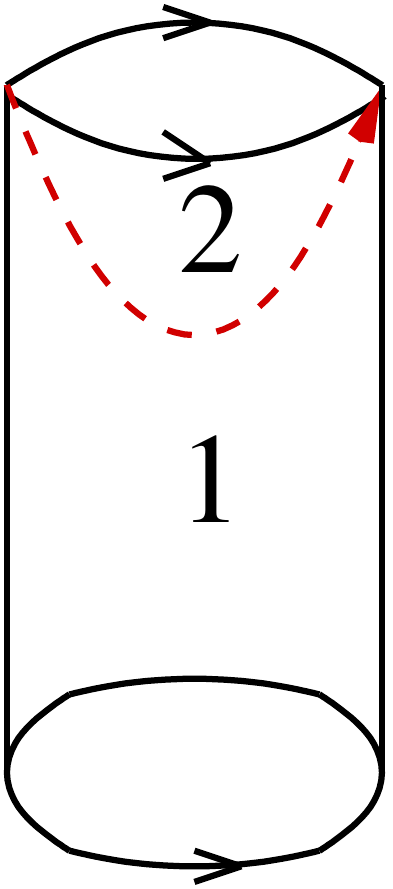}}
 \end{equation} 
together with the identity cobordisms (that is, cylinders over the bi-web, and over the positively oriented and negatively oriented circles) and the analogue of those depicted in~\eqref{eg:generators_circle} but with opposite orientation for the boundary circles.
 
 We use a dashed red curve to represent a singular arc that lies on the back of a cobordisms; otherwise we use continuous red curves.
 
 We draw the composite cobordisms \,$\raisebox{-13pt}{\includegraphics[height=0.4in]{cozipper_second.pdf}} \circ  \raisebox{-13pt}{\includegraphics[height=0.4in]{zipper.pdf}}$\, and \, $\raisebox{-13pt}{\includegraphics[height=0.4in]{cozipper.pdf}} \circ  \raisebox{-13pt}{\includegraphics[height=0.4in]{zipper_second.pdf}}$\, as:
\begin{equation} \label{eq:morphisms f and g}
\raisebox{-13pt}{\includegraphics[height=0.4in]{cozipper_second.pdf}} \circ  \raisebox{-13pt}{\includegraphics[height=0.4in]{zipper.pdf}} = \raisebox{-13pt}{\includegraphics[height=0.4in]{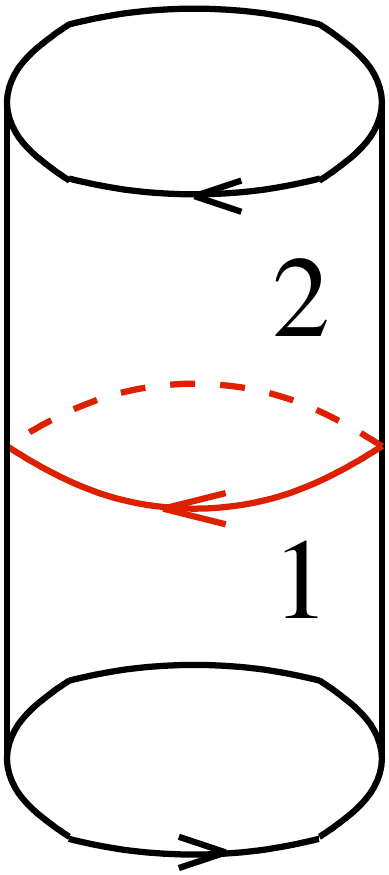}}  \hspace{1cm}  \raisebox{-13pt}{\includegraphics[height=0.4in]{cozipper.pdf}} \circ  \raisebox{-13pt}{\includegraphics[height=0.4in]{zipper_second.pdf}}  = \raisebox{-13pt}{\includegraphics[height=0.4in]{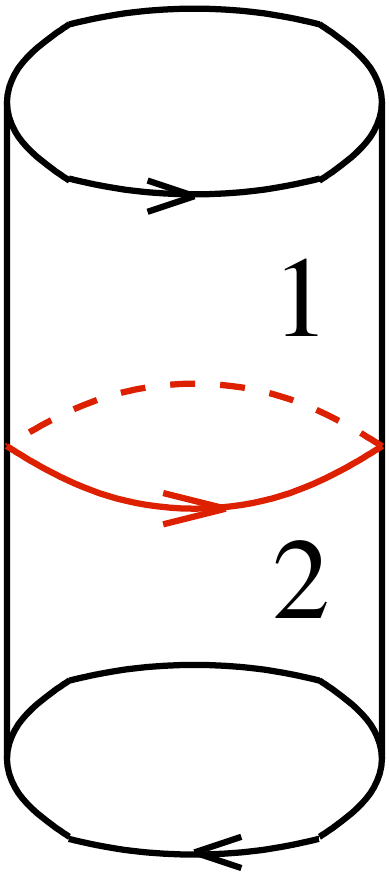}} \end{equation}

There is a finite set of relations among the morphisms that mimic the equations defining an enhanced twin Frobenius algebra. Specifically, the following diffeomorphisms hold in $\textbf{Sing-2Cob}$ (these can be verified as in the proof of~\cite[Proposition 3]{CC3}).

\begin{proposition}\label{prop:relations}
The following relations hold in the symmetric monoidal category $\textbf{Sing-2Cob}:$
\begin{enumerate}
\item The oriented circles $\raisebox{-3pt}{\includegraphics[height=0.15in]{circle.pdf}}$ and $\raisebox{-3pt}{\includegraphics[height=0.15in]{loop.pdf}}$ form commutative Frobenius algebra objects (we give below the relations for the negatively oriented circle, but the reader should have in mind that there are similar relations for the positively oriented circle):

\begin{equation*}
\raisebox{-13pt}{\includegraphics[height=0.6in]{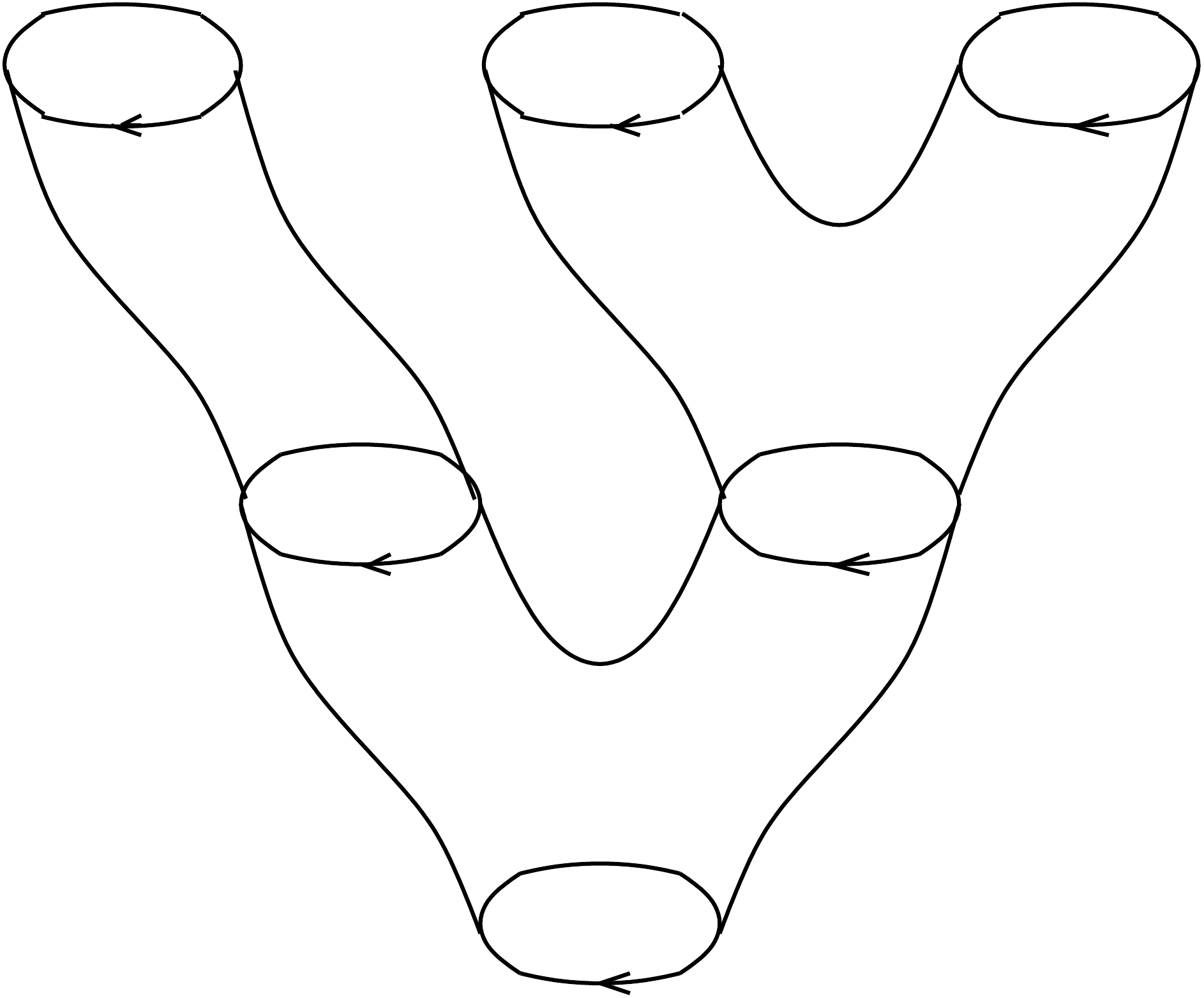}}\quad \cong \quad  \raisebox{-13pt}{\includegraphics[height=0.6in]{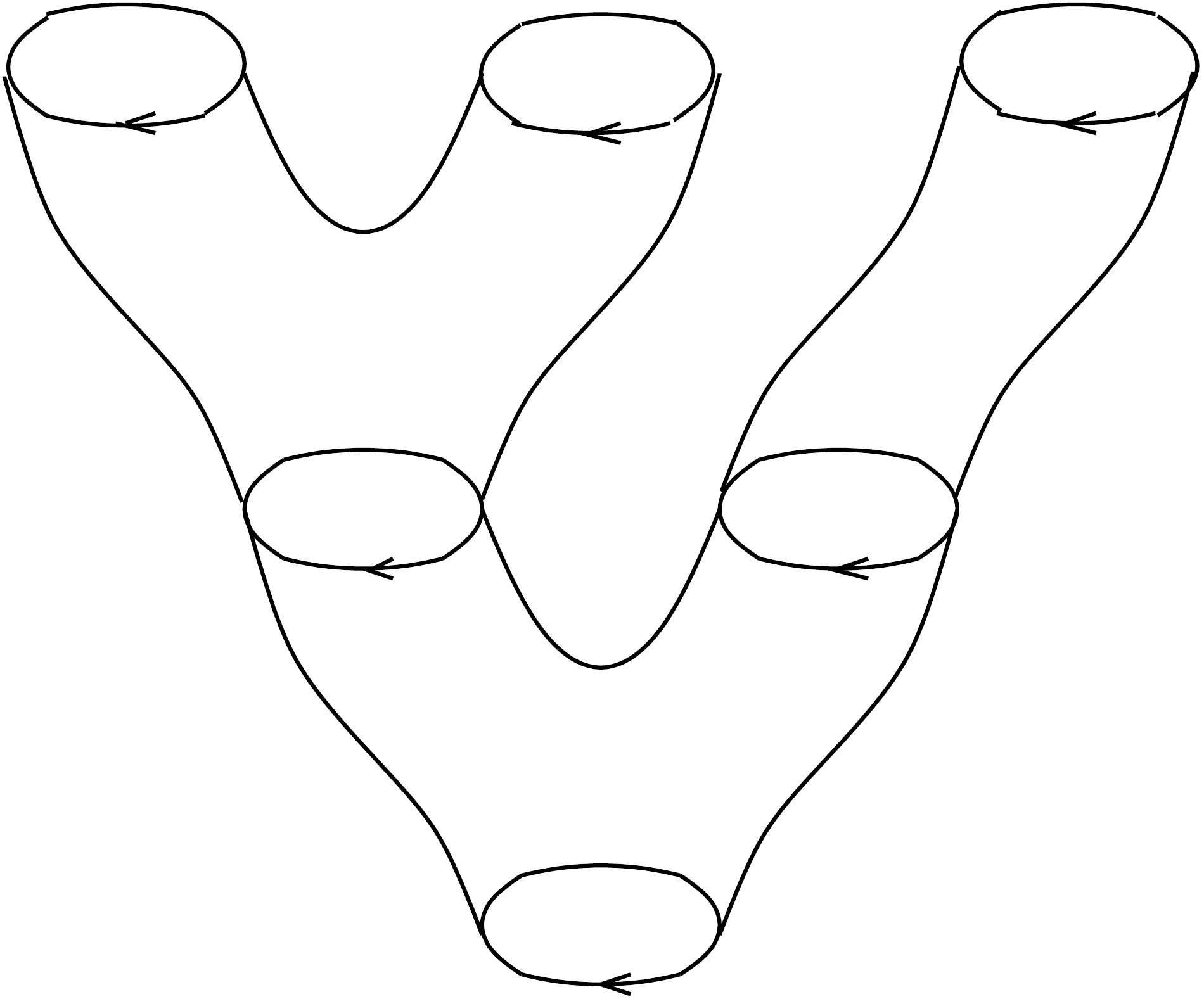}} \hspace{1cm} \raisebox{-13pt}{\includegraphics[height=0.6in]{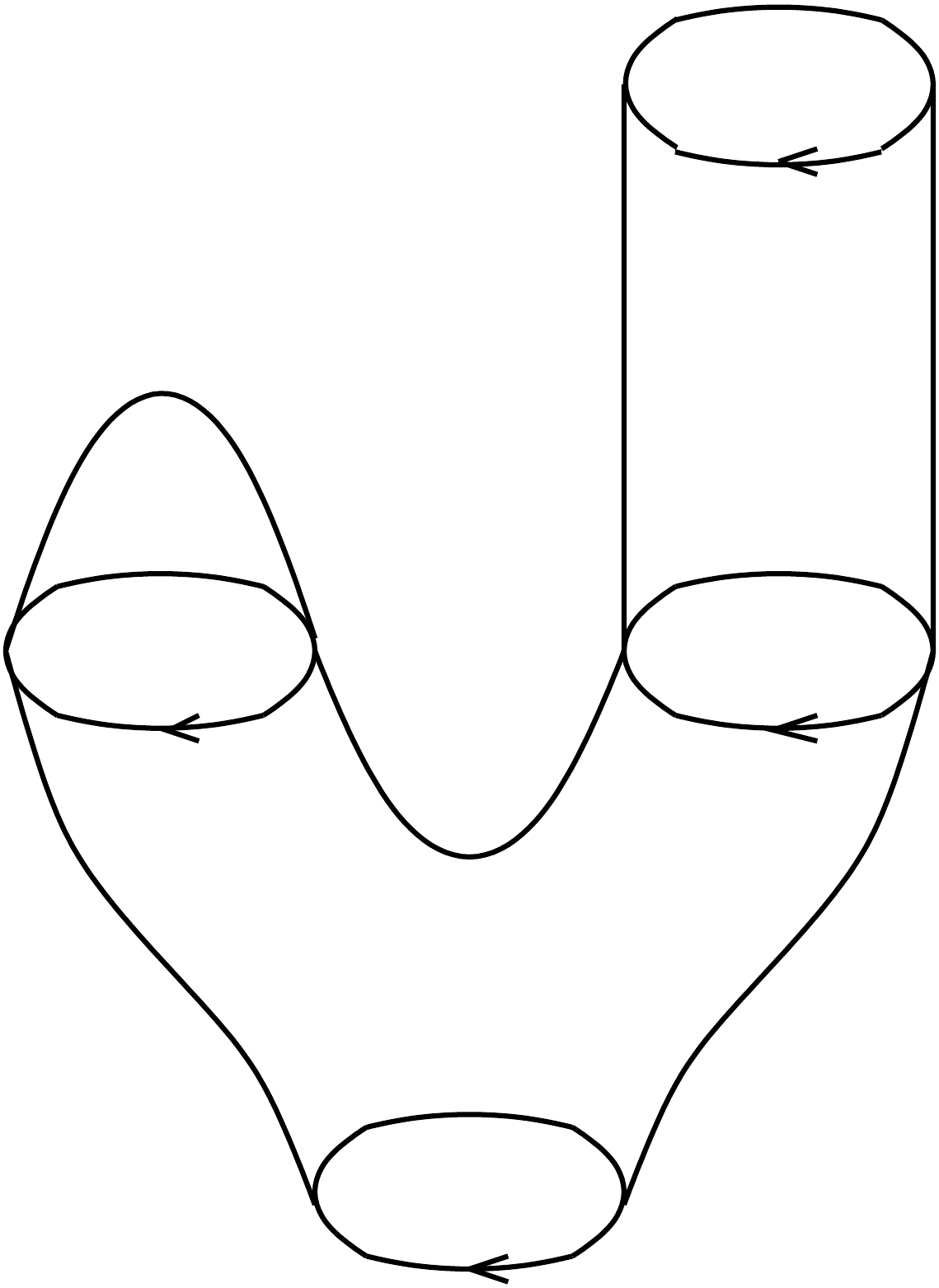}}\quad \cong \quad  \raisebox{-13pt}{\includegraphics[height=0.6in]{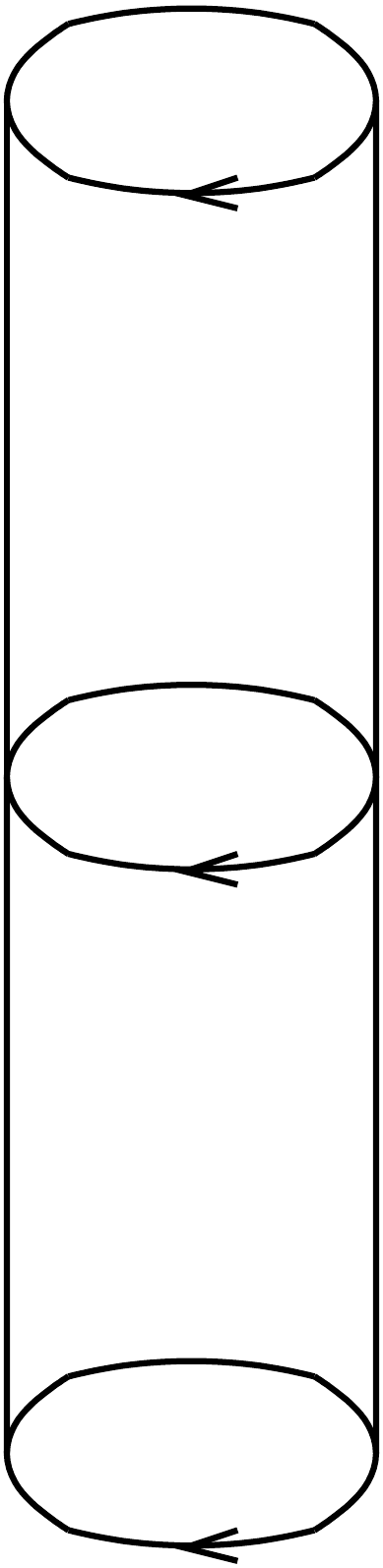}}\quad \cong \quad \raisebox{-13pt}{\includegraphics[height=0.6in]{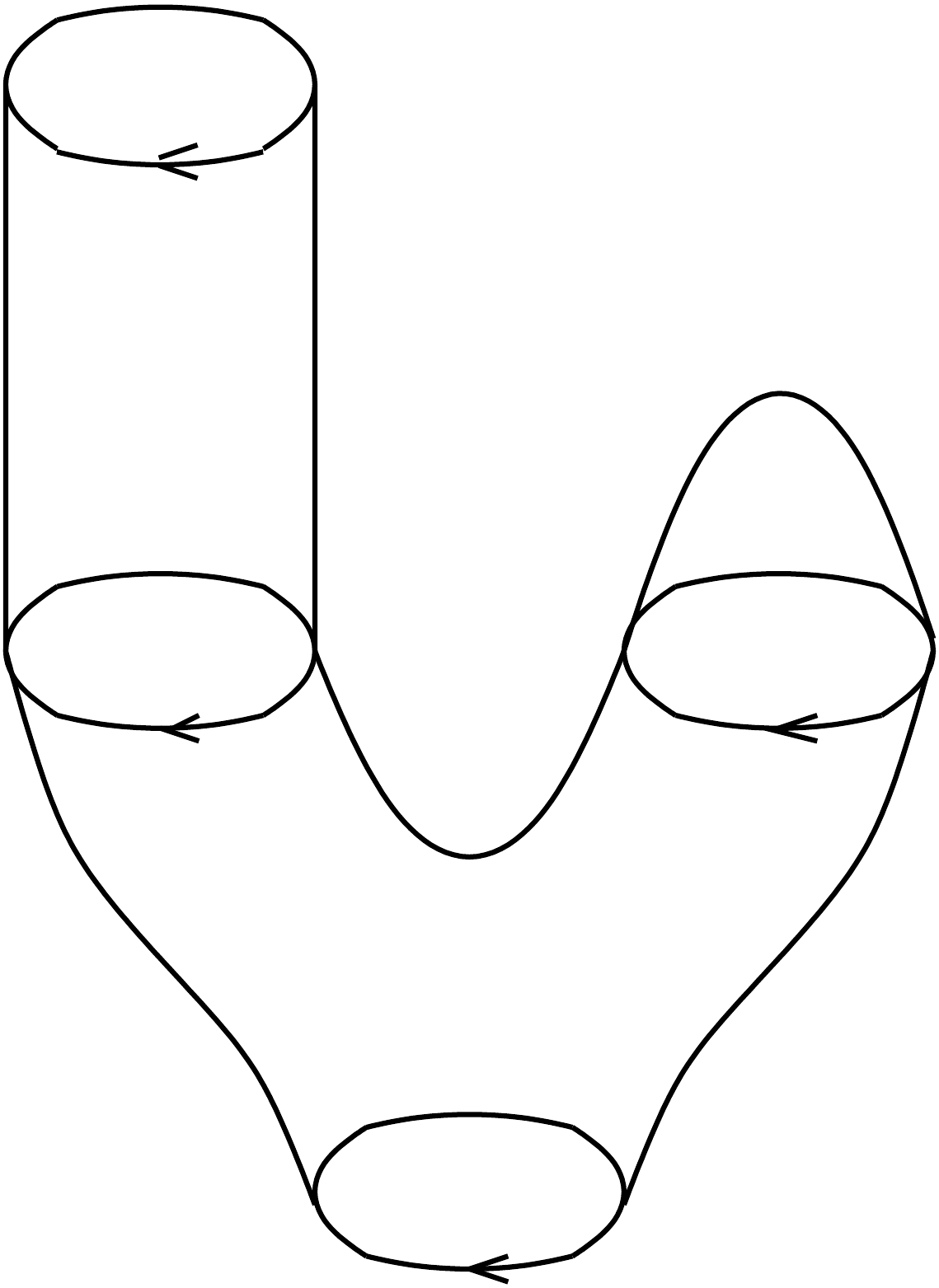}}\label{eq:circle_frob1}
\end{equation*}
\begin{equation*}
\raisebox{-13pt}{\includegraphics[height=0.6in]{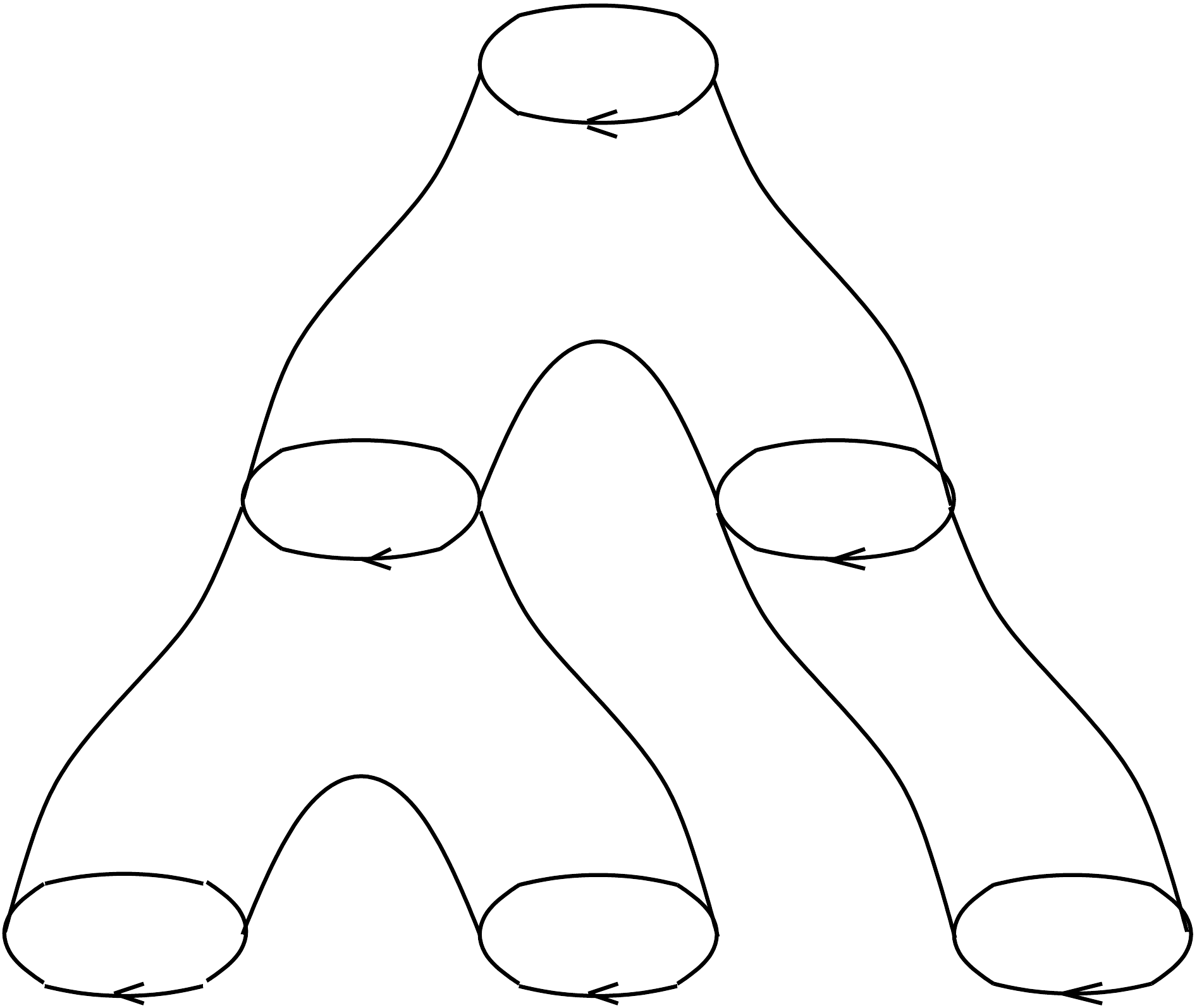}} \quad \cong \quad \raisebox{-13pt}{\includegraphics[height=0.6in]{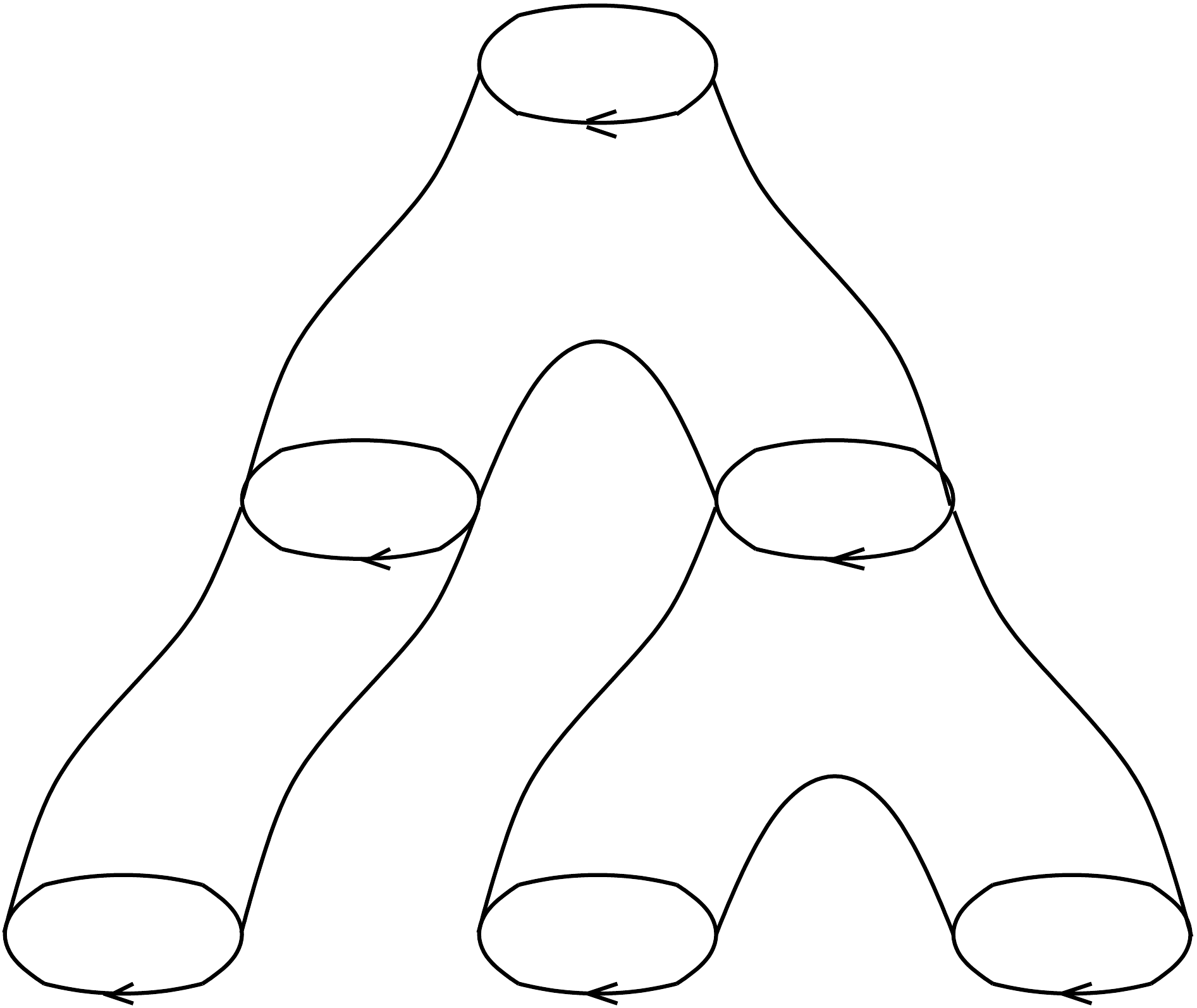}}  \hspace{1cm} \raisebox{-13pt}{\includegraphics[height=0.6in]{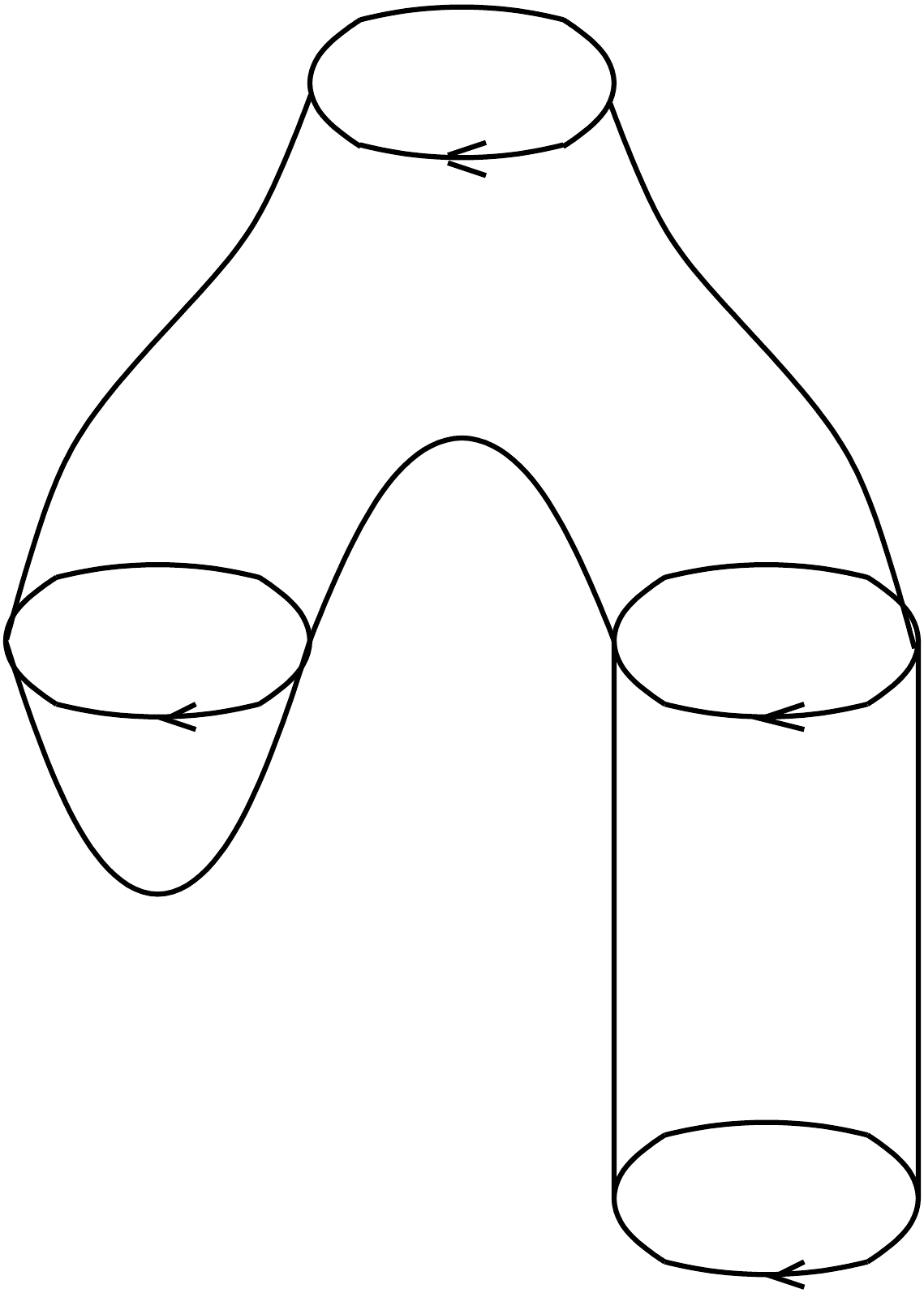}}\quad \cong \quad  \raisebox{-13pt}{\includegraphics[height=0.6in]{circle_frob4.pdf}}\quad \cong \quad \raisebox{-13pt}{\includegraphics[height=0.6in]{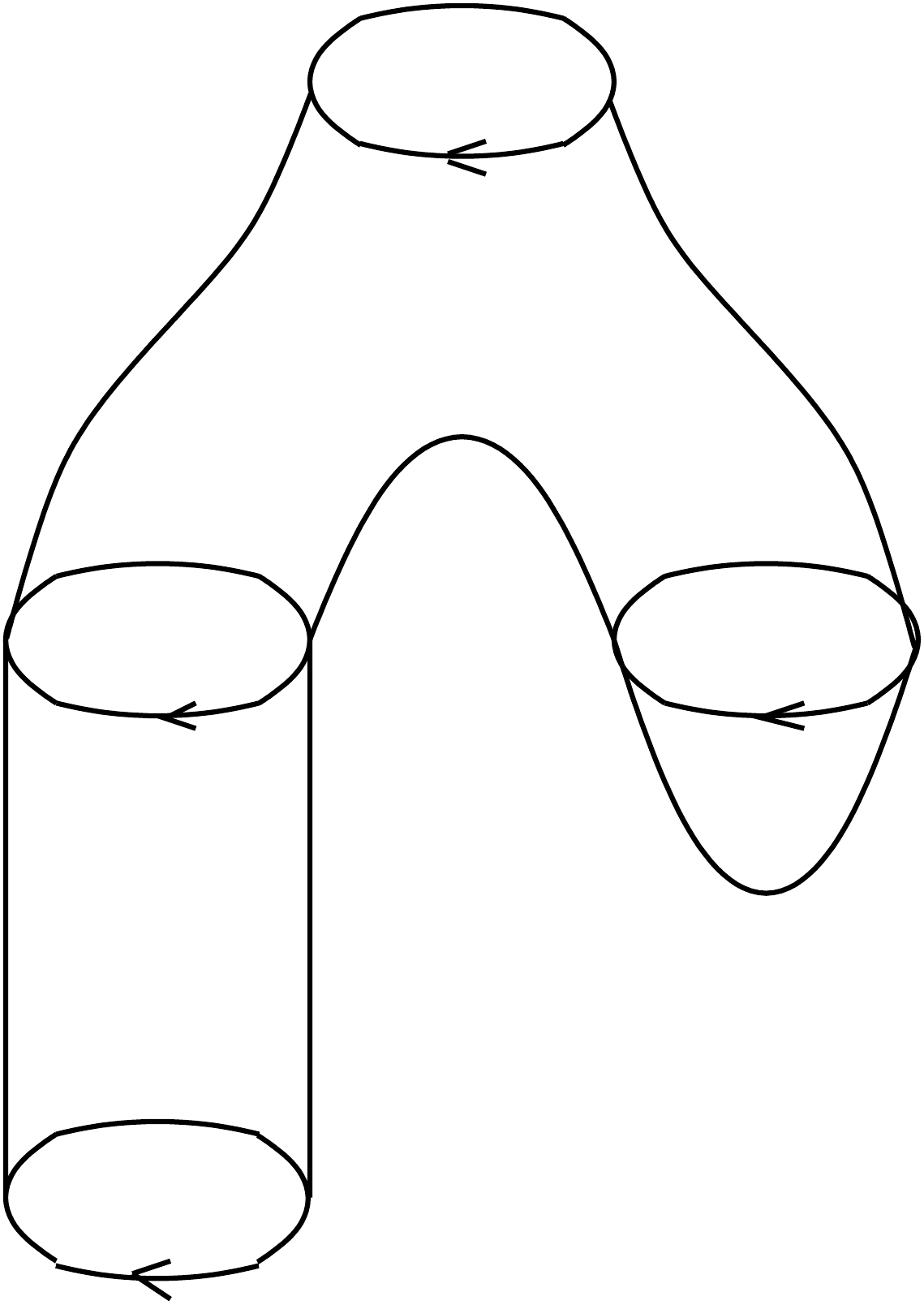}}\label{eq:circle_frob2}
\end{equation*}
\begin{equation*}
\raisebox{-13pt}{\includegraphics[height=0.6in]{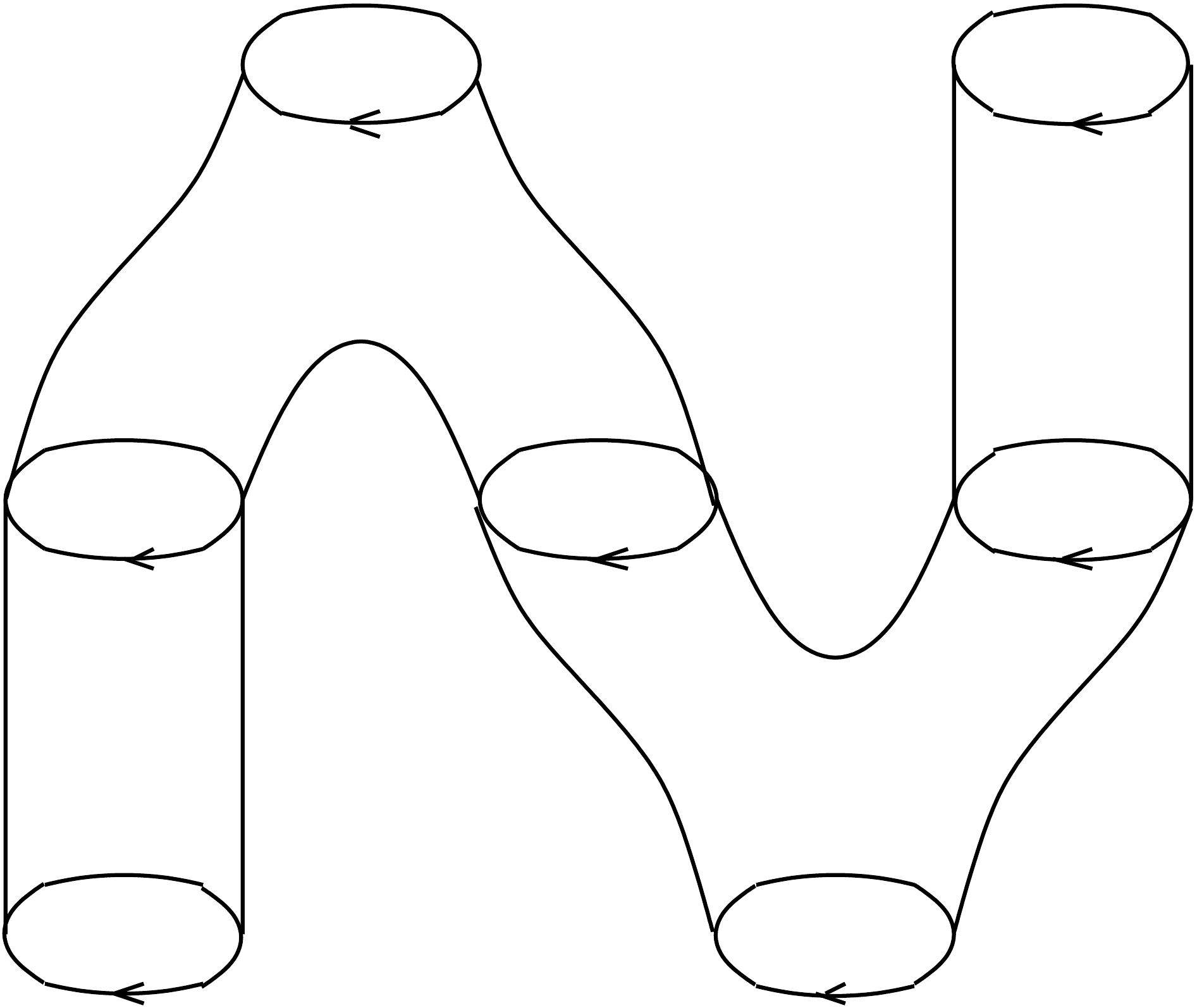}} \quad \cong \quad \raisebox{-13pt}{\includegraphics[height=0.6in]{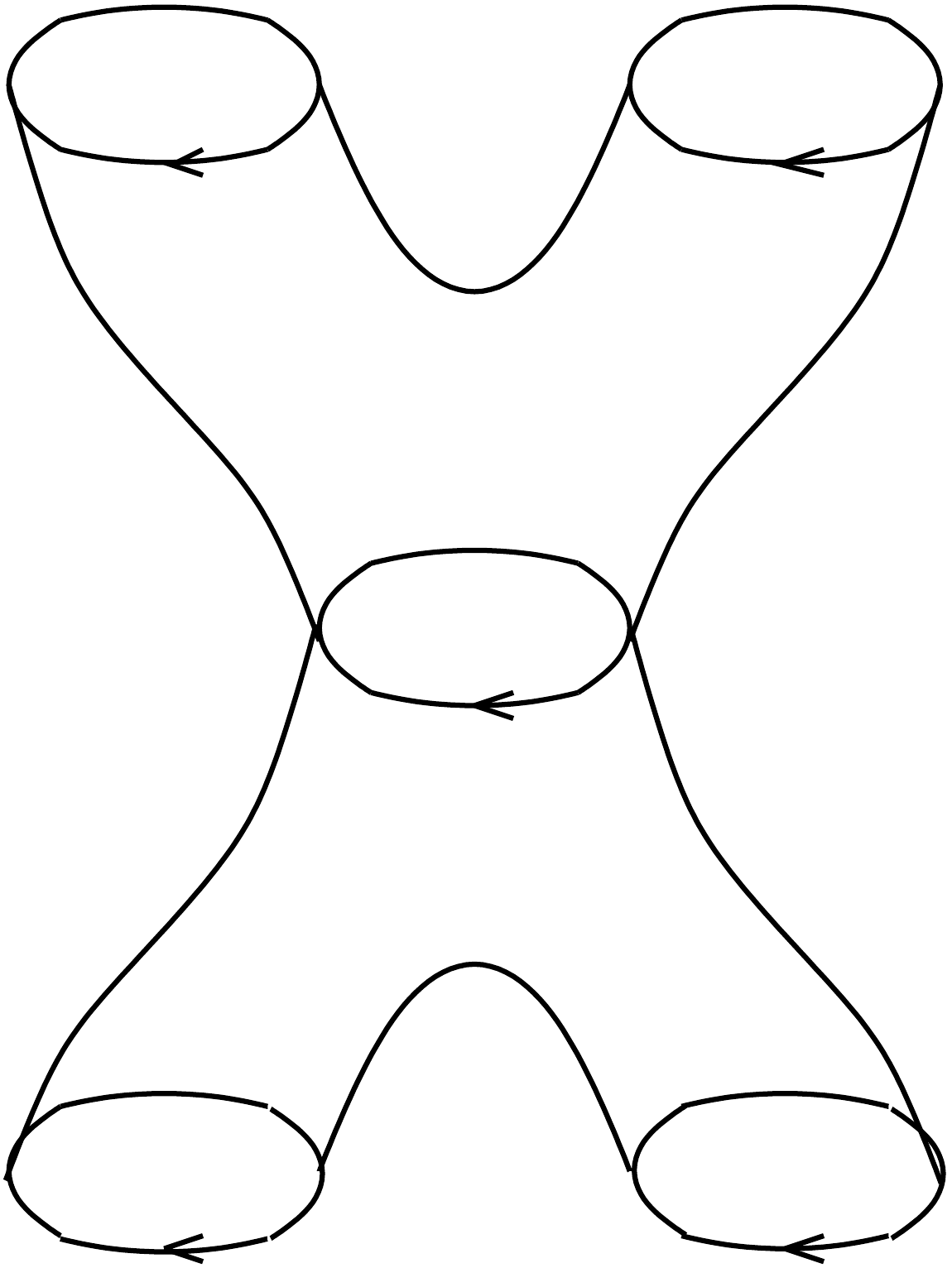}} \quad \cong \quad \raisebox{-13pt}{\includegraphics[height=0.6in]{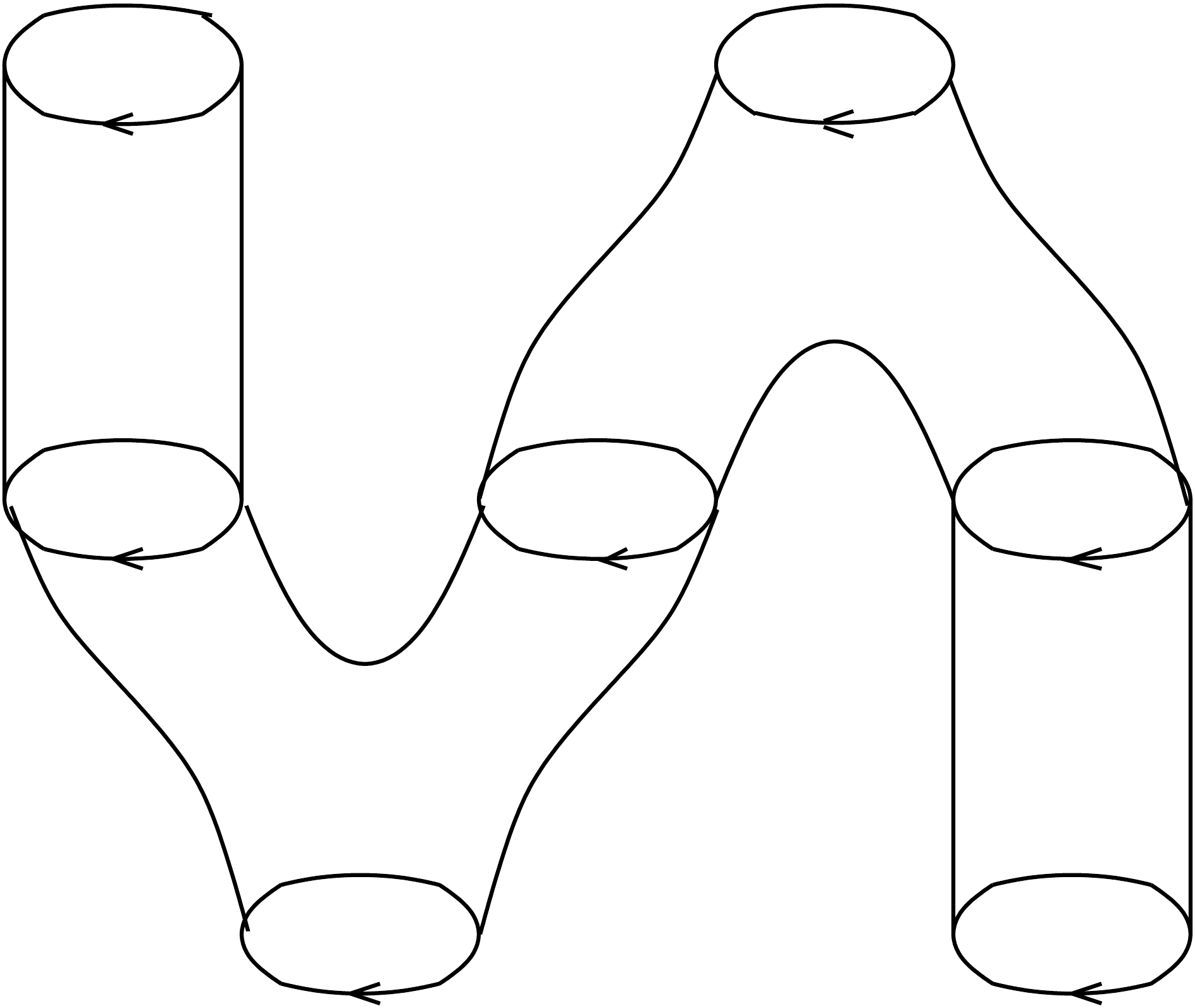}} \hspace{1cm} \raisebox{-13pt}{\includegraphics[height=0.6in]{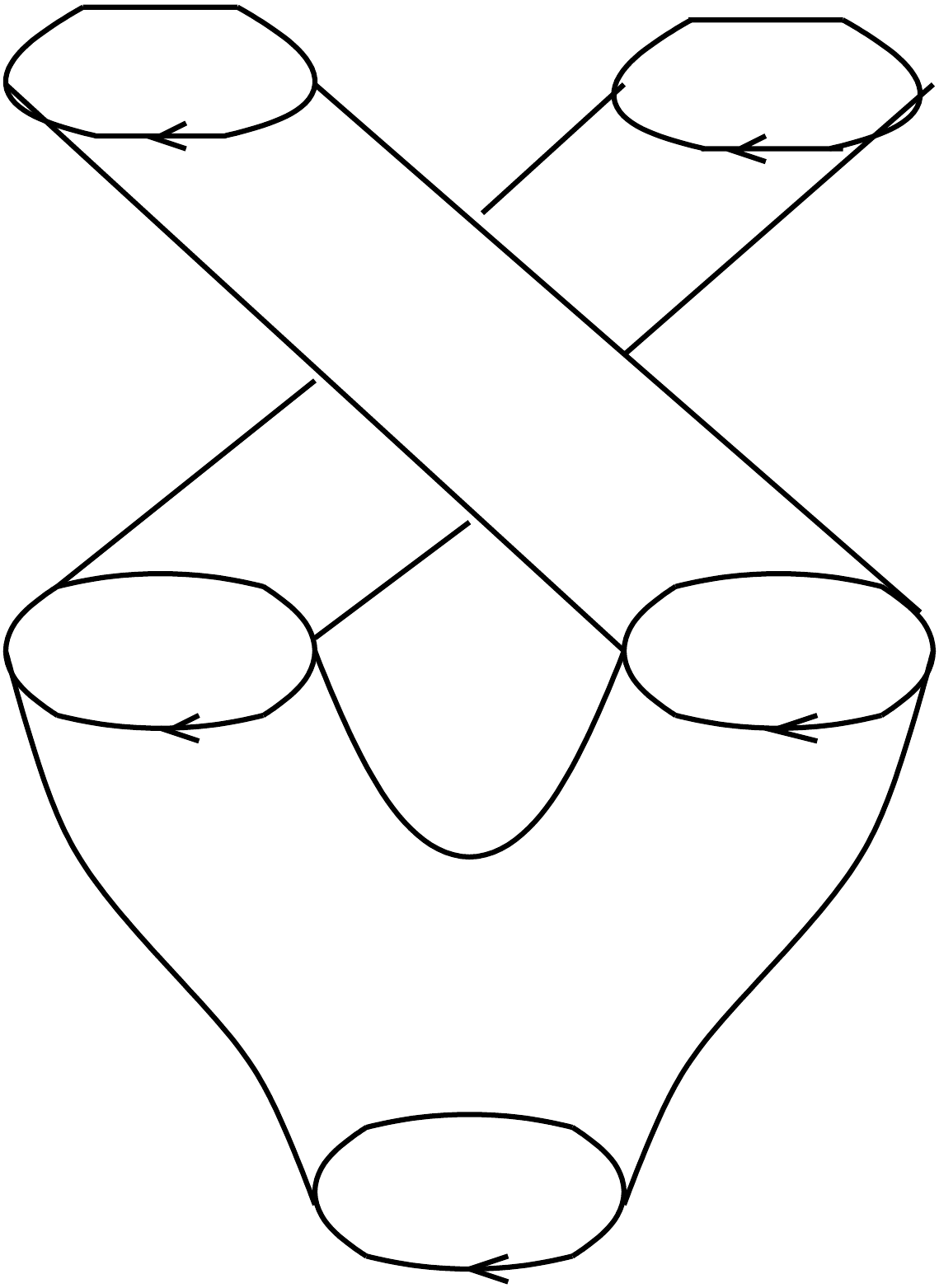}} \quad \cong \quad \raisebox{-13pt}{\includegraphics[height=0.6in]{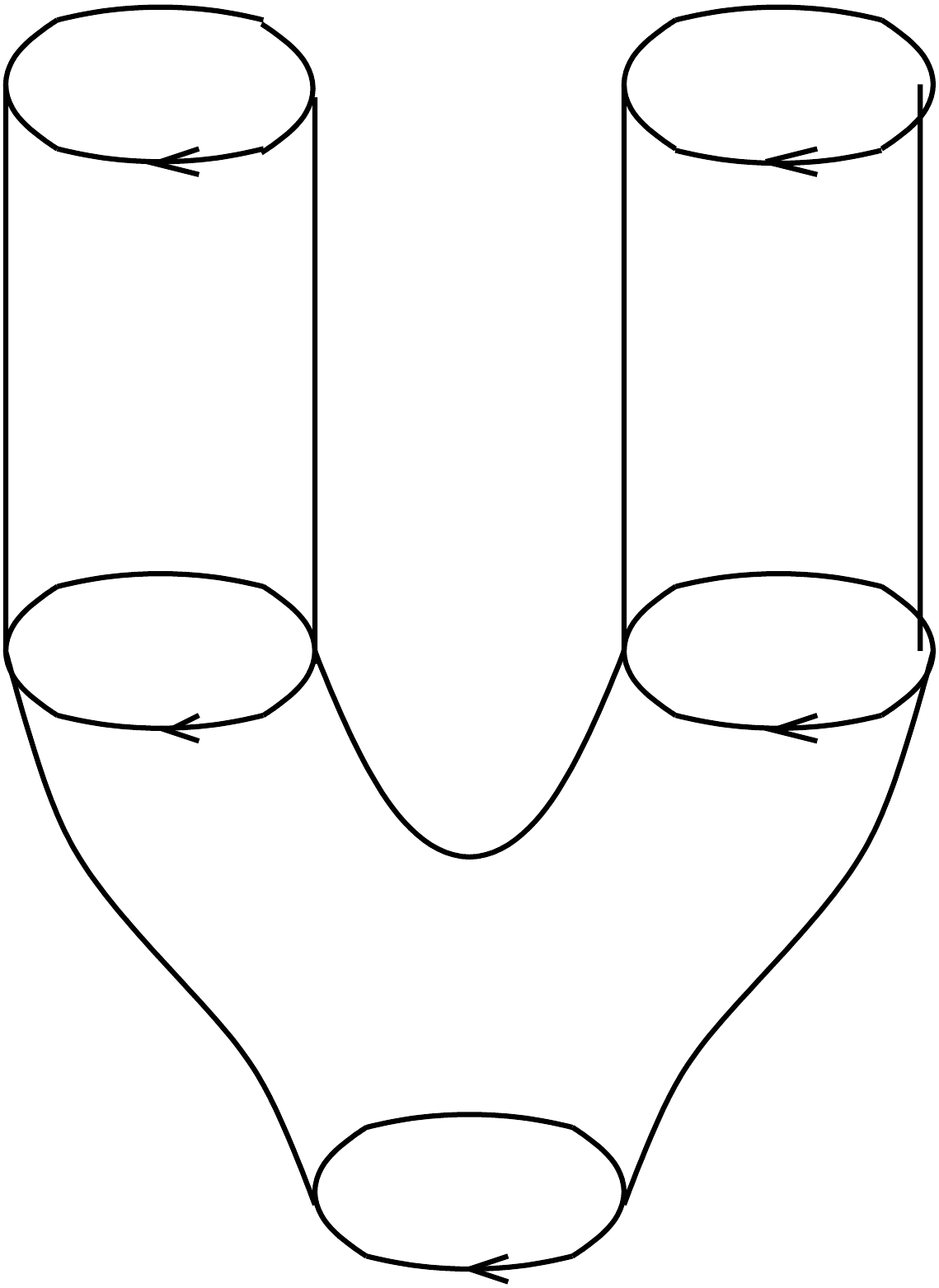}}\label{eq:circle_frob3}
\end{equation*}

\item The bi-web  $\raisebox{-3pt}{\includegraphics[height=0.15in]{singcircle.pdf}}$ forms a symmetric Frobenius algebra object:

\begin{equation*}
\raisebox{-13pt}{\includegraphics[height=0.6in]{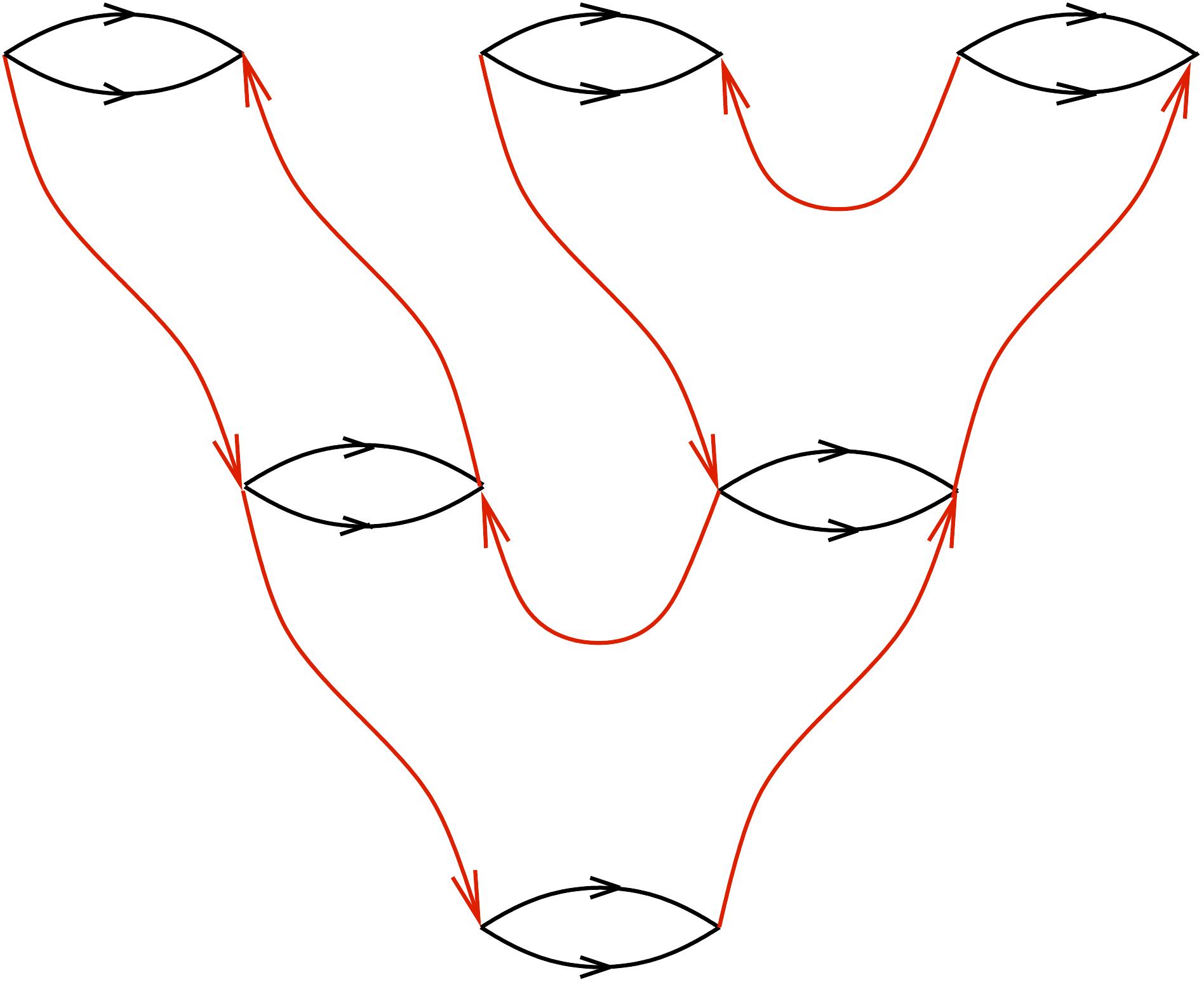}}\quad \cong \quad  \raisebox{-13pt}{\includegraphics[height=0.6in]{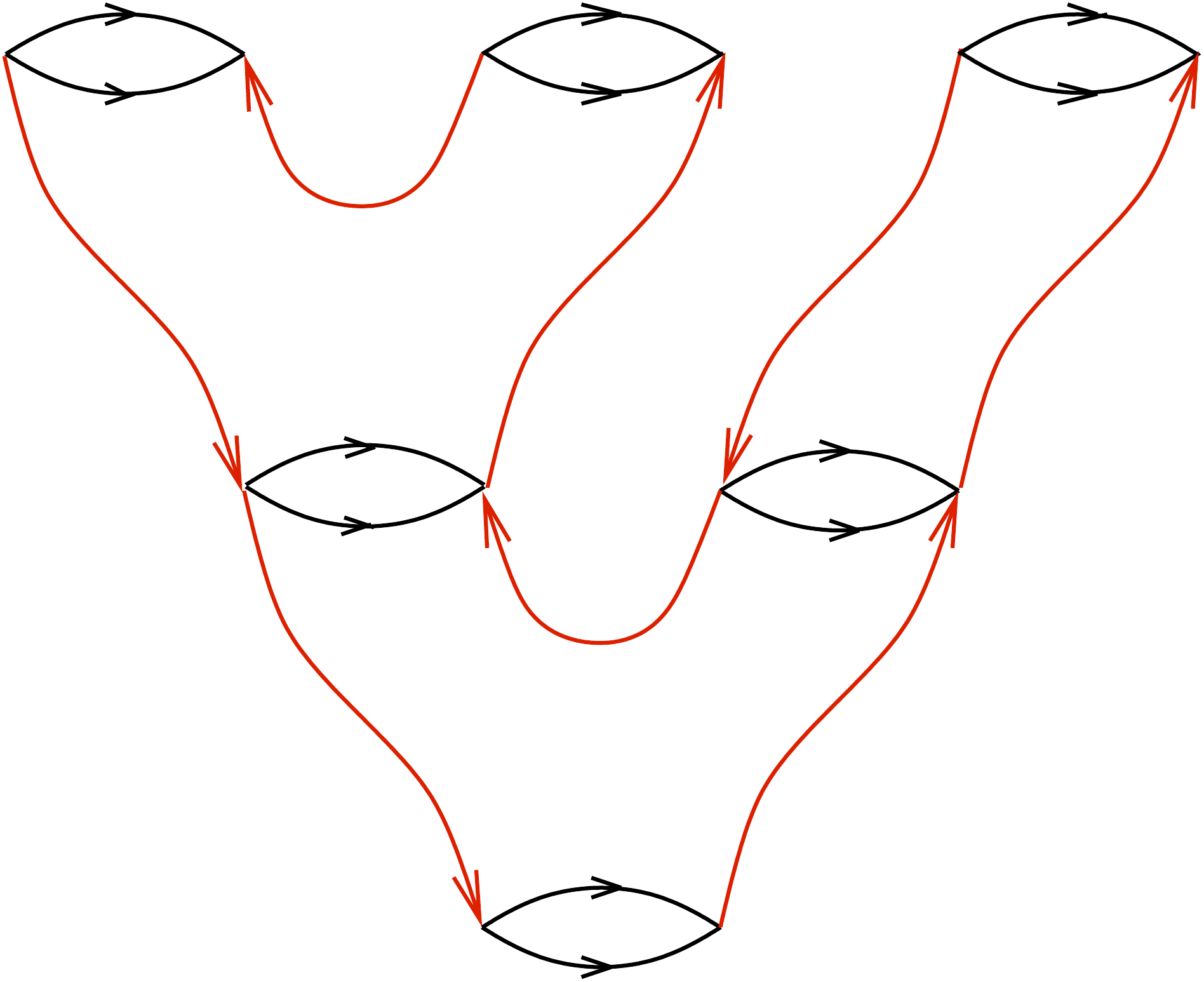}} \hspace{1cm} \raisebox{-13pt}{\includegraphics[height=0.6in]{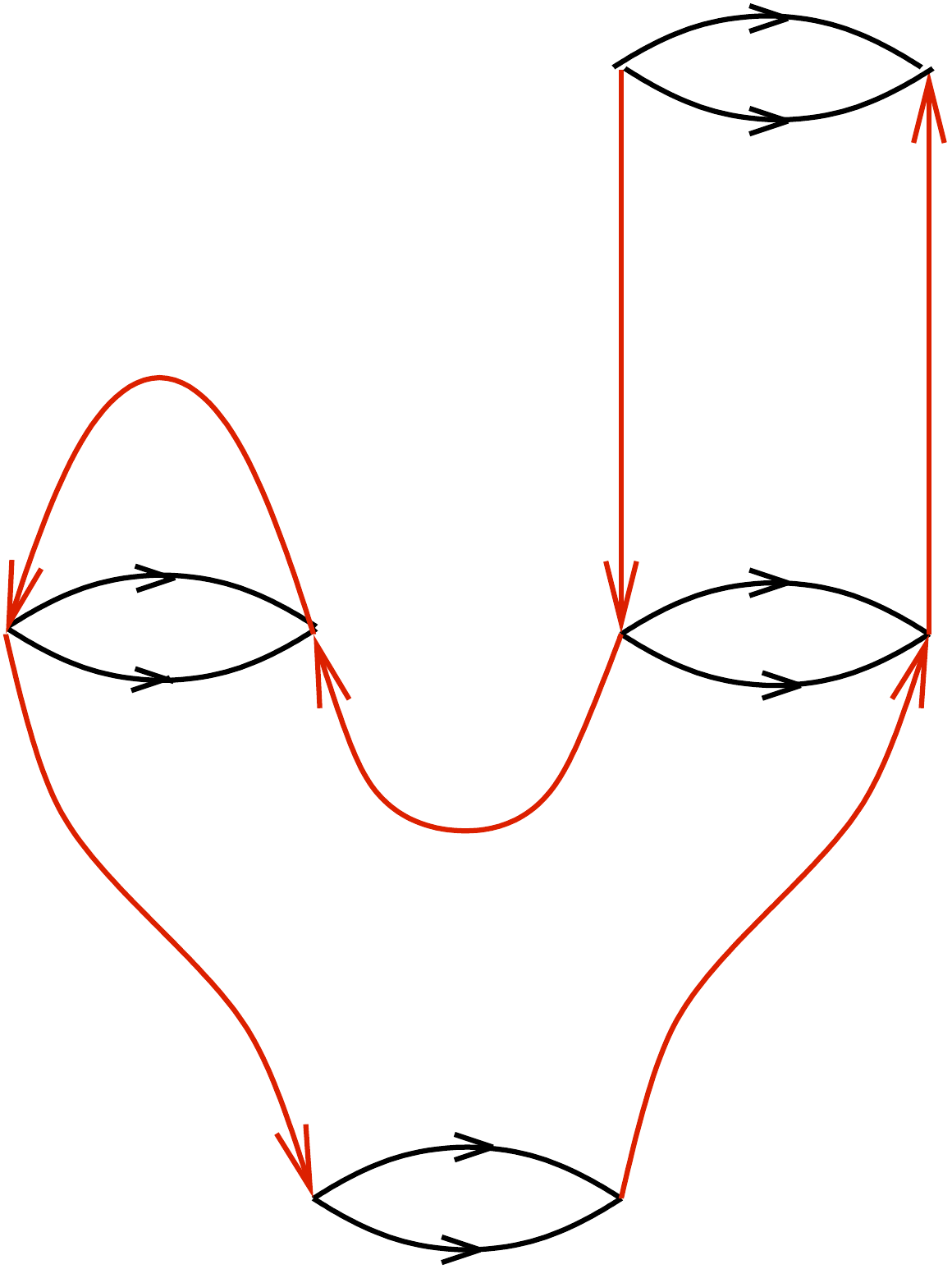}}\quad \cong \quad  \raisebox{-13pt}{\includegraphics[height=0.6in]{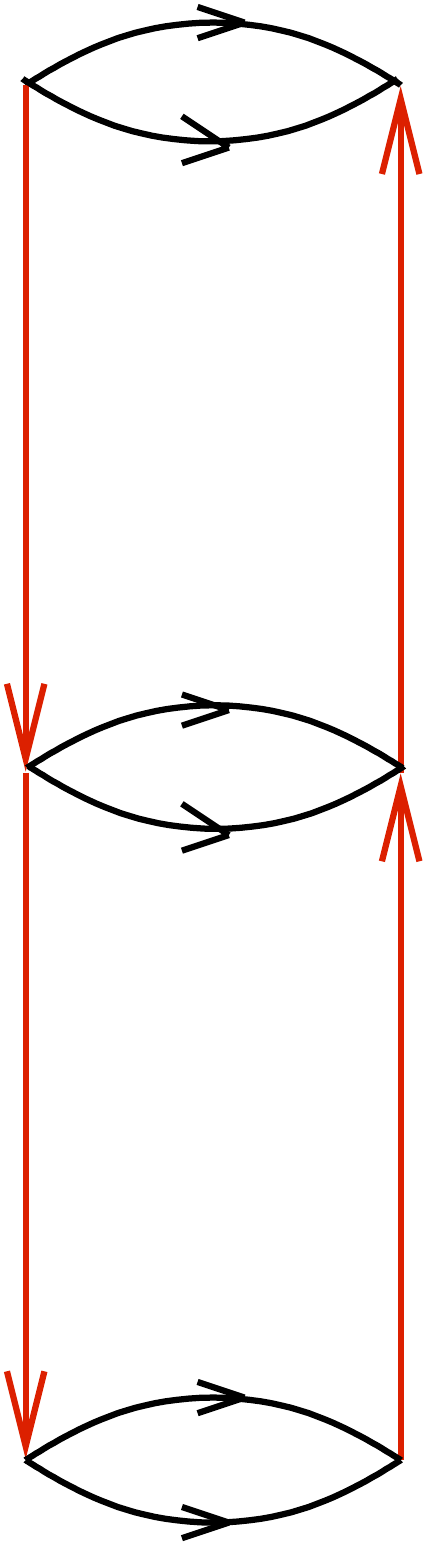}}\quad \cong \quad \raisebox{-13pt}{\includegraphics[height=0.6in]{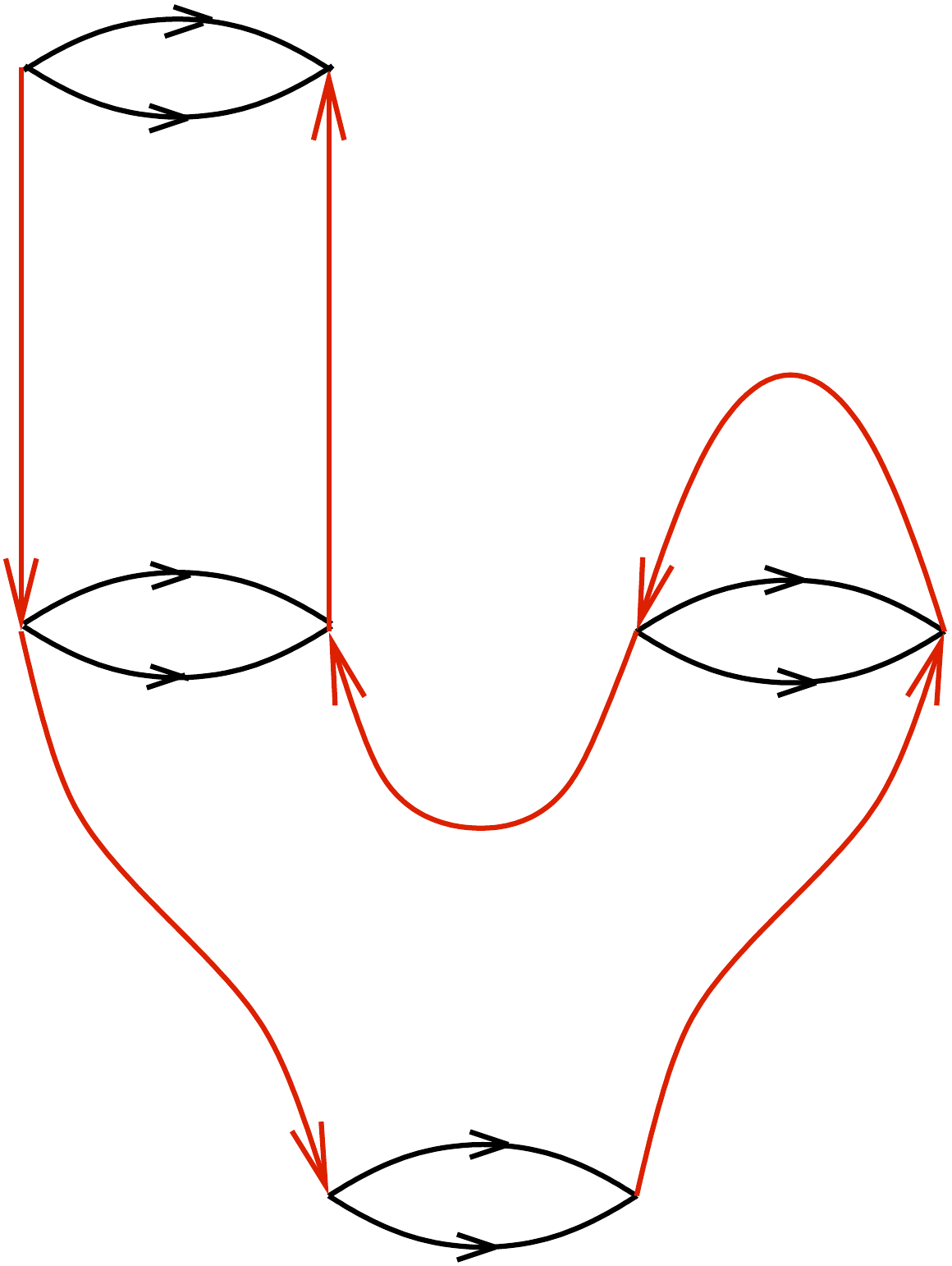}} \label{eq:web_frob1}
\end{equation*}
\begin{equation*}
\raisebox{-13pt}{\includegraphics[height=0.6in]{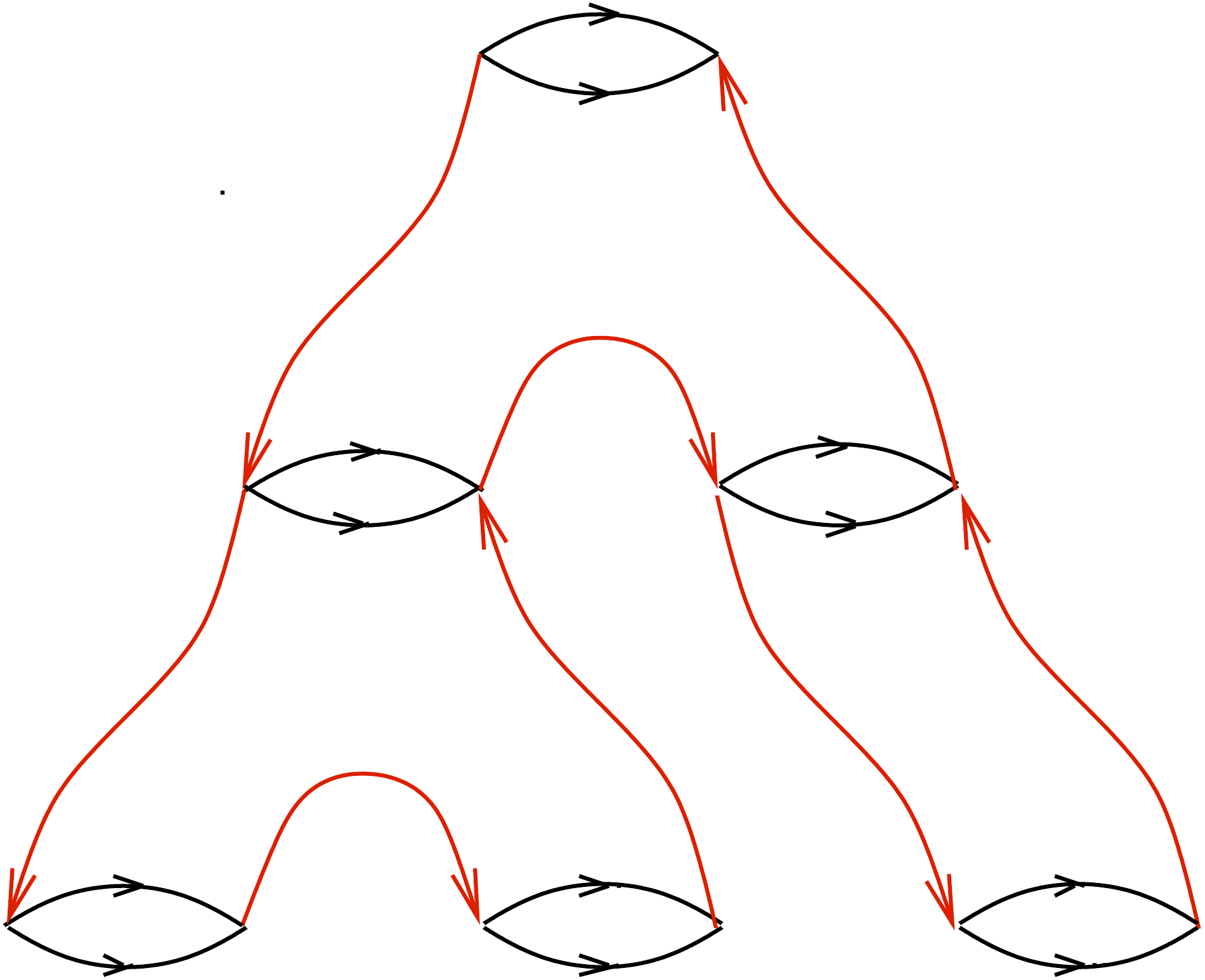}} \quad \cong \quad \raisebox{-13pt}{\includegraphics[height=0.6in]{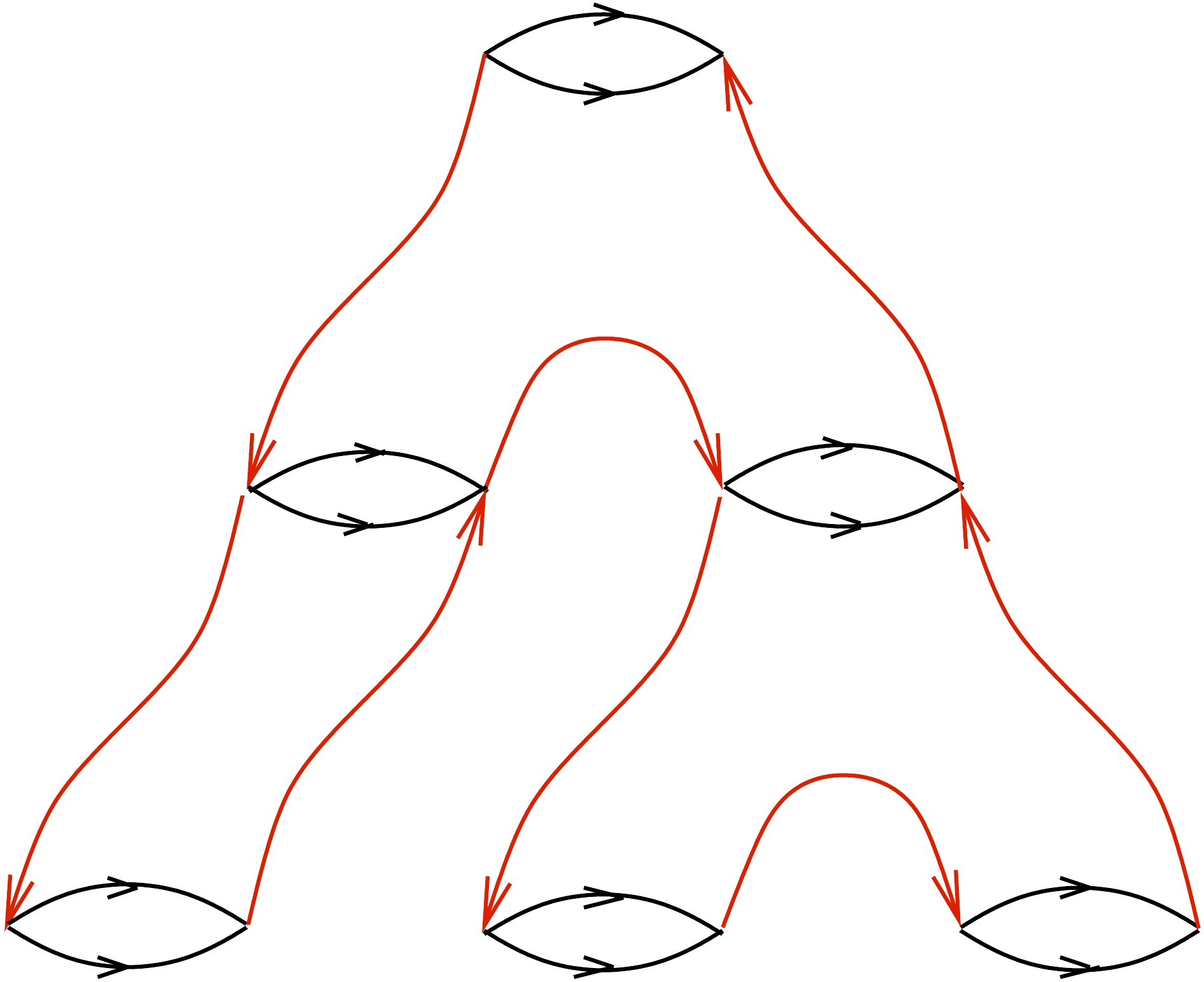}}  \hspace{1cm} \raisebox{-13pt}{\includegraphics[height=0.6in]{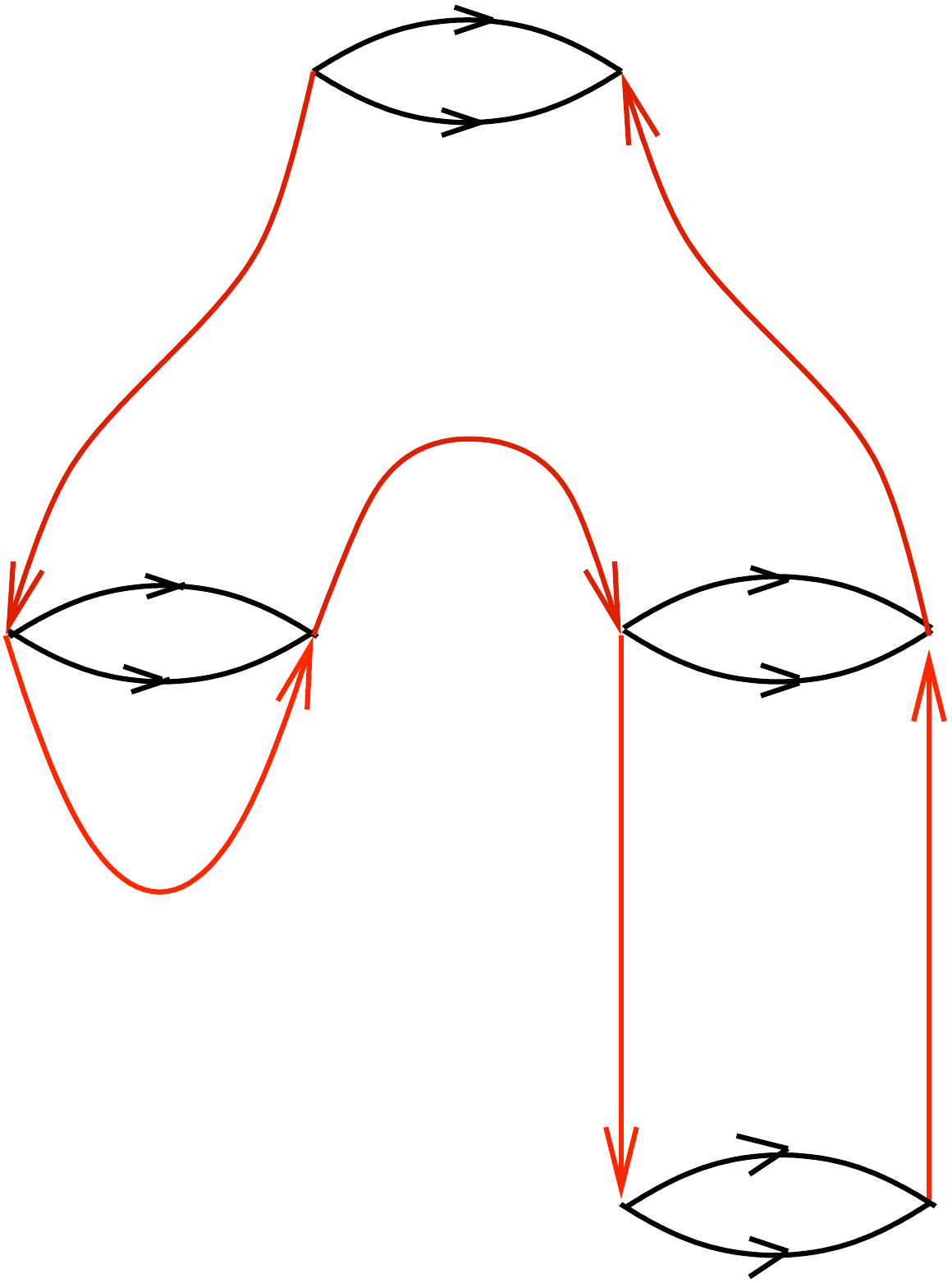}}\quad \cong \quad  \raisebox{-13pt}{\includegraphics[height=0.6in]{web_frob4.pdf}}\quad \cong \quad \raisebox{-13pt}{\includegraphics[height=0.6in]{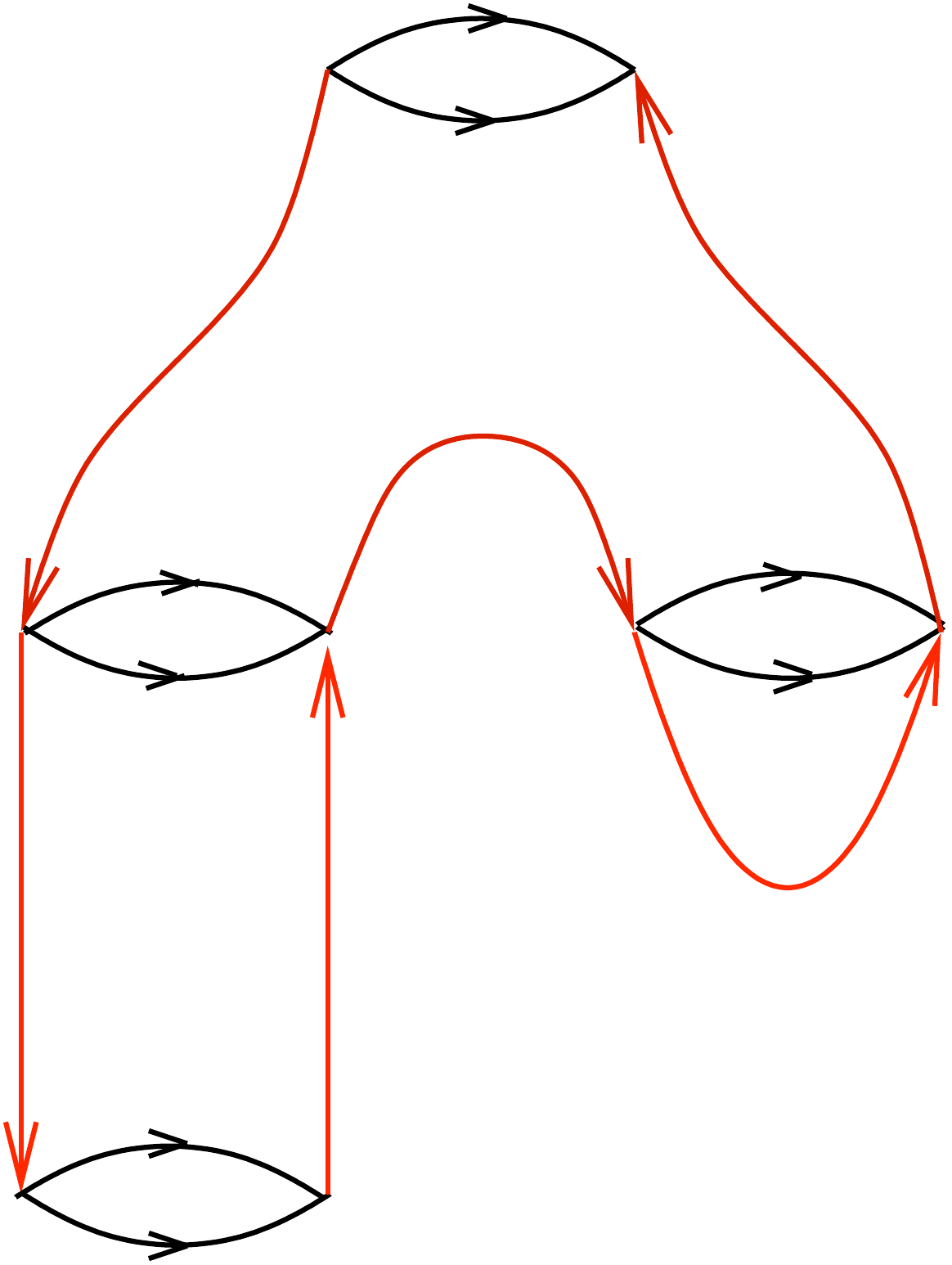}}
\label{eq:web_frob2}
\end{equation*}
\begin{equation*}
 \raisebox{-13pt}{\includegraphics[height=0.6in]{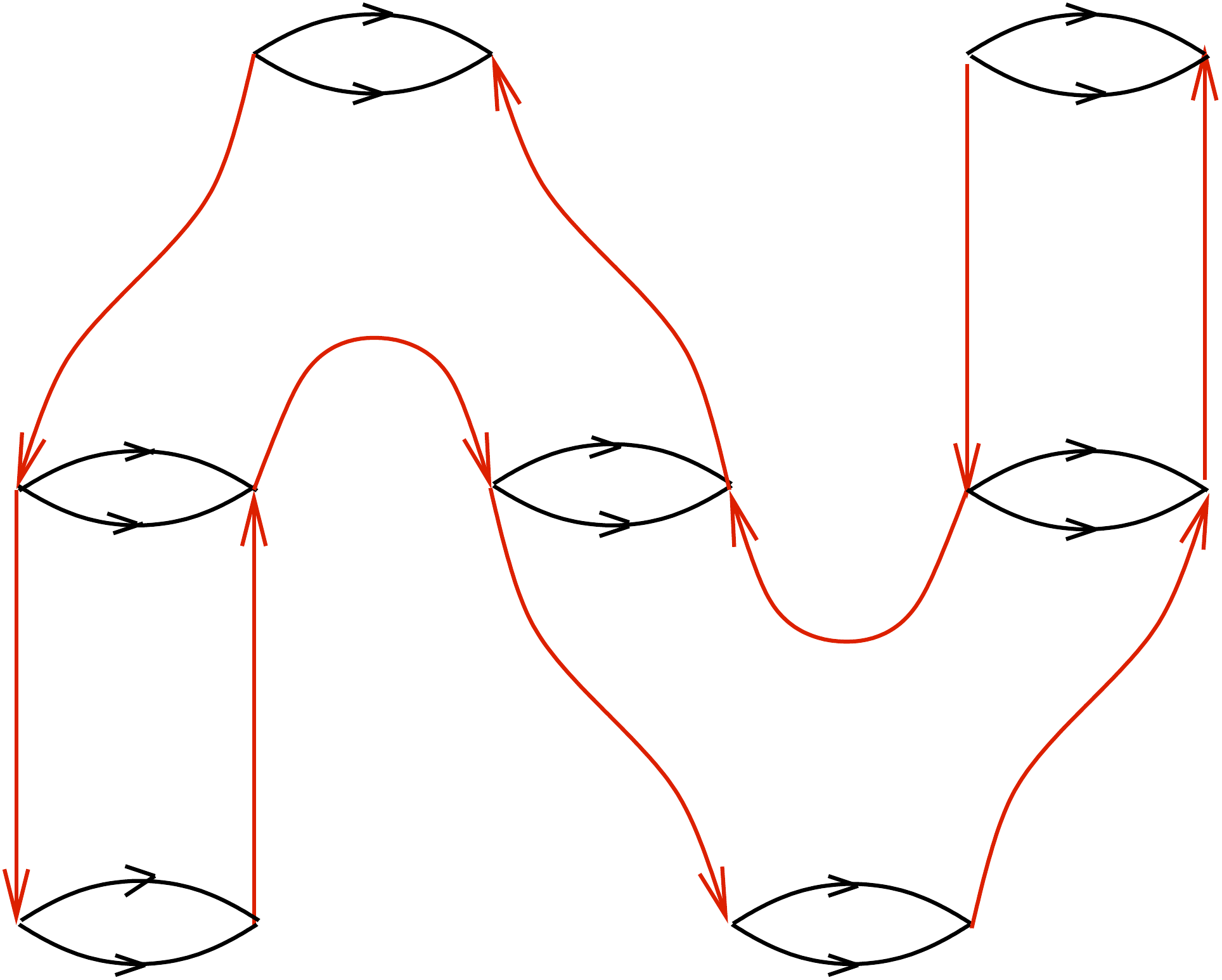}} \quad \cong \quad \raisebox{-13pt}{\includegraphics[height=0.6in]{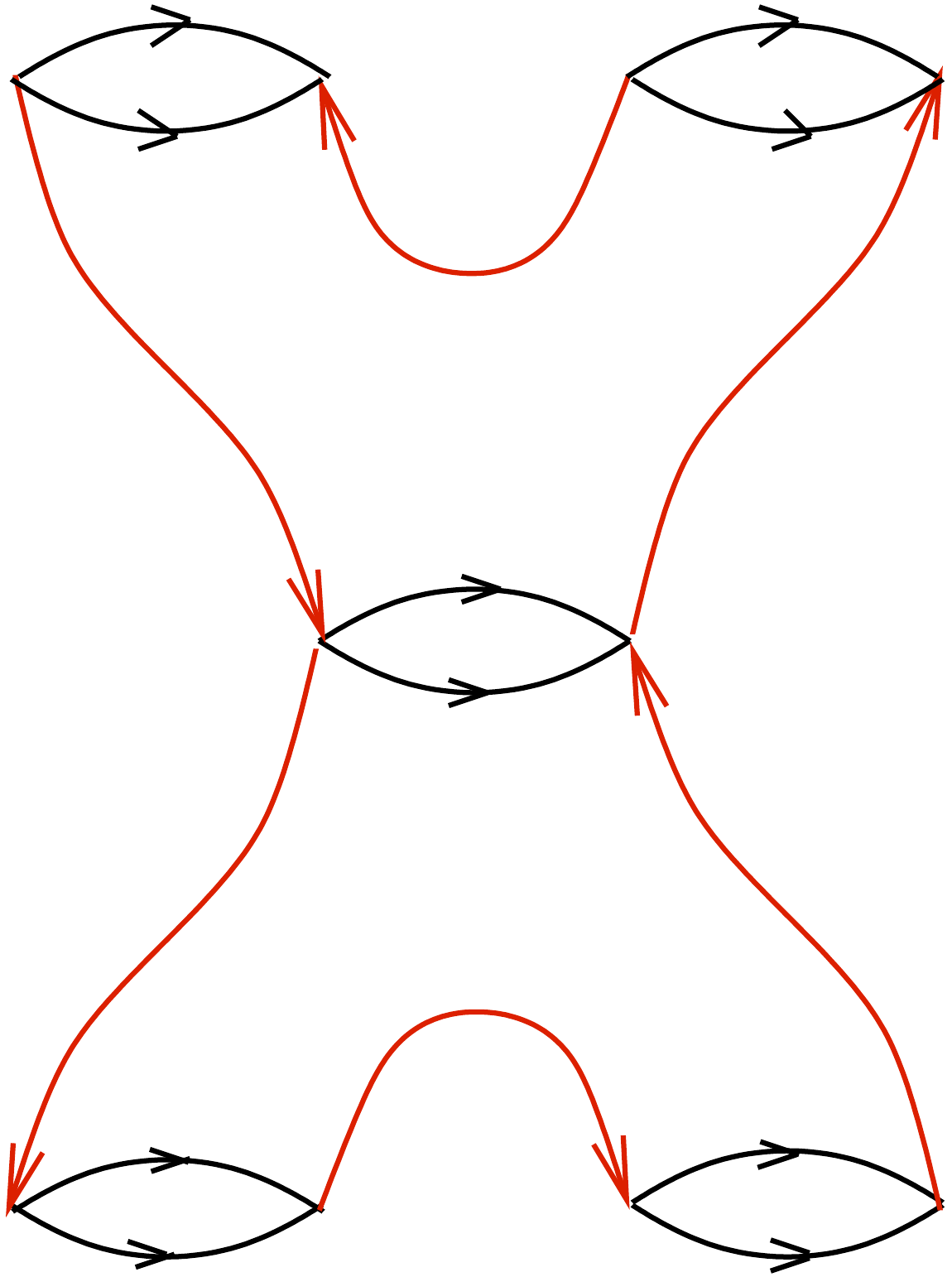}} \quad \cong \quad \raisebox{-13pt}{\includegraphics[height=0.6in]{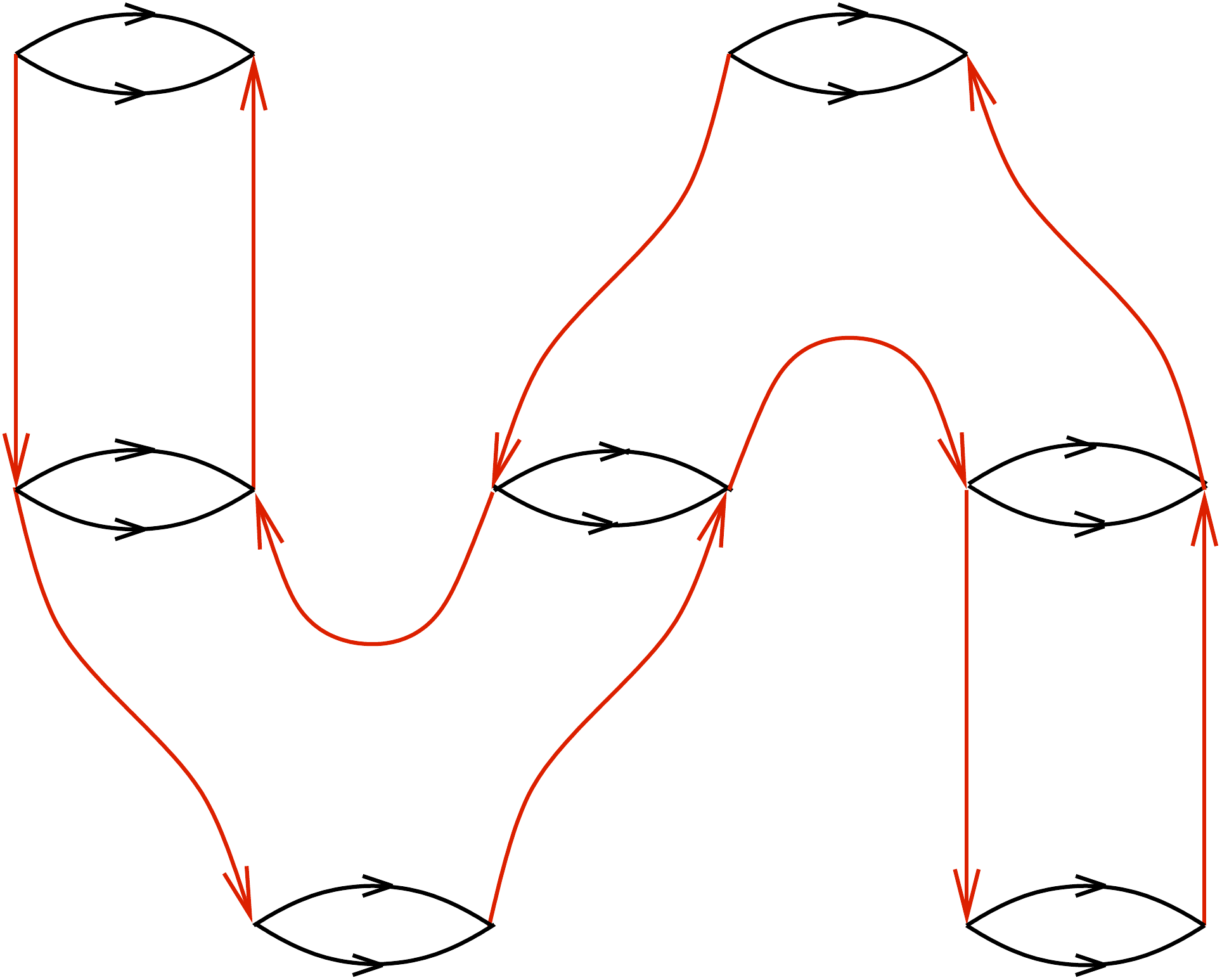}} \hspace{1cm} \raisebox{-13pt}{\includegraphics[height=0.6in]{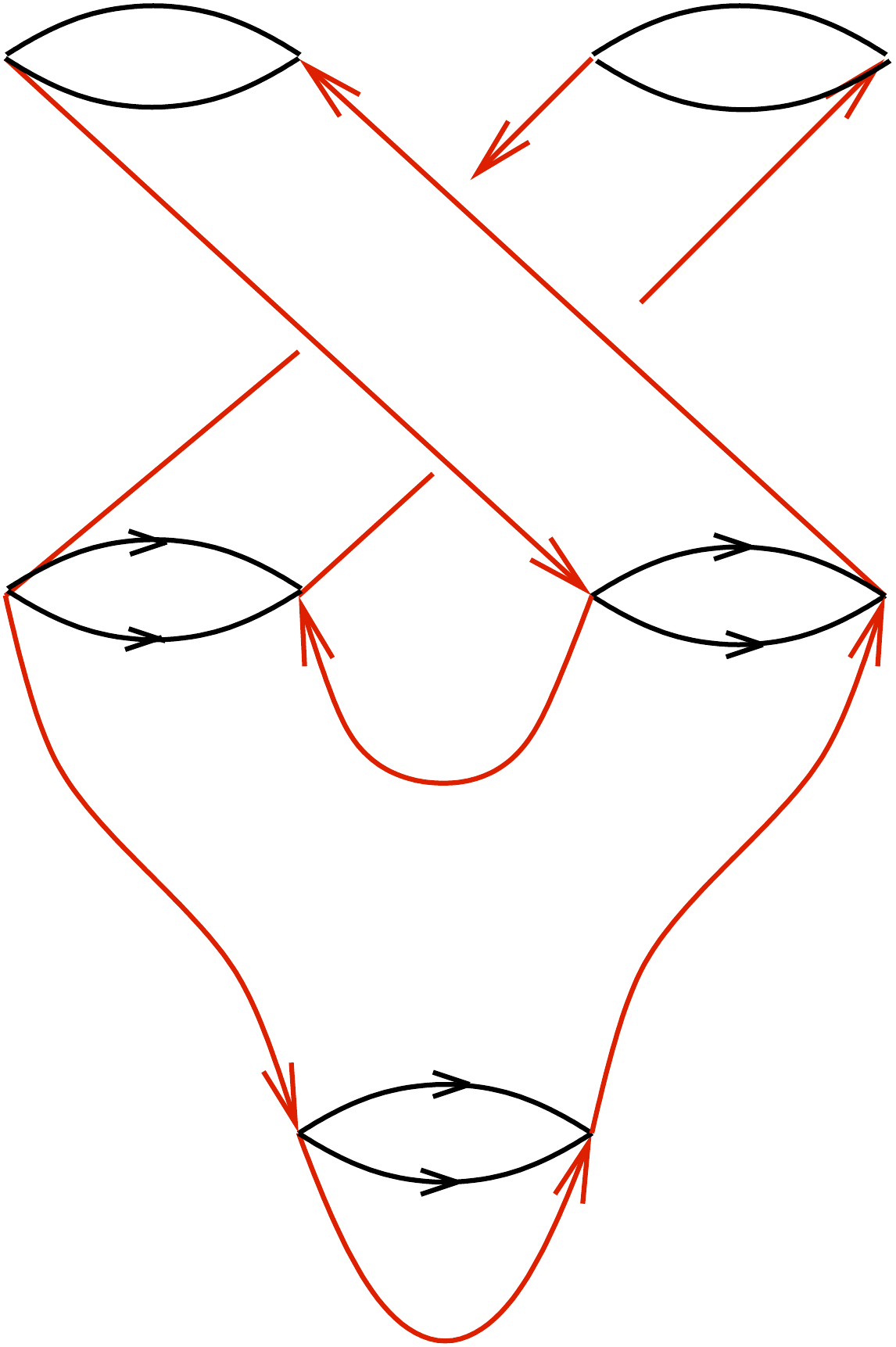}} \quad \cong \quad \raisebox{-13pt}{\includegraphics[height=0.6in]{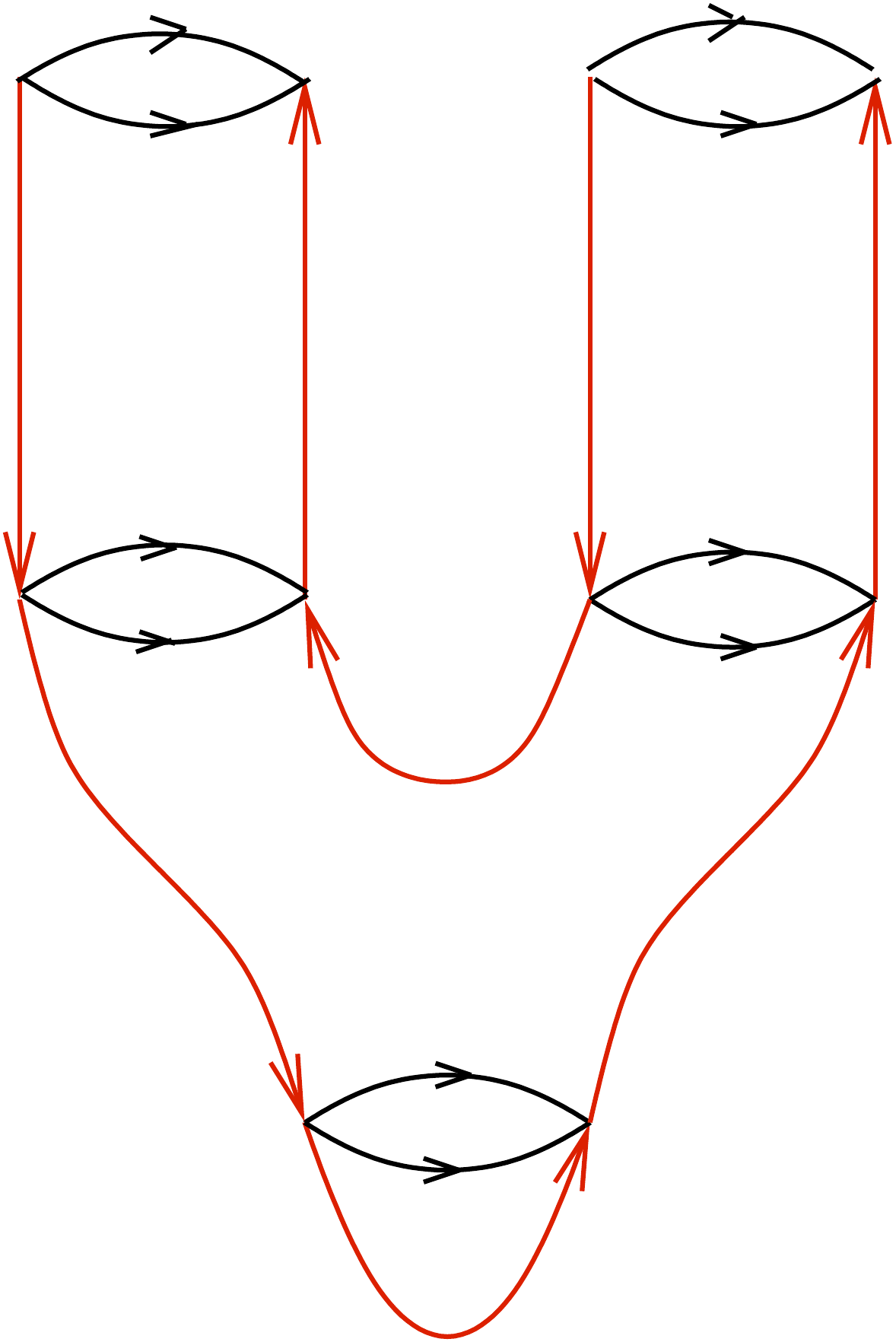}}
 \label{eq:web_frob3}
 \end{equation*}

\item The ``zipper''  cobordisms \,$ \raisebox{-10pt}{\includegraphics[height=0.4in]{zipper.pdf}}$\, and \,$\raisebox{-10pt}{\includegraphics[height=0.4in]{zipper_second.pdf}}$\,  form algebra homomorphisms:

\begin{equation*}
 \psset{xunit=.22cm,yunit=.22cm}
 \begin{pspicture}(4,5)
  \rput(-2, 2.9){\includegraphics[height=0.4in]{zipper.pdf}}
  \rput(2, 2.9){\includegraphics[height=0.4in]{zipper.pdf}}
\rput(0, -0.8){\includegraphics[height=0.4in]{singmult.pdf}}
\end{pspicture} \cong
 \psset{xunit=.22cm,yunit=.22cm}
 \begin{pspicture}(5,5)
  \rput(2.5, 2.5){\includegraphics[height=0.4in]{mult.pdf}}
  \rput(2.5, -1){\includegraphics[height=0.4in]{zipper.pdf}}
\end{pspicture}
\vspace{0.5cm}
\hspace{2cm}
 \psset{xunit=.22cm,yunit=.22cm}
 \begin{pspicture}(2,5)
  \rput(0, 2){\includegraphics[height=0.2in]{unit.pdf}}
  \rput(0, -0.3){\includegraphics[height=0.4in]{zipper.pdf}}
\end{pspicture} \cong
 \psset{xunit=.22cm,yunit=.22cm}
 \begin{pspicture}(2,5)
\rput(2,0.5){\includegraphics[height=0.2in]{singunit.pdf}}
\end{pspicture}
\end{equation*}

\begin{equation*}
 \psset{xunit=.22cm,yunit=.22cm}
 \begin{pspicture}(4,5)
  \rput(-2, 2.9){\includegraphics[height=0.4in]{zipper_second.pdf}}
  \rput(2, 2.9){\includegraphics[height=0.4in]{zipper_second.pdf}}
\rput(0, -0.8){\includegraphics[height=0.4in]{singmult.pdf}}
\end{pspicture} \cong
 \psset{xunit=.22cm,yunit=.22cm}
 \begin{pspicture}(5,5)
  \rput(2.5, 2.5){\includegraphics[height=0.4in]{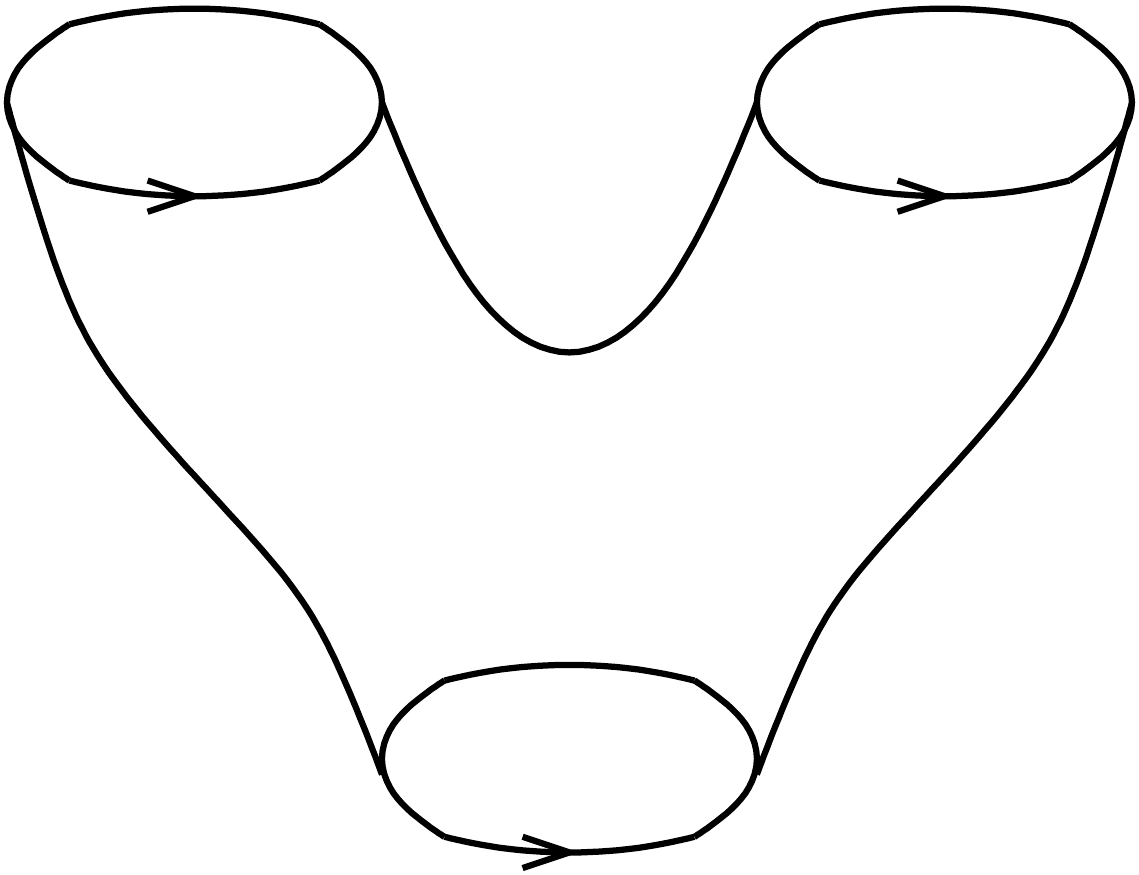}}
  \rput(2.5, -1){\includegraphics[height=0.4in]{zipper_second.pdf}}
\end{pspicture}
\vspace{0.5cm}
\hspace{2cm}
 \psset{xunit=.22cm,yunit=.22cm}
 \begin{pspicture}(2,5)
  \rput(0, 2){\includegraphics[height=0.2in]{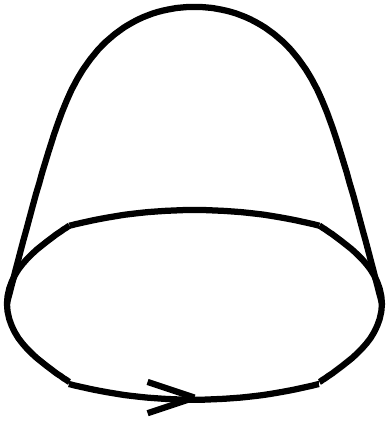}}
  \rput(0, -0.3){\includegraphics[height=0.4in]{zipper_second.pdf}}
\end{pspicture} \cong
 \psset{xunit=.22cm,yunit=.22cm}
 \begin{pspicture}(2,5)
\rput(2,0.5){\includegraphics[height=0.2in]{singunit.pdf}}
\end{pspicture}
\vspace{0.4cm}
\end{equation*}

\item The ``cozipper'' cobordism \,$\raisebox{-10pt}{\includegraphics[height=0.4in]{cozipper.pdf}}$\, is dual to the zipper \,$\raisebox{-10pt}{\includegraphics[height=0.4in]{zipper.pdf}}$\,. Likewise,  the other cozipper \,$\raisebox{-10pt}{\includegraphics[height=0.4in]{cozipper_second.pdf}}$\, is dual to the zipper \,$\raisebox{-10pt}{\includegraphics[height=0.4in]{zipper_second.pdf}}$\,:
\begin{equation*}
 \psset{xunit=.22cm,yunit=.22cm}
 \begin{pspicture}(4,5)
  \rput(-2, 2.7){\includegraphics[height=0.4in]{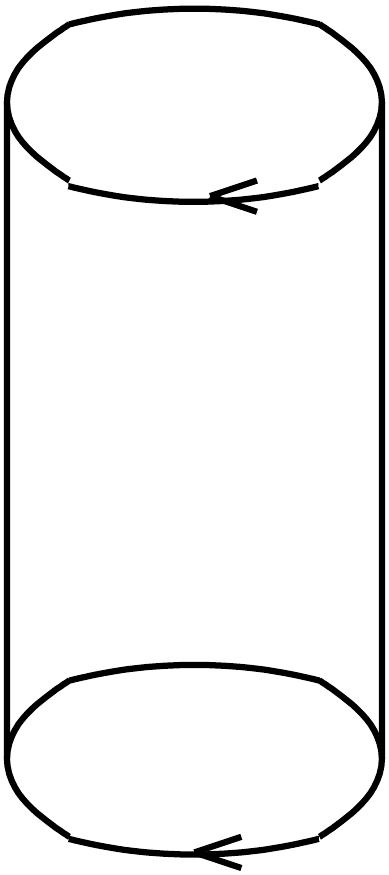}}
  \rput(2, 2.7){\includegraphics[height=0.4in]{cozipper.pdf}}
\rput(0, -0.8){\includegraphics[height=0.4in]{mult.pdf}}
\rput(0, -3.1){\includegraphics[height=0.2in]{counit.pdf}}
\end{pspicture} \cong
 \psset{xunit=.22cm,yunit=.22cm}
 \begin{pspicture}(5,5)
   \rput(2, 2.9){\includegraphics[height=0.4in]{zipper.pdf}}
  \rput(6, 2.9){\includegraphics[height=0.4in]{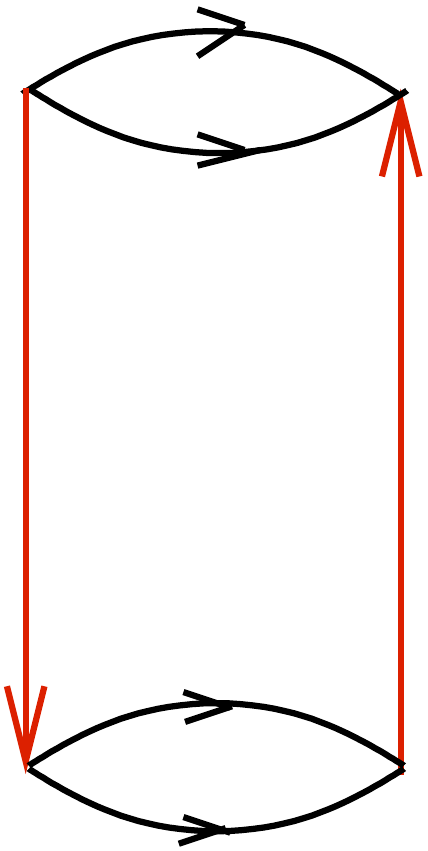}}
\rput(4, -0.8){\includegraphics[height=0.4in]{singmult.pdf}}
\rput(4, -3.2){\includegraphics[height=0.17in]{singcounit.pdf}}
\end{pspicture}
\hspace{3cm}
 \psset{xunit=.22cm,yunit=.22cm}
 \begin{pspicture}(4,6)
  \rput(-2, 2.7){\includegraphics[height=0.4in]{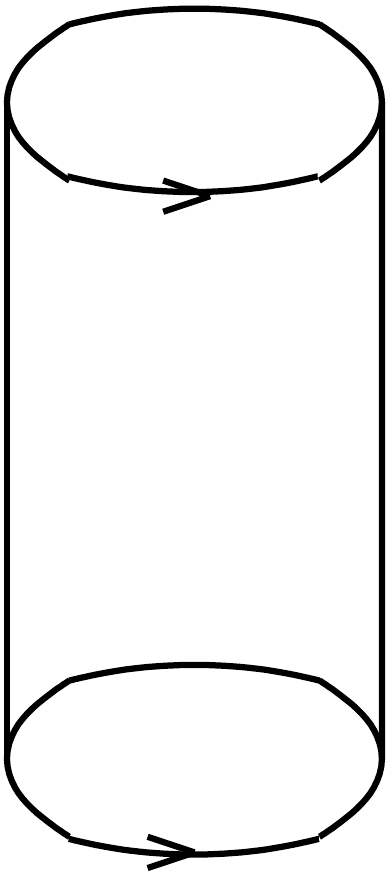}}
  \rput(2, 2.7){\includegraphics[height=0.4in]{cozipper_second.pdf}}
\rput(0, -0.8){\includegraphics[height=0.4in]{mult_second.pdf}}
\rput(0, -3.1){\includegraphics[height=0.2in]{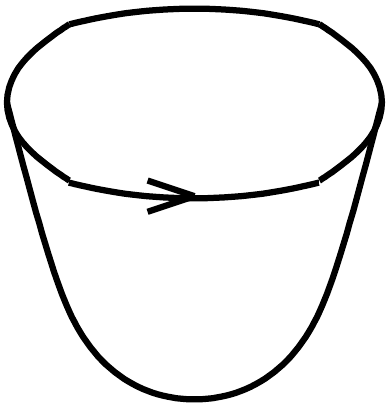}}
\end{pspicture} \cong
 \psset{xunit=.22cm,yunit=.22cm}
 \begin{pspicture}(5,5)
   \rput(2, 2.9){\includegraphics[height=0.4in]{zipper_second.pdf}}
  \rput(6, 2.9){\includegraphics[height=0.4in]{identity_web.pdf}}
\rput(4, -0.8){\includegraphics[height=0.4in]{singmult.pdf}}
\rput(4, -3.2){\includegraphics[height=0.17in]{singcounit.pdf}}
\end{pspicture}
\vspace{0.5cm}
\end{equation*}
\vspace{0.7cm}

 \item \textit{Centrality relations}.

\vspace{0.2cm}
\begin{equation*}
 \psset{xunit=.22cm,yunit=.22cm}
 \begin{pspicture}(4,4)
  \rput(-2, 2.7){\includegraphics[height=0.4in]{identity_web.pdf}}
  \rput(2, 2.7){\includegraphics[height=0.4in]{zipper.pdf}}
\rput(0, -1){\includegraphics[height=0.4in]{singmult.pdf}}
\end{pspicture} \cong
 \psset{xunit=.22cm,yunit=.22cm}
 \begin{pspicture}(4,7)
  \rput(2, 6.2){\includegraphics[height=0.4in]{identity_web.pdf}}
  \rput(6, 6.2){\includegraphics[height=0.4in]{zipper.pdf}}
  \rput(4, 2.5){\includegraphics[height=0.4in]{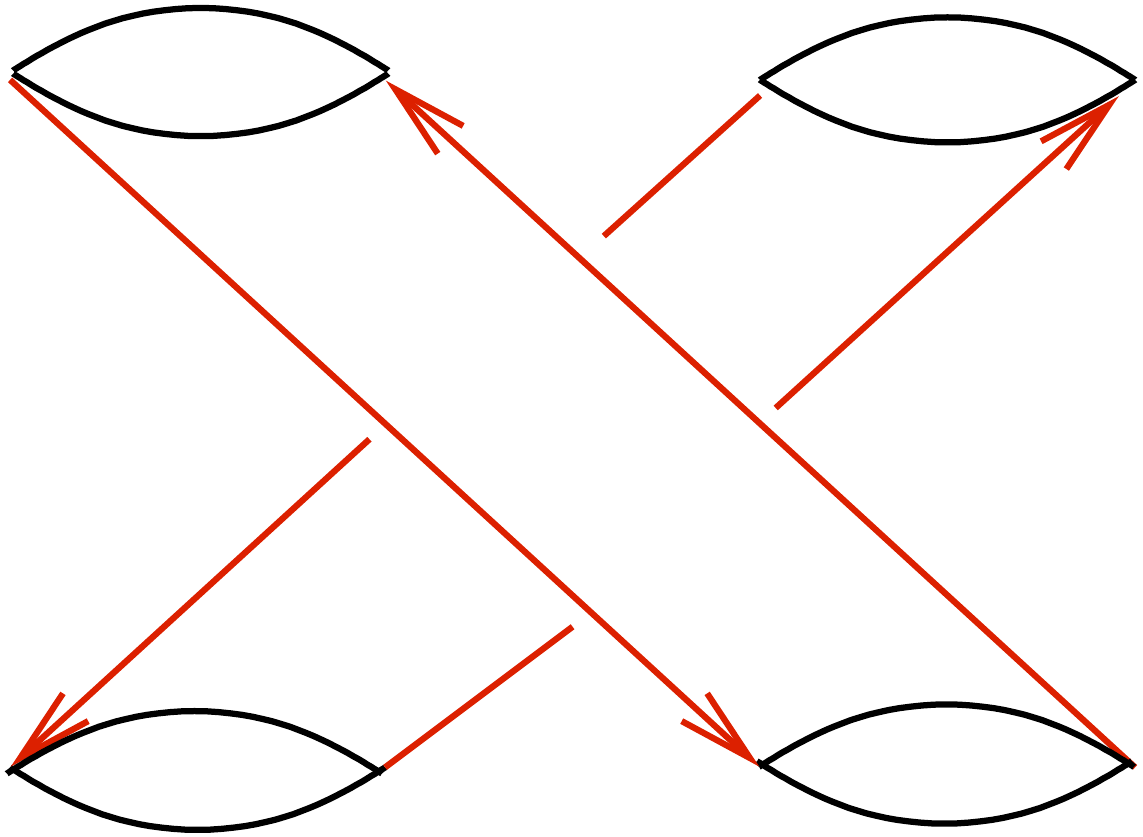}}
\rput(4, -1.3){\includegraphics[height=0.4in]{singmult.pdf}}
\end{pspicture}
\vspace{0.5cm}
\hspace{3cm}
 \psset{xunit=.22cm,yunit=.22cm}
 \begin{pspicture}(4,4)
  \rput(-2, 2.7){\includegraphics[height=0.4in]{identity_web.pdf}}
  \rput(2, 2.7){\includegraphics[height=0.4in]{zipper_second.pdf}}
\rput(0, -1){\includegraphics[height=0.4in]{singmult.pdf}}
\end{pspicture} \cong
 \psset{xunit=.22cm,yunit=.22cm}
 \begin{pspicture}(4,7)
  \rput(2, 6.2){\includegraphics[height=0.4in]{identity_web.pdf}}
  \rput(6, 6.2){\includegraphics[height=0.4in]{zipper_second.pdf}}
  \rput(4, 2.5){\includegraphics[height=0.4in]{twist.pdf}}
\rput(4, -1.2){\includegraphics[height=0.4in]{singmult.pdf}}
\end{pspicture}
\vspace{0.4cm}
\end{equation*}

\item \textit{Genus-one relations}.

\begin{equation*}
\raisebox{-30pt}{\includegraphics[width = 0.5in, height=1in]{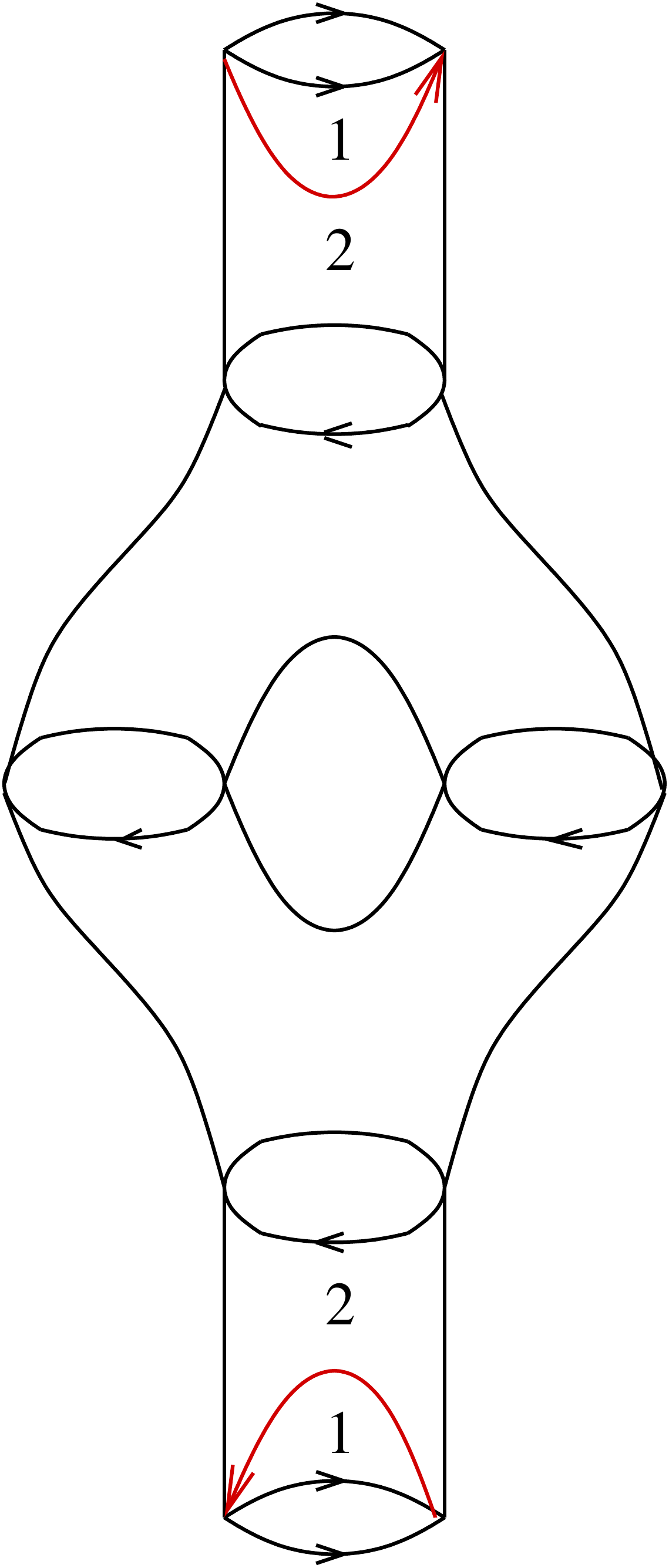}}\quad \cong \quad  \raisebox{-30pt}{\includegraphics[height=1in]{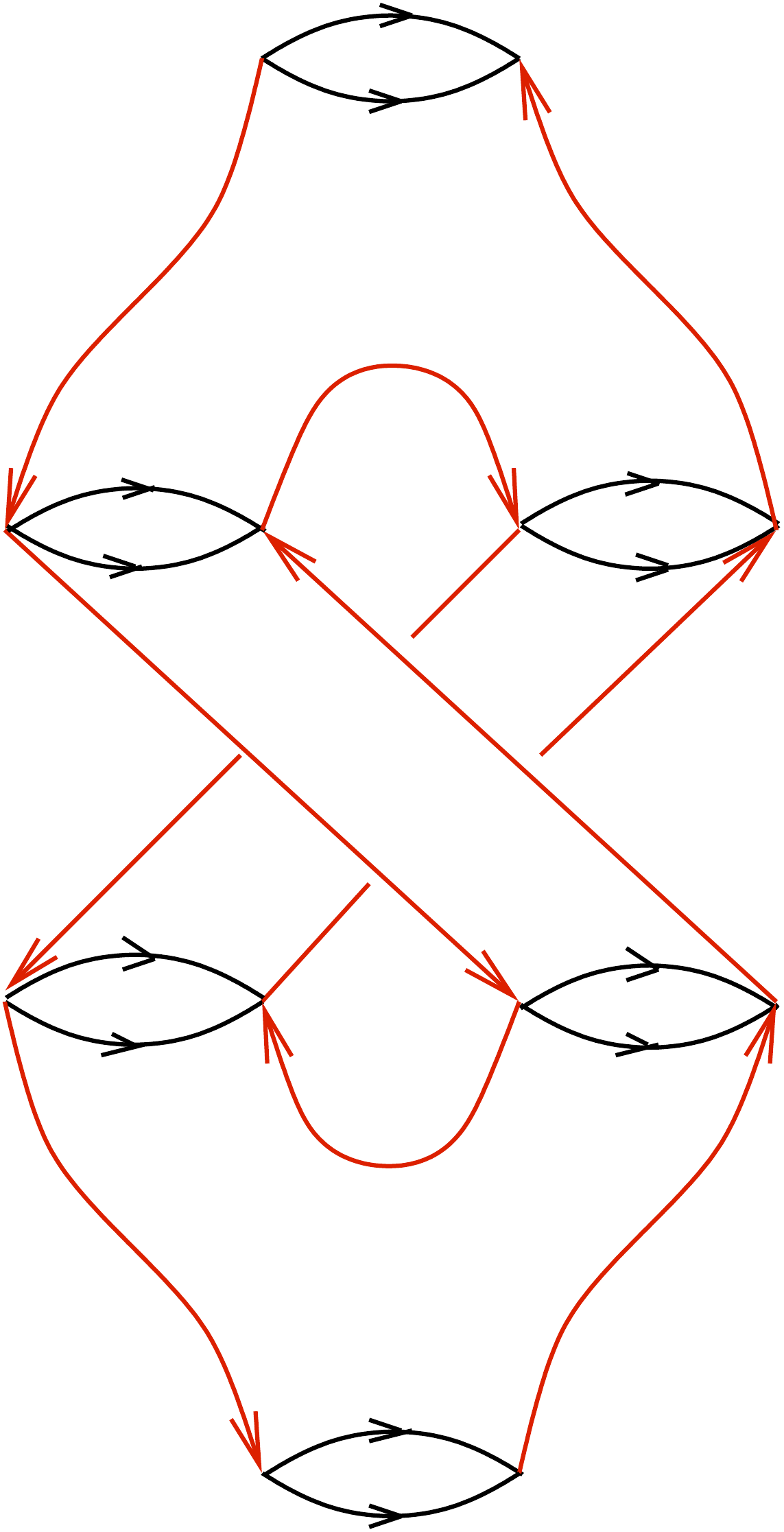}} \quad \cong \quad \raisebox{-30pt}{\includegraphics[width = 0.5in, height=1in]{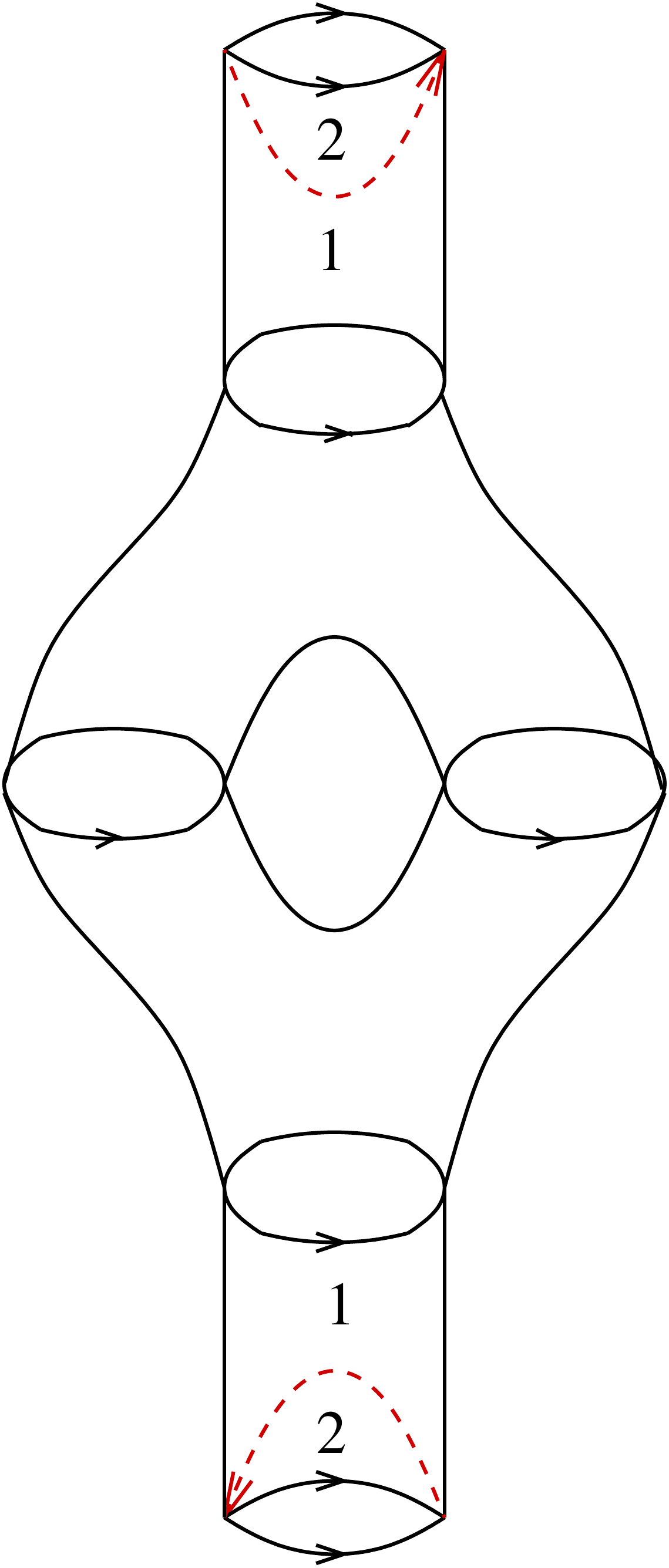}} 
\end{equation*}
\end{enumerate}
\end{proposition}

As a result of the above relations, the following holds: 

\begin{proposition}\label{prop:structure of Sing-2Cob}
The category $\textbf{Sing-2Cob}$ of singular 2-cobordisms is equivalent to the symmetric monoidal category freely generated by an enhanced twin Frobenius algebra.
 \end{proposition} 


\subsection{The category $\textbf{eSing-2Cob}$}

We extend the set of morphisms in $\textbf{Sing-2Cob}$ to allow formal linear combinations of the original morphisms---with coefficients in the ground field or ring (recall that the category $\mathbf{C}$ is either $\textbf{Vect}_k$ or $\textbf{R-Mod}$, where $k$ is a field and $R$ a commutative ring)---and to extend the composition maps in the natural bilinear way. For the remaining of the paper, we require that the ground field/ring contains the fourth root of unity $i.$

We mod out the set of morphisms of $\textbf{Sing-2Cob}$ by the following local relations and denote the resulting category by $\textbf{eSing-2Cob}$:
 \begin{equation} \label{imposed-relation1}
  \psset{xunit=.22cm,yunit=.22cm}
 \begin{pspicture}(3,5)
  \rput(1, 2.6){\includegraphics[height=0.4in]{cozipper.pdf}}
  \rput(1, -0.8){\includegraphics[height=0.4in]{zipper.pdf}}
   \end{pspicture} 
    = -i\,  \raisebox{-10pt}{\includegraphics[height=0.4in]{identity_web.pdf}} \hspace{2cm}
 \psset{xunit=.22cm,yunit=.22cm}
 \begin{pspicture}(3,5)
  \rput(1, 2.9){\includegraphics[height=0.4in]{zipper.pdf}}
  \rput(1, -0.8){\includegraphics[height=0.4in]{cozipper.pdf}}
   \end{pspicture}
 = -i \, \raisebox{-10pt}{\includegraphics[height=0.4in]{identity_circle.pdf}}
  \end{equation}
  
   \begin{equation} \label{imposed-relation2}
  \psset{xunit=.22cm,yunit=.22cm}
 \begin{pspicture}(3,5)
  \rput(1, 2.6){\includegraphics[height=0.4in]{cozipper_second.pdf}}
  \rput(1, -0.8){\includegraphics[height=0.4in]{zipper_second.pdf}}
   \end{pspicture} 
  = i\,  \raisebox{-10pt}{\includegraphics[height=0.4in]{identity_web.pdf}} \hspace{2cm}
 \psset{xunit=.22cm,yunit=.22cm}
 \begin{pspicture}(3,5)
  \rput(1, 2.9){\includegraphics[height=0.4in]{zipper_second.pdf}}
  \rput(1, -0.8){\includegraphics[height=0.4in]{cozipper_second.pdf}}
   \end{pspicture}
 = i \, \raisebox{-10pt}{\includegraphics[height=0.4in]{identity_circle_sec.pdf}}
 \vspace{0.3cm}
  \end{equation}
  
 A quick inspection reveals that relations~\eqref{imposed-relation1}-\eqref{imposed-relation2} are particular cases of the ``curtain identities" (CI-1)-(CI-2). The reason for imposing the above relations is motivated by our goal of constructing a TQFT-approach cohomology for links isomorphic to the (tautological) $sl(2)$ foam cohomology, thus this new theory must satisfy the local relations of the $sl(2)$ foam cohomology. This reason will become more transparent in Section~\ref{sec:link cohomology}, via the `resolution-simplification step'.
 
\begin{remark}\label{rem1}
Considered as an element in the category $\textbf{eSing-2Cob}$, the bi-web forms a commutative Frobenius algebra:
  
     \begin{equation*}
  \psset{xunit=.22cm,yunit=.22cm}
 \begin{pspicture}(5,9)
  \rput(3, 8.3){\includegraphics[height=0.4in]{identity_web.pdf}}
 \rput(-1, 8.3){\includegraphics[height=0.4in]{identity_web.pdf}}
  \rput(1, 4.6){\includegraphics[height=0.4in]{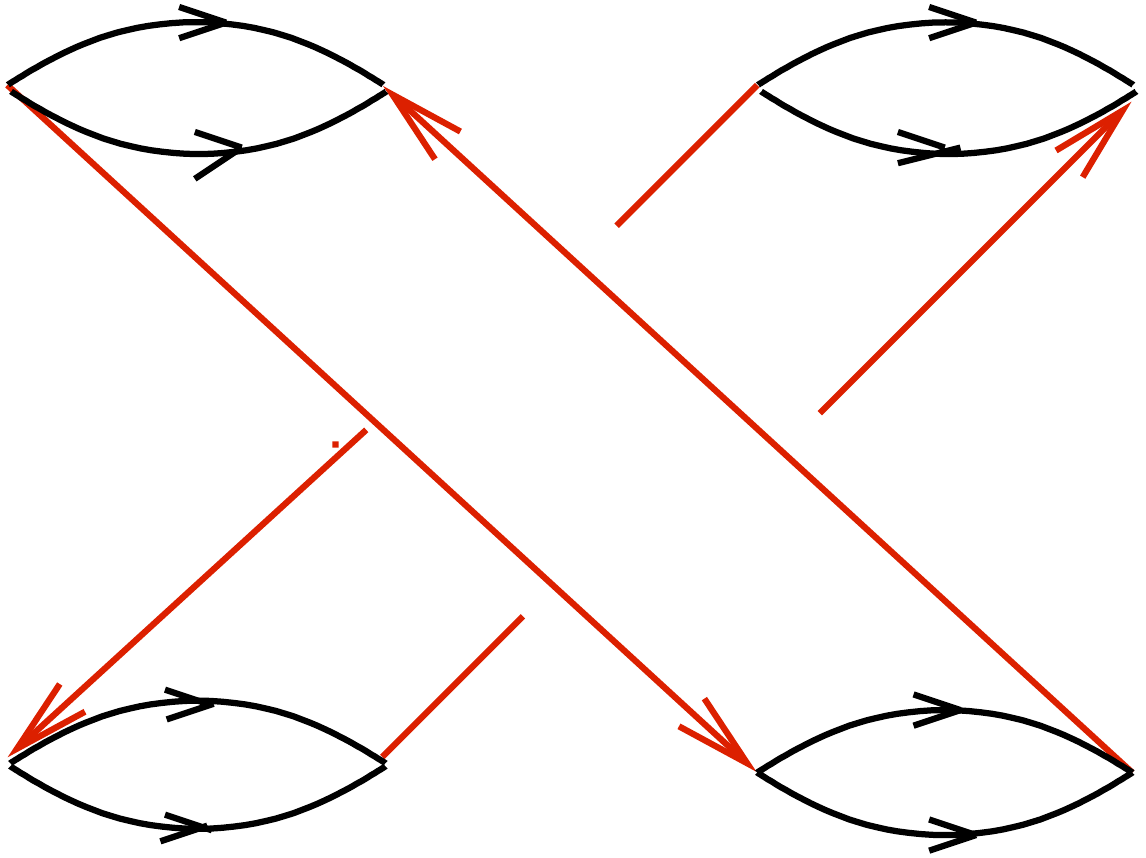}}
  \rput(1, 1){\includegraphics[height=0.4in]{singmult.pdf}}
   \end{pspicture} 
   \stackrel{\eqref{imposed-relation1}}{=} i^2 \,\,
    \begin{pspicture}(7,10)
  \rput(5, 8.3){\includegraphics[height=0.4in]{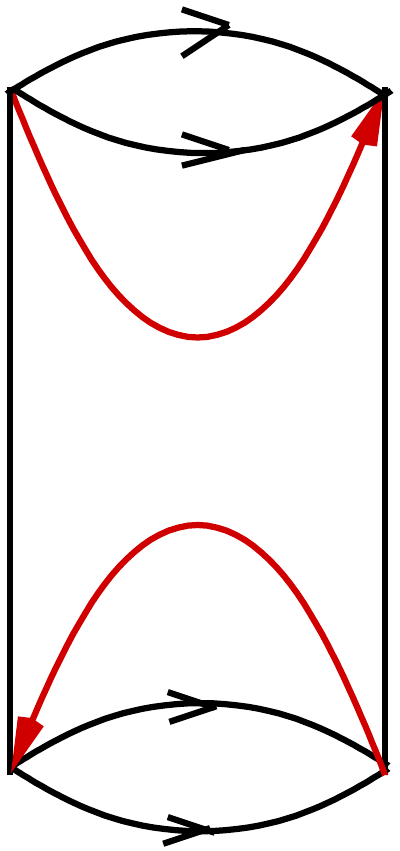}}
 \rput(1, 8.3){\includegraphics[height=0.4in]{cozipper_zipper.pdf}}
  \rput(3, 4.6){\includegraphics[height=0.4in]{braiding_WW.pdf}}
  \rput(3, 1){\includegraphics[height=0.4in]{singmult.pdf}}
   \end{pspicture} 
   \cong i^2\,\,
     \begin{pspicture}(7,10)
  \rput(5, 8.3){\includegraphics[height=0.4in]{cozipper.pdf}}
 \rput(1, 8.3){\includegraphics[height=0.4in]{cozipper.pdf}}
  \rput(3, 4.8){\includegraphics[height=0.4in]{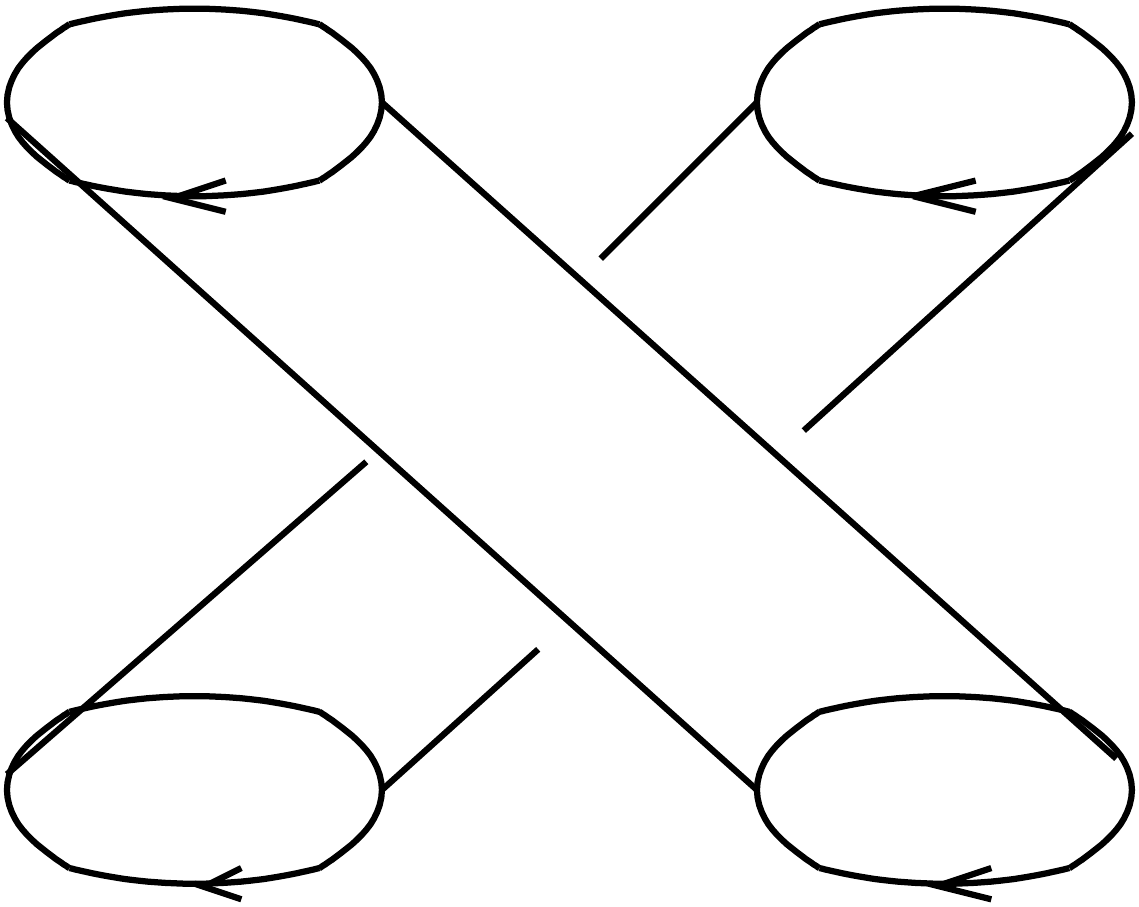}}
    \rput(5, 1.3){\includegraphics[height=0.4in]{zipper.pdf}}
      \rput(1, 1.3){\includegraphics[height=0.4in]{zipper.pdf}}
  \rput(3, -2.4){\includegraphics[height=0.4in]{singmult.pdf}}
   \end{pspicture} 
     \cong i^2\,\,
     \begin{pspicture}(7,10)
  \rput(5, 4.7){\includegraphics[height=0.4in]{cozipper_zipper.pdf}}
 \rput(1, 4.7){\includegraphics[height=0.4in]{cozipper_zipper.pdf}}
  \rput(3, 1){\includegraphics[height=0.4in]{singmult.pdf}}
   \end{pspicture} 
     \stackrel{\eqref{imposed-relation1}}{=} 
     \begin{pspicture}(7,10)
  \rput(5, 4.7){\includegraphics[height=0.4in]{identity_web.pdf}}
   \rput(1, 4.7){\includegraphics[height=0.4in]{identity_web.pdf}}
  \rput(3, 1){\includegraphics[height=0.4in]{singmult.pdf}}
   \end{pspicture} 
    \end{equation*}
    \vspace{0.5cm}
 
Moreover, the ``centrality relations" are implied by relations~\eqref{imposed-relation1}-\eqref{imposed-relation2} and hold with equality in $\textbf{eSing-2Cob}$. Furthermore, the ``genus-one relations" take the following form in the new category $\textbf{eSing-2Cob}$:

\begin{equation*}
-\,\,\raisebox{-30pt}{\includegraphics[width = 0.5in, height=1in]{genus1_rel1.pdf}}\quad = \quad  \raisebox{-30pt}{\includegraphics[height=1in]{genus1_rel_sing.pdf}} \quad = \quad -\,\, \raisebox{-30pt}{\includegraphics[width = 0.5in, height=1in]{genus1_rel2.pdf}} 
\end{equation*}
\end{remark}

Noticing the minus signs in the above equations, we recall that the twin TQFT that we mentioned about in the introduction and given in~\cite[Example 1]{CC3} failed to satisfy the ``genus-one condition" of a twin Frobenius algebra because of a minus sign.

\begin{remark}\label{rem2}
The imposed relations~\eqref{imposed-relation1}-\eqref{imposed-relation2} imply that in the category $\textbf{eSing-2Cob}$,  the zippers \,$\raisebox{-10pt}{\includegraphics[height=0.4in]{zipper.pdf}}$\, and \,$\raisebox{-10pt}{\includegraphics[height=0.4in]{zipper_second.pdf}}$\, are isomorphisms with inverses \,$i \,\raisebox{-10pt}{\includegraphics[height=0.4in]{cozipper.pdf}}$\, and $-i \,\raisebox{-10pt}{\includegraphics[height=0.4in]{cozipper_second.pdf}}$\,, respectively (compare with the isomorphisms~\eqref{isomorphisms}).   
It follows that the cobordisms \,$ \raisebox{-13pt}{\includegraphics[height=0.4in]{morphism_f.pdf}}$\, and \,$\raisebox{-13pt}{\includegraphics[height=0.4in]{morphism_g.pdf}} $\, are mutually inverse isomorphisms:
\begin{equation*}
 \psset{xunit=.22cm,yunit=.22cm}
 \begin{pspicture}(1.5,5)
  \rput(0,2.5){\includegraphics[height=0.4in]{morphism_f.pdf}}
  \rput(0, -1){\includegraphics[height=0.4in]{morphism_g.pdf}}
\end{pspicture} = \raisebox{-10pt}{\includegraphics[height=0.4in]{identity_circle.pdf}}
\hspace{2cm}
 \psset{xunit=.22cm,yunit=.22cm}
 \begin{pspicture}(1.5,5)
  \rput(0,2.5){\includegraphics[height=0.4in]{morphism_g.pdf}}
  \rput(0, -1){\includegraphics[height=0.4in]{morphism_f.pdf}}
\end{pspicture} = \raisebox{-10pt}{\includegraphics[height=0.4in]{identity_circle_sec.pdf}}
\end{equation*}
\vspace{0.1cm}
\end{remark}

Proposition~\ref{prop:structure of Sing-2Cob} reveals the algebraic structure of the category $\textbf{Sing-2Cob}$. Then a natural question arises:

\begin{question}
What is the algebraic description of the new category $\textbf{eSing-2Cob}$?
\end{question}

 The answer is given by the above two remarks, along with Proposition~\ref{prop:relations}. First we observe that in the category $\textbf{eSing-2Cob}$, not only the circle but also the bi-web forms a commutative Frobenius algebra, and that $z_1$ and $iz_1^*$ must be mutually inverse isomorphisms, as well as $z_2$ and $-iz_2^*$. That is:
\begin{equation}\label{eq:quotient algebra}
 (iz_1^*) \circ z_1 = \id_C, \,\, z_1 \circ (iz_1^*) = \id_W \,\, \text{and} \,\, (-iz_2^*) \circ z_2 = \id_C, \,\, z_2 \circ (-iz_2^*) = \id_W. 
 \end{equation} 

In particular, we have that $z_1$ and $z_2$ are algebra isomorphisms, while $z_1^*$ and $z_2^*$ are coalgebra isomorphisms such that $f = z_2^* \circ z_1$ and $g = z_1^* \circ z_2$ are mutually inverse isomorphisms of algebra objects in $\mathbf{C}$.

Second, the ``centrality condition" $m_W \circ (\id_W \otimes z_{k}) = m_W \circ \tau_{W, W} \circ (\id_W \otimes z_{k})$ and the ``genus-one condition" $z_k \circ m_C \circ \Delta_C \circ z_k^* = m_W \circ \tau_{W,W} \circ \Delta_W$ that an enhanced twin Frobenius algebra is required to satisfy need not appear in the definition of the new algebra, as they are implied by the imposed relations~\eqref{eq:quotient algebra} (actually the genus-one condition `holds' with a minus sign in the new algebra; i.e. the new algebra satisfies $- z_k \circ m_C \circ \Delta_C \circ z_k^* = m_W \circ \tau_{W,W} \circ \Delta_W$, for each $k = 1,2$.)

We put these together and give a formal definition/description of the algebraic structure governing the category $\textbf{eSing-2Cob}$.

\begin{definition} \label{def:new algebra}
An \textit{ identical twin Frobenius algebra} $\textbf{iT} := (C, W, z_1, z_1^*,z_2, z_2^*)$ in $\textbf{C}$ consists of two commutative Frobenius algebras 
\[C = (C, m_C, \iota_C, \Delta_C, \epsilon_C)\,\, \text{and} \,\ W = (W, m_W, \iota_W, \Delta_W, \epsilon_W)\] and four morphisms $z_1, z_2 \co C \to W$ and $z_1^*, z_2^* \co W \to C$, such that $z_1, z_2$ are isomorphisms of algebra objects in $\textbf{C}$ with dual (coalgebra) isomorphisms $z_1^*, z_2^*$, respectively, and inverse isomorphisms $iz_1^*, -iz_2^*$, respectively.
\end{definition}

A homomorphism of identical twin Frobenius algebras and the tensor product of two identical twin Frobenius algebras are defined similarly as their analogues corresponding to enhanced twin Frobenius algebras (see Definition~\ref{def:homs in eT-Frob}).

\begin{definition}
An \textit{identical twin TQFT} in \textbf{C} is a symmetric monoidal functor $\textbf{eSing-2Cob} \to \mathbf{C}$. A \textit{homomorphism of identical twin TQFTs} is a monoidal natural transformation of such functors.
\end{definition}

Given an identical twin TQFT, call it $\mathcal{T}$, there is an associated identical twin Frobenius algebra $\textbf{iT} = (C, W, z_1, z_1^*, z_2, z_2^*)$ such that $\mathcal{T}(\emptyset) = \textbf{1}$ and $\mathcal{T}(0_{-}) = C = \mathcal{T}(0_{+}) , \mathcal{T}(1) = W$. The structure maps of the algebra $\textbf{iT}$ are the images under $\mathcal{T}$ of the generators of \textbf{eSing-2Cob}, as explained in Figure~\ref{fig:TQFT}.
 
 \begin{figure}[ht]
\begin{equation*} \mathcal{T} \co\raisebox{-13pt}{\includegraphics[height=0.4in]{comult.pdf}} \to \Delta_C, \quad \mathcal{T} \co \raisebox{-13pt}{\includegraphics[height=0.4in]{mult.pdf}} \to m_C, \quad \mathcal{T} \co \raisebox{-5pt}{\includegraphics[height=.2in]{unit.pdf}} \to \iota_C,  \quad \mathcal{T} \co \raisebox{-5pt}{\includegraphics[height=0.2in]{counit.pdf}} \to \epsilon_C \label{eq:generators_circle} 
\end{equation*}
\begin{equation*} \mathcal{T} \co\raisebox{-13pt}{\includegraphics[height=0.4in]{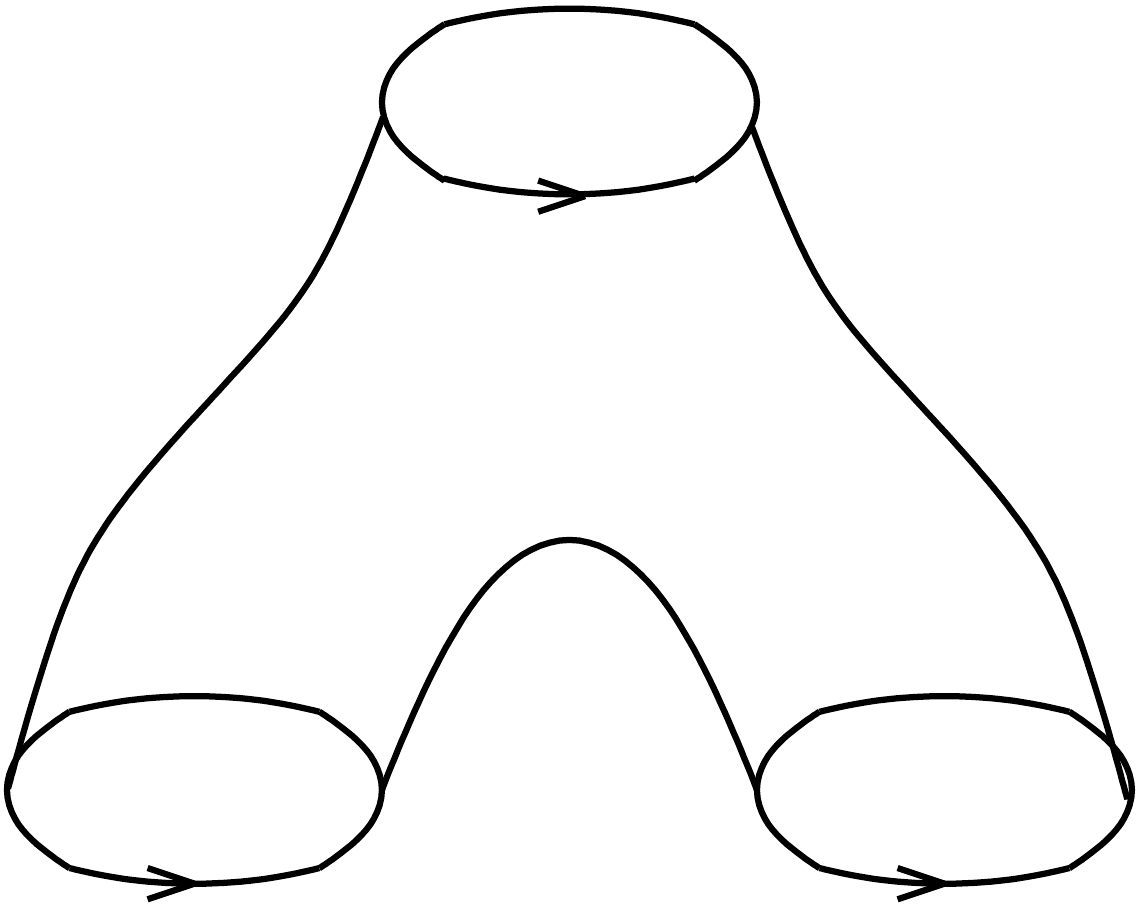}} \to \Delta_C, \quad \mathcal{T} \co \raisebox{-13pt}{\includegraphics[height=0.4in]{mult_second.pdf}} \to m_C, \quad \mathcal{T} \co \raisebox{-5pt}{\includegraphics[height=.2in]{unit_second.pdf}} \to \iota_C,  \quad \mathcal{T} \co \raisebox{-5pt}{\includegraphics[height=0.2in]{counit_second.pdf}} \to \epsilon_C \label{eq:generators_circle} 
\end{equation*}
\begin{equation*} \mathcal{T} \co \raisebox{-13pt}{\includegraphics[height=0.4in]{singcomult.pdf}} \to \Delta_W, \quad  \mathcal{T} \co \raisebox{-13pt}{\includegraphics[height=0.4in]{singmult.pdf}} \to m_W, \quad \mathcal{T} \co \raisebox{-5pt}{\includegraphics[height=.2in]{singunit.pdf}} \to \iota_W,  \quad \mathcal{T} \co \raisebox{-5pt}{\includegraphics[height=0.2in]{singcounit.pdf}} \to \epsilon_W \label{eq:generators_web}
\end{equation*}
\begin{equation*}
\mathcal{T} \co\raisebox{-13pt}{\includegraphics[height=0.4in]{zipper.pdf}} \to z_1, \quad \mathcal{T} \co\raisebox{-13pt}{\includegraphics[height=0.4in]{cozipper.pdf}} \to z_1 ^*, \quad
\mathcal{T} \co\raisebox{-13pt}{\includegraphics[height=0.4in]{zipper_second.pdf}} \to z_2, \quad \mathcal{T} \co\raisebox{-13pt}{\includegraphics[height=0.4in]{cozipper_second.pdf}} \to z_2 ^* .
\end{equation*} 
\caption{The assignments of $\mathcal{T}$ on generators} \label{fig:TQFT} 
\end{figure}

\begin{proposition}\label{prop:structure of eSing-2Cob}
The category $\textbf{eSing-2Cob}$ is equivalent to the symmetric monoidal category freely generated by an identical twin Frobenius algebra.
 \end{proposition} 
 
 \begin{proof}
 The proof is similar to that of Theorem 3 in~\cite{CC3}, and for the sake of brevity, we leave the details to the enthusiastic reader.
 \end{proof}
 
 \begin{corollary}
The category of identical twin Frobenius algebras in \textbf{C} is equivalent, as a symmetric monoidal category, to the category of identical twin TQFTs in \textbf{C}.
\end{corollary}

\begin{definition}
 Let $\Sigma \in \textbf{Sing-2Cob}$ be a singular 2-cobordism. Define the \textit{degree} of $\Sigma$ as $\deg(\Sigma): = -\chi(\Sigma),$ where $\chi(\Sigma)$ is the Euler characteristic of $\Sigma$.
 \end{definition}
  
 The degree of a cobordism is additive under composition. We also remark that the degree of multiplication/comultiplication cobordisms and unit/counit cobordisms (the first two and last two cobordisms, respectively, given in~\eqref{eg:generators_circle} and \eqref{eg:generators_web}) is 1 and $-1$, respectively, while the singular cobordisms depicted in ~\eqref{eg:generators_web_circle_one} have degree zero. 

Thus the category $\textbf{Sing-2Cob}$ is now graded, and since relations \eqref{imposed-relation1} and \eqref{imposed-relation2} are degree-homogeneous, so is the category $\textbf{eSing-2Cob}$. 


\section{An identical twin TQFT}\label{sec:TQFT}
 
We are ready to describe now the degree-preserving (identical twin) TQFT which allows the recovery of the universal dot-free $sl(2)$ foam cohomology for links. 

Let $R = \mathbb{Z}[\frac{1}{2}, i][ a, h]$ be the graded ring considered in Section~\ref{sec:univKh}, and let $\mathbf{C}$ be the category $\textbf{R-Mod}$.
Consider the $R$-module $\mathcal{A} = R [X]/(X^2 - hX - a)$ with inclusion map $\iota \co R \to \mathcal{A}, \iota(1) = 1$, and make it graded by $\text{deg}(1) = -1$ and $\text{deg}(X) = 1$. Equip $\mathcal{A}$ with two commutative Frobenius structures 
 \[ \mathcal{A}_C = (\mathcal{A}, m_C, \iota_C, \Delta_C, \epsilon_C), \quad \mathcal{A}_W = (\mathcal{A}, m_W, \iota_W, \Delta_W, \epsilon_W),\]
 where $\iota_C = \iota_W = \iota.$ The multiplication maps $m_{C,W}\co \mathcal{A} \otimes \mathcal{A} \rightarrow \mathcal{A}$ are given in the basis $\{1, X\}$ by
$$ \begin{cases}
 m_{C,W}(1 \otimes X) = m_{C,W}(X \otimes 1) = X\\ 
m_{C,W}(1 \otimes 1) =1,  m_{C,W}(X \otimes X) = hX + a.
 \end{cases}$$
 The comultiplication maps $\Delta_{C,W} \co \mathcal{A} \to \mathcal{A} \otimes \mathcal{A}$ are dual to the multiplication maps via the trace maps $\epsilon_{C,W} \co \mathcal{A} \to R$ 
  \[\begin{cases}\epsilon_C(1) = 0 \\ \epsilon_C(X) = 1, \end{cases} \quad \,\begin{cases}\epsilon_W(1) = 0 \\ \epsilon_W(X) = -i,\end{cases} \]
and are defined by the rules
 \[ \begin{cases}
 \Delta_C(1) = 1 \otimes X + X \otimes 1-h 1 \otimes 1\\ 
 \Delta_C(X) = X \otimes X + a 1 \otimes 1,\end{cases} \quad  \begin{cases}
 \Delta_W(1) = i(1 \otimes X + X \otimes 1 - h 1\otimes 1)\\ 
 \Delta_W(X) = i(X \otimes X + a 1 \otimes 1).\end{cases} \]
  
We remark that $\mathcal{A}_W$ is a \textit{twisting} of $\mathcal{A}_C.$ Specifically, the comultiplication $\Delta_W$ and counit $\epsilon_W$ are obtained from $\Delta_C$ and $\epsilon_C$ by `twisting' them with invertible element $-i \in \mathcal{A}$: 
  \[ \epsilon_W (x) = \epsilon_C(-ix), \quad \Delta_W(x) = \Delta_C((-i)^{-1}x) = \Delta_C(ix), \,\,\text {for all}\, \,x \in \mathcal{A}. \]

We define the following homomorphisms:
\[z_1 \co \mathcal{A}_C \to \mathcal{A}_W, \begin{cases} z_1(1) = 1 \\z_1(X) = X, \end{cases} \quad z_1^* \co\mathcal{A}_W \to \mathcal{A}_C, \begin{cases} z_1^*(1) = -i \\z_1^*(X) = -iX. \end{cases} \]

\[z_2 \co \mathcal{A}_C \to \mathcal{A}_W, \begin{cases} z_2(1) = 1 \\z_2(X) =  h - X, \end{cases} \quad z_2^* \co\mathcal{A}_W \to \mathcal{A}_C, \begin{cases} z_2^*(1) = i \\z_2^*(X) = i(h - X). \end{cases} \]
Straightforward computations show that $(\mathcal{A}_C, \mathcal{A}_W, z_1, z_1^*, z_2, z_2^*)$ satisfies the axioms of an identical twin Frobenius algebra.  
 
The corresponding identical twin TQFT $\mathcal{T} \co \textbf{eSing-2Cob} \to \textbf{R-Mod}$ assigns the ground ring $R$ to the empty 1-manifold, and assigns $\mathcal{A}^{\otimes k}$ to an object $\textbf{n} = (n_1, n_2, \dots, n_k)$ in \textbf{eSing-2Cob}. The $i$-th factor of $\mathcal{A}^{\otimes k}$  is endowed with the structure $\mathcal{A}_C$ if $n_i = 0_{-} = \raisebox{-3pt}{\includegraphics[height=0.15in]{circle.pdf}} $ or $n_i = 0_{+} = \raisebox{-3pt}{\includegraphics[height=0.15in]{loop.pdf}},$ and with the structure $\mathcal{A}_W$ if $n_i = 1 = \raisebox{-3pt}{\includegraphics[height=0.15in]{singcircle.pdf}}.$ 
 On the generating morphisms of the category \textbf{eSing-2Cob}, the functor $\mathcal{T}$ is defined as depicted in Figure~\ref{fig:TQFT}. 
 
Since $(\mathcal{A}_C, \mathcal{A}_W, z_1, z_1^*, z_2, z_2^*)$ forms an identical twin Frobenius algebra, 
the functor $\mathcal{T}$ respects the relations among the set of generators for \textbf{eSing-2Cob}, and therefore, it is well defined. It is also easy to verify that $\mathcal{T}$ is degree-preserving. (Note that $m_{C,W}$ and $\Delta_{C,W}$ are maps of degree $1,$ while $\iota_{C,W}$ and $\epsilon_{C,W}$ are maps of degree $-1.$)  

 \begin{proposition}\label{prop:functor T and local relations}
 The functor $\mathcal{T}$ satisfies the local relations $\tilde{\ell}$ of the universal dot-free $sl(2)$ foam theory.
 \end{proposition}

 \begin{proof}
 First, let us find the composite morphisms $f = z_2^* \circ z_1$ and $g = z_1^* \circ z_2$:
\[ f \co \mathcal{A}_C \to \mathcal{A}_C, \begin{cases} f(1) = i \\f(X) = i(h - X), \end{cases} \quad g \co\mathcal{A}_C \to \mathcal{A}_C, \begin{cases} g(1) = -i \\ g(X) = -i(h - X). \end{cases} \] 

 Then, notice that $\mathcal{T}\left(\,\raisebox{-4pt}{\includegraphics[width=0.4in]{torus.pdf}}\,\right)=2$ since $\mathcal{T}\left(\,\raisebox{-4pt}{\includegraphics[width=0.4in]{torus.pdf}}\,\right) = \epsilon_C \circ m_C \circ \Delta_C \circ \iota_C,$ and $ (\epsilon_C \circ m_C \circ \Delta_C \circ \iota_C)(1) = (\epsilon_C \circ m_C \circ \Delta_C)(1) = \epsilon_C (2X - h) = 2.$ 

Also $\mathcal{T}\left(\raisebox{-7pt}{\includegraphics[height=.25in]{ufo-g0.pdf}}\right) = 0 = \mathcal{T}\left(\, \raisebox{-8pt}{\includegraphics[width=0.25in]{sph.pdf}}\,\right)$ since $(\epsilon_C \circ f \circ \iota_C)(1) = 0$ and $(\epsilon_C \circ \iota_C)(1)  = \epsilon_C(1) = 0.$ 
  
We have that $\mathcal{T}\left(\,\raisebox{-5pt}{\includegraphics[height=.25in]{cap-g2.pdf}}\,\right) =m_C \circ \Delta_C \circ m_C \circ \Delta_C  \circ \iota_C.$  
Moreover, $(m_C \circ  \Delta_C) (1) = m_C(1 \otimes X + X \otimes 1 - h 1 \otimes 1) = 2X -h$ and $(m_C \circ  \Delta_C) (X) =  m_C(X \otimes X + a 1 \otimes 1) = hX + 2 a.$  
Furthermore, $(m_C \circ \Delta_C \circ m_C \circ \Delta_C \circ \iota_C)(1) =  (m_C \circ \Delta_C)(2X - h) =  2(hX + 2a) - h(2X - h) = h^2 + 4a = (h^2 + 4a) \iota_C(1).$ Therefore  $\mathcal{T}\left(\,\raisebox{-5pt}{\includegraphics[height=.25in]{cap-g2.pdf}}\,\right) = (h^2 + 4a)\,\mathcal{T}\left(\, \raisebox{-5pt}{\includegraphics[height=.2in]{capor.pdf}}\,\right)$.

  Furthermore, $(\epsilon_C \circ f \circ m_C \circ \Delta_C \circ \iota_C)(1) = (\epsilon_C \circ f)(2X -h) = \epsilon_C(2i(h -X) -i h) =  -2i,$ which corresponds to the local relation $\mathcal{T} \left (\raisebox{-8pt}{\includegraphics[height=.3in]{ufo-g1-upper.pdf}}\right ) = -2i.$
 Similar computations show that $\mathcal{T} \left (\raisebox{-8pt}{\includegraphics[height=.3in]{ufo-g1-lower.pdf}} \right)= 2i$ and  $\mathcal{T} \left (\raisebox{-8pt}{\includegraphics[height=.3in]{ufo-g2.pdf}} \right) = 0.$
 
 Finally, $\mathcal{T}\left(\,\raisebox{-15pt}{\includegraphics[height=.45in]{surgery1-g1.pdf}}\,\right) + \mathcal{T}\left(\,\raisebox{-15pt}{\includegraphics[height=.45in]{surgery2-g1.pdf}}\,\right) = [\iota_C \circ (\epsilon_C \circ m_C \circ \Delta_C)] + [(m_C \circ\Delta_C \circ \iota_C) \circ \epsilon_C].$ Then
  $[\iota_C \circ (\epsilon_C \circ m_C \circ \Delta_C)](1) +  [(m_C \circ\Delta_C \circ \iota_C) \circ \epsilon_C](1) = (\iota_C \circ \epsilon_C)(2X -h) + 0 = 2,$ 
and $[\iota_C \circ (\epsilon_C \circ m_C \circ \Delta_C)](X) +  [(m_C \circ\Delta_C \circ \iota_C) \circ \epsilon_C](X) = (\iota_C \circ \epsilon_C)(hX + 2a) + (m_C \circ \Delta_C)(1) = h + (2X -h) = 2X.$ Therefore, the following holds:
  \[ \mathcal{T}\left(\,\raisebox{-15pt}{\includegraphics[height=.45in]{identity_circle.pdf}}\,\right) = \frac{1}{2} \mathcal{T}\left(\,\raisebox{-15pt}{\includegraphics[height=.45in]{surgery1-g1.pdf}}\,\right) + \frac{1}{2}\mathcal{T}\left(\,\raisebox{-15pt}{\includegraphics[height=.45in]{surgery2-g1.pdf}}\,\right).\] \end{proof}
  
  Observe that $\displaystyle \frac{i}{2} [ \iota_W \circ (\epsilon_C \circ m_C \circ \Delta_C \circ z_1^*)] + \frac{i}{2} [(z_1 \circ m_C \circ \Delta_C \circ \iota_C) \circ \epsilon_W] = \id_{\mathcal{A}_W},$ implying that $\mathcal{T}$  satisfies the ``cutting-neck" relation (CN) given in Section~\ref{sec:univKh}:

\begin{equation} \mathcal{T} \left(\, \raisebox{-20pt}{\includegraphics[ height=.7in]{cneck0-new.pdf}}\,\right ) = \displaystyle \frac{i}{2} \mathcal{T} \left (\, \raisebox{-20pt}{\includegraphics[height=.7in]{cneck1-g1-new.pdf}} \,\right ) + \frac{i}{2} \mathcal{T} \left (\,\raisebox{-20pt}{\includegraphics[height=0.7in]{cneck2-g1-new.pdf}} \,\right).\end{equation}

  \begin{proposition}\label{prop:isom_circle}
  The $R$-module map $\alpha_C \co \mathcal{A}_C \to R\{1\} \oplus R\{-1\}$ given by $\alpha_C =\left (\,\mathcal{T}(\raisebox{-5pt}{\includegraphics[height=0.2in]{counit.pdf}}), \mathcal{T}(\frac{1}{2}\raisebox{-5pt}{\includegraphics[height=0.2in]{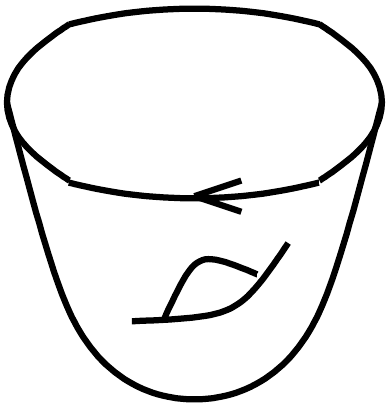}} + \frac{h}{2} \raisebox{-5pt}{\includegraphics[height=0.2in]{counit.pdf}}) \,\right)^t$ realizes the isomorphism $\mathcal{A}_C \cong R\{1\} \oplus R\{-1\}.$
  \end{proposition}
  
  \begin{proof}
First, we remark that $\alpha_C$ is degree-preserving. Consider the $R$-module map $\beta_C \co R\{1\} \oplus R\{-1\} \to \mathcal{A}_C$ given by $\beta_C = \left(\, \mathcal{T}(\frac{1}{2}\raisebox{-5pt}{\includegraphics[height=0.2in]{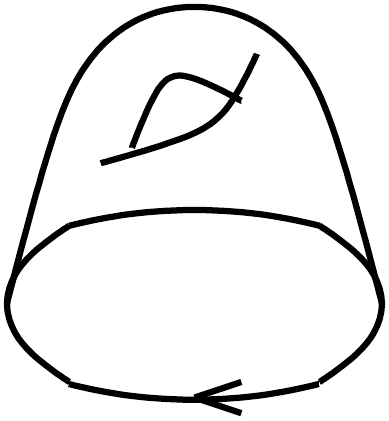}} - \frac{h}{2} \raisebox{-5pt}{\includegraphics[height=0.2in]{unit.pdf}}), \mathcal{T}(\raisebox{-5pt}{\includegraphics[height=0.2in]{unit.pdf}})\,\right)$ and notice that it is degree-preserving and satisfies $\beta_C \circ \alpha_C = \mathcal{T}\left(\,\raisebox{-8pt}{\includegraphics[height=0.3in]{identity_circle.pdf}}\,\right) = \id_{\mathcal{A}_C},$ and $\alpha_C \circ \beta_C = \id_{R\{1\} \oplus R\{-1\}}.$ Therefore, $\alpha_C$ and $\beta_C$ are mutually inverse isomorphisms in $\textbf{R-Mod}.$
  \end{proof}
  
    \begin{proposition}\label{prop:isom_web}
  The $R$-module map $\alpha_W \co \mathcal{A}_W \to R\{1\} \oplus R\{-1\}$ given by $\alpha_W = \left(\,\mathcal{T}(\raisebox{-5pt}{\includegraphics[height=0.2in]{singcounit.pdf}}), \mathcal{T}(\frac{1}{2}\raisebox{-8pt}{\includegraphics[height=0.25in]{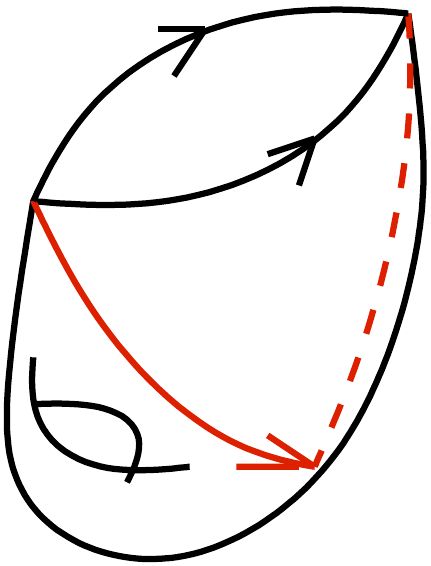}} + \frac{h}{2} \raisebox{-5pt}{\includegraphics[height=0.2in]{singcounit.pdf}}) \,\right)^t$ realizes the isomorphism $\mathcal{A}_W \cong R\{1\} \oplus R\{-1\}.$
  \end{proposition}

\begin{proof}
It is easy to see that $\alpha_W$ is degree-preserving. Consider the $R$-module map $\beta_W \co R\{1\} \oplus R\{-1\} \to \mathcal{A}_C$ given by $\beta_W = \left(\, \mathcal{T}( \frac{i}{2}\raisebox{-5pt}{\includegraphics[height=0.25in]{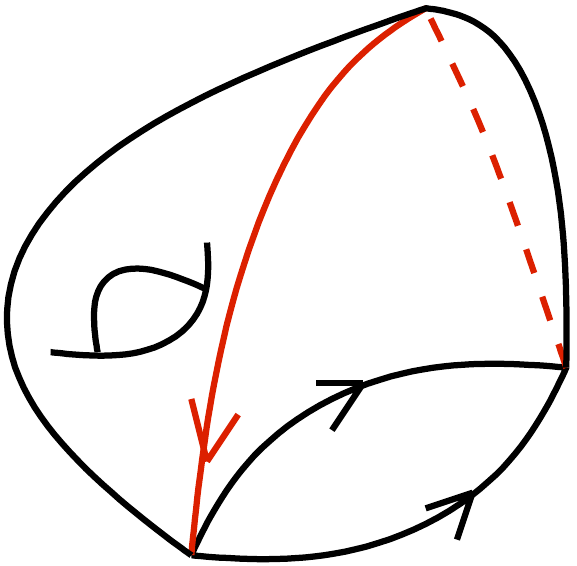}} - \frac{ih}{2} \raisebox{-5pt}{\includegraphics[height=0.2in]{singunit.pdf}}), \mathcal{T}(i \raisebox{-5pt}{\includegraphics[height=0.2in]{singunit.pdf}})\,\right).$ We have that $\beta_W \circ \alpha_W = \mathcal{T}\left(\,\raisebox{-8pt}{\includegraphics[height=0.3in]{identity_web.pdf}}\,\right) = \id_{\mathcal{A}_W},$ and $\alpha_W \circ \beta_W = \id_{R\{1\} \oplus R\{-1\}}.$ Therefore, $\alpha_W$ and $\beta_W$ are mutually inverse isomorphisms in $\textbf{R-Mod}.$
\end{proof}

\section{ A new link cohomology}\label{sec:link cohomology}

Given a plane diagram $D$ representing an oriented link $L$, we construct a cochain complex $\mathcal{C}(D)$ of graded modules over the commutative ring $R$ considered before, whose graded Euler characteristic is $P_2(L)$, the quantum $sl(2)$ polynomial of $L$. 

Let $I$ be the set of crossings of $D,$ and $n_{+}$ (respectively $n_{-}$) be the number of positive (respectively negative) crossings; let $n = \vert I \vert = n_+ + n_-.$ We associate to $D$ an $n$-dimensional cube $\overline{\mathcal{U}}$ described below, whose vertices and edges are objects and morphisms, respectively,  in the category $\textbf{eSing-2Cob}$. Vertices of the cube are in one-to-one correspondence with subsets of $I.$ 

We begin by associating to each crossing in $D$ either the oriented resolution or the singular resolution, as explained in Figure~\ref{fig:resolutions}. To $J \subset I$ we associate a web $\Gamma_J,$ namely the $J$-resolution of $D,$ where the crossing $k$ receives its $1$-resolution if $k \in J,$ otherwise it receives the $0$-resolution. Each resolution $\Gamma_J$ is a disjoint union of oriented circles and closed webs with an even number of vertices. Notice that these are exactly the webs decorating the vertices of the $n$-dimensional cube $\mathcal{U}$ in the $sl(2)$ foam cohomology construction.

 \begin{figure}[ht]
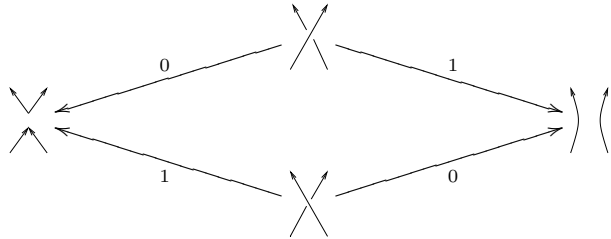

$$\xymatrix@C=30mm@R=0.1mm{
  & \raisebox{-12pt} {\includegraphics[height=0.35in]{poscrossing.pdf}} \ar[ld]_0\ar[rd]^1& \\
\raisebox{-10pt}{\includegraphics[height=0.35in]{singres.pdf}} & & 
\raisebox{-10pt}{\includegraphics[height=0.35in]{orienres.pdf}} \\
  & \raisebox{-8pt}{\includegraphics[height=0.35in]{negcrossing.pdf}} \ar[ul]^1 \ar[ur]_0&
}$$
\caption{Resolutions}
\label{fig:resolutions}
\end{figure}

Next step consists in simplifying each web component of each resolution $\Gamma_J$ by removing pairs of adjacent vertices of the same type, thus `imitating' the web skein relations~\eqref{skein_rembv}, until each web is replaced by a bi-web. Each simplified resolution, call it $\overline{\Gamma}_J$, is then a disjoint union of oriented circles and bi-webs, and may be regarded as an object in the category $\textbf{eSing-2Cob}$. 
  
 Number the components of each $\overline{\Gamma}_J$ by $1, \dots, k_J$, with $k_J \in \mathbb{N},$ and define a sequence $m^{(J)} = (m_1^{(J)}, \dots, m_{k_J}^{(J)}),$ where $m_l^{(J)} = 0_{-}$ or $m_l^{(J)} = 0_{+}$ if the $l$-th component is a negatively or positively oriented circle, respectively, and $m_l^{(J)} = 1$ if the $l$-th component is a bi-web.    
 
 For every $k \in I$ and $J \subset I \backslash \{k\},$ the components of the resolutions $\overline{\Gamma}_J$ and $\overline{\Gamma}_{J \cup k}$ differ by either an oriented circle or a bi-web. Consider the singular 2-cobordism $\overline{S}_{(J, k)} \co \overline{ \Gamma}_J \to \overline {\Gamma}_{J \cup k}$ which is a cylinder over $\overline{\Gamma}_J$ except for a small neighborhood of the crossing $k,$ where it looks like one of the following singular cobordisms:
 
  \begin{equation}\label{eq:cobordisms in U}
   \psset{xunit=.22cm,yunit=.22cm}
 \quad
 \begin{pspicture}(5,6)
   \rput(3.6, 5.7){\includegraphics[height=0.35in]{identity_circle.pdf}}
 \rput(0.2, 5.7){\includegraphics[height=0.35in]{morphism_g.pdf}} 
 \rput(1.85, -0.5){\includegraphics[height=0.35in]{zipper.pdf}}
 \rput(1.9, 2.6){\includegraphics[height=0.35in]{mult.pdf}}
 \end{pspicture} \quad 
 \begin{pspicture}(6,6)
   \rput(0.2, 5.7){\includegraphics[height=0.35in]{identity_circle.pdf}}
 \rput(3.6, 5.7){\includegraphics[height=0.35in]{morphism_g.pdf}} 
 \rput(1.85, - 0.5){\includegraphics[height=0.35in]{zipper.pdf}}
 \rput(1.9, 2.6){\includegraphics[height=0.35in]{mult.pdf}}
 \end{pspicture} \quad 
  \begin{pspicture}(6,6)
  \rput(3.6, 5.8){\includegraphics[height=0.35in]{zipper.pdf}}
 \rput(0.2, 5.8){\includegraphics[height=0.35in]{identity_web.pdf}} 
 \rput(1.9, 2.6){\includegraphics[height=0.35in]{singmult.pdf}}
 \end{pspicture} \quad
  \begin{pspicture}(6,6)
  \rput(3.6, 5.8){\includegraphics[height=0.35in]{identity_web.pdf}}
 \rput(0.2, 5.8){\includegraphics[height=0.35in]{zipper.pdf}} 
 \rput(1.9, 2.6){\includegraphics[height=0.35in]{singmult.pdf}}
 \end{pspicture} \quad
  \begin{pspicture}(6,6)
    \rput(3.6, 6.8){\includegraphics[height=0.35in]{identity_web.pdf}}
 \rput(0.2, 6.6){\includegraphics[height=0.35in]{morphism_g.pdf}} 
  \rput(3.6, 3.6){\includegraphics[height=0.35in]{identity_web.pdf}}
 \rput(0.2, 3.6){\includegraphics[height=0.35in]{zipper.pdf}} 
 \rput(1.9, 0.3){\includegraphics[height=0.35in]{singmult.pdf}}
 \end{pspicture} \quad
  \begin{pspicture}(5,6)
    \rput(3.6, 6.6){\includegraphics[height=0.35in]{morphism_g.pdf}}
 \rput(0.2, 6.8){\includegraphics[height=0.35in]{identity_web.pdf}} 
  \rput(3.6, 3.6){\includegraphics[height=0.35in]{zipper.pdf}}
 \rput(0.2, 3.6){\includegraphics[height=0.35in]{identity_web.pdf}} 
 \rput(1.9, 0.3){\includegraphics[height=0.35in]{singmult.pdf}}
 \end{pspicture} \quad
\begin{pspicture}(5,5)
 \rput(3.9, 2.6){\includegraphics[height=0.35in]{singmult.pdf}}
 \end{pspicture} \quad
\end{equation}
 
  \begin{equation}\label{eq:cobordisms in U-bis}
   \psset{xunit=.22cm,yunit=.22cm}
 \quad
\begin{pspicture}(5,6)
-\,\,\rput(4, -0.4){\includegraphics[height=0.35in]{identity_circle.pdf}}
 \rput(0.8, -0.4){\includegraphics[height=0.35in]{morphism_f.pdf}} 
 \rput(2.4, 5.7){\includegraphics[height=0.35in]{cozipper.pdf}}
\rput(2.4, 2.6){\includegraphics[height=0.35in]{comult.pdf}}
 \end{pspicture} \quad 
 \begin{pspicture}(6,6)
-\,\,\rput(4, -0.4){\includegraphics[height=0.35in]{morphism_f.pdf}}
 \rput(0.8, -0.4){\includegraphics[height=0.35in]{identity_circle.pdf}} 
 \rput(2.4, 5.7){\includegraphics[height=0.35in]{cozipper.pdf}}
\rput(2.4, 2.6){\includegraphics[height=0.35in]{comult.pdf}}
 \end{pspicture} \quad 
  \begin{pspicture}(6,6)
 i \,\, \rput(4.1, -0.4){\includegraphics[height=0.35in]{cozipper.pdf}}
 \rput(0.7, -0.4){\includegraphics[height=0.35in]{identity_web.pdf}} 
 \rput(2.4, 2.8){\includegraphics[height=0.35in]{singcomult.pdf}}
 \end{pspicture} \quad
  \begin{pspicture}(6,6)
 i\,\, \rput(4.1, -0.4){\includegraphics[height=0.35in]{identity_web.pdf}}
 \rput(0.7, -0.4){\includegraphics[height=0.35in]{cozipper.pdf}} 
 \rput(2.4, 2.8){\includegraphics[height=0.35in]{singcomult.pdf}}
 \end{pspicture} \quad
  \begin{pspicture}(6,6)
 i\,\, \rput(4.1, 2.5){\includegraphics[height=0.35in]{identity_web.pdf}}
 \rput(0.7, 2.5){\includegraphics[height=0.35in]{cozipper.pdf}} 
\rput(2.4, 5.7){\includegraphics[height=0.35in]{singcomult.pdf}}
 \rput(4.1, -0.7){\includegraphics[height=0.35in]{identity_web.pdf}}
 \rput(0.7, -0.5){\includegraphics[height=0.35in]{morphism_f.pdf}} 
 \end{pspicture} \quad
   \begin{pspicture}(5,6)
 i\,\, \rput(4.1, 2.5){\includegraphics[height=0.35in]{cozipper.pdf}}
 \rput(0.7, 2.5){\includegraphics[height=0.35in]{identity_web.pdf}} 
\rput(2.4, 5.7){\includegraphics[height=0.35in]{singcomult.pdf}}
 \rput(4.1, -0.5){\includegraphics[height=0.35in]{morphism_f.pdf}}
 \rput(0.7,- 0.7){\includegraphics[height=0.35in]{identity_web.pdf}} 
 \end{pspicture} \quad
  \begin{pspicture}(5,5)
  \rput(2.9, 0.6){\includegraphics[height=0.35in]{singcomult.pdf}}
 \end{pspicture} 
\end{equation}
\vspace{0.1cm}

We define the $n$-dimensional cube $\overline{\mathcal{U}}$ to have as vertices the sequences $m^{(J)}\{2n_+ -n_- -|J|\},$ and as edges the singular 2-cobordisms $\overline{S}_{(J, k)}$ regarded as morphisms in the category $ \textbf{eSing-2Cob}$. Here $\{m \}$ is the grading shift operator that lowers the grading by $m$. 

We remark that the resolution-simplification step used above suggests that we need to impose the relations \eqref{imposed-relation1} and~ \eqref{imposed-relation2} right at the `topological world', before we apply any possible TQFT. (Compare again the web skein relations~\eqref{skein_rembv} with the isomorphisms~\eqref{isomorphisms}, and then with relations \eqref{imposed-relation1}-\eqref{imposed-relation2}.) Hence the cube $\overline{\mathcal{U}}$ is regarded as lying in the category $ \textbf{eSing-2Cob}$, rather than $\textbf{Sing-2Cob}$.

The morphism decorating an edge of the cube $\overline{\mathcal{U}}$ is a disjoint union of a finite number of cylinders over an oriented circle and/or over a bi-web with a saddle cobordism $S$ that looks like one of those in~\eqref{eq:cobordisms in U} or~\eqref{eq:cobordisms in U-bis}; such a saddle cobordism $S$ is a composition of generating morphisms of \textbf{eSing-2Cob}. Relations~\eqref{imposed-relation1} and~\eqref{imposed-relation2} together with the fact that the bi-web forms a commutative Frobenius algebra object in \textbf{eSing-2Cob} imply that each face of $\overline{\mathcal{U}}$ commutes.

Now that we have the topological picture, namely the cube $\overline{\mathcal{U}},$ we are ready to apply the degree-preserving TQFT $\mathcal{T} \co \textbf{eSing-2Cob} \to \textbf{R-Mod}$ given in Section~\ref{sec:TQFT}, and turn the commutative cube $\overline{\mathcal{U}}$ in $\textbf{eSing-2Cob}$ into a commutative cube $\mathcal{T}(\overline{\mathcal{U}})$ in $\textbf{R-Mod}.$ The degree shift of each vertex assures that each edge of $\mathcal{T}(\overline{\mathcal{U}})$ is a grading-preserving homomorphism. 

We have seen that the functor $\mathcal{T}$ associates $\mathcal{A}_C$ to both, positively and negatively oriented circles. However, it will be useful to work with different bases for  $\mathcal{A}_C,$ namely
\[\mathcal{T}(\raisebox{-3pt}{\includegraphics[height=0.15in]{circle.pdf}} ) =  <1, X>_R \quad \text{and} \quad \mathcal{T}(\raisebox{-3pt}{\includegraphics[height=0.15in]{loop.pdf}} ) =  <1, h -X>_R.\]
Under this convention for $\mathcal{A}_C,$ the homomorphisms decorating the commutative cube $\mathcal{T}(\overline{\mathcal{U}})$ are given by the following rules (we omit here $m_W$ and $\Delta_W,$ whose rules are already clear):

\begin{eqnarray*}\mathcal{T}\left (\psset{xunit=.22cm,yunit=.22cm}
 \begin{pspicture}(4,4.5)
  \rput(3.1, 2.75){\includegraphics[height=0.25in]{identity_circle.pdf}}
 \rput(0.7, 2.75){\includegraphics[height=0.25in]{morphism_g.pdf}} 
 \rput(1.85, -1.6){\includegraphics[height=0.25in]{zipper.pdf}}
 \rput(1.9, 0.6){\includegraphics[height=0.25in]{mult.pdf}}
 \end{pspicture} \right ) &=& z_1 \circ m_C \circ (g \otimes \id_{\mathcal{A}_C})
  : \begin{cases}
1 \otimes 1 \to -i\\ 1 \otimes X \to -iX\\ (h - X) \otimes 1 \to -iX\\ (h - X) \otimes X \to -i(hX + a)
 \end{cases}\\
\mathcal{T}\left (\psset{xunit=.22cm,yunit=.22cm}
 \begin{pspicture}(4,4.5)
  \rput(0.7, 2.75){\includegraphics[height=0.25in]{identity_circle.pdf}}
 \rput(3.1, 2.75){\includegraphics[height=0.25in]{morphism_g.pdf}} 
 \rput(1.85, -1.6){\includegraphics[height=0.25in]{zipper.pdf}}
 \rput(1.9, 0.6){\includegraphics[height=0.25in]{mult.pdf}}
 \end{pspicture} \right ) &=& z_1 \circ m_C \circ (\id_{\mathcal{A}_C} \otimes \, g)
 : \begin{cases}
1 \otimes 1 \to -i\\ 1 \otimes (h -X) \to -iX\\ X \otimes 1 \to -iX\\ X \otimes (h -X) \to -i(hX + a)
 \end{cases} \end{eqnarray*}
 
 \begin{eqnarray*}
\mathcal{T}\left (\psset{xunit=.22cm,yunit=.22cm}
\begin{pspicture}(4,3)
 \rput(3.1, 1.4){\includegraphics[height=0.25in]{identity_web.pdf}}
 \rput(0.7, 1.4){\includegraphics[height=0.25in]{zipper.pdf}} 
 \rput(1.9, -0.9){\includegraphics[height=0.25in]{singmult.pdf}}
 \end{pspicture} \right) &=& m_W \circ (z_1 \otimes \id_{\mathcal{A}_W})\\ &=& m_W \circ (\id_{\mathcal{A}_W} \otimes \, z_1) =
 \mathcal{T} \left (\psset{xunit=.22cm,yunit=.22cm}
 \begin{pspicture}(4,3)
  \rput(3.1, 1.4){\includegraphics[height=0.25in]{zipper.pdf}}
 \rput(0.7, 1.4){\includegraphics[height=0.25in]{identity_web.pdf}} 
 \rput(1.9, -0.9){\includegraphics[height=0.25in]{singmult.pdf}}
 \end{pspicture}
 \right )
 : \begin{cases}
1 \otimes 1 \to 1\\ 1 \otimes X \to X\\ X \otimes 1 \to X\\ X \otimes X \to  hX + a
 \end{cases}\end{eqnarray*}
 \begin{eqnarray*}
 \mathcal{T}\left (\psset{xunit=.22cm,yunit=.22cm}
  \begin{pspicture}(4,5)
    \rput(3.1, 2.9){\includegraphics[height=0.25in]{identity_web.pdf}}
 \rput(0.7, 2.8){\includegraphics[height=0.25in]{morphism_g.pdf}} 
  \rput(3.1, 0.6){\includegraphics[height=0.25in]{identity_web.pdf}}
 \rput(0.7, 0.6){\includegraphics[height=0.25in]{zipper.pdf}} 
 \rput(1.9, -1.7){\includegraphics[height=0.25in]{singmult.pdf}}
 \end{pspicture} \right ) &=& m_W \circ ((z_1 \circ g) \otimes \id_{\mathcal{A}_W}) 
  : \begin{cases}
1 \otimes 1 \to -i\\ 1 \otimes X \to -iX\\ (h - X) \otimes 1 \to -iX\\ (h - X) \otimes X \to -i(hX + a)
 \end{cases}\end{eqnarray*}
 \begin{eqnarray*}
  \mathcal{T}\left (\psset{xunit=.22cm,yunit=.22cm}
  \begin{pspicture}(4,5)
    \rput(0.7, 2.9){\includegraphics[height=0.25in]{identity_web.pdf}}
 \rput(3.1, 2.8){\includegraphics[height=0.25in]{morphism_g.pdf}} 
  \rput(0.7, 0.6){\includegraphics[height=0.25in]{identity_web.pdf}}
 \rput(3.1, 0.6){\includegraphics[height=0.25in]{zipper.pdf}} 
 \rput(1.9, -1.7){\includegraphics[height=0.25in]{singmult.pdf}}
 \end{pspicture} \right ) &=& m_W \circ ( \id_{\mathcal{A}_W} \otimes \,(z_1 \circ g)) 
: \begin{cases}
1 \otimes 1 \to -i\\ 1 \otimes (h - X) \to -iX\\ X \otimes 1 \to -iX\\  X \otimes (h - X)  \to -i(hX + a).
 \end{cases}\end{eqnarray*}
 
 For the comultiplication type maps we have:
 \begin{eqnarray*}
  \mathcal{T}\left (\psset{xunit=.22cm,yunit=.22cm}
  \begin{pspicture}(4.8,4)
-\rput(3.6, -1.6){\includegraphics[height=0.25in]{identity_circle.pdf}}
 \rput(1.2, -1.6){\includegraphics[height=0.25in]{morphism_f.pdf}} 
 \rput(2.4, 2.8){\includegraphics[height=0.25in]{cozipper.pdf}}
\rput(2.4, 0.6){\includegraphics[height=0.25in]{comult.pdf}}
 \end{pspicture} \right ) 
 &:& \begin{cases}
 1 \to - [1 \otimes X + (h - X) \otimes 1 - h 1 \otimes 1]\\ X \to -[ (h - X) \otimes X + a 1 \otimes 1]
 \end{cases}\\
  \mathcal{T}\left (\psset{xunit=.22cm,yunit=.22cm}
  \begin{pspicture}(4.8,4)
-\rput(1.2, -1.6){\includegraphics[height=0.25in]{identity_circle.pdf}}
 \rput(3.6, -1.6){\includegraphics[height=0.25in]{morphism_f.pdf}} 
 \rput(2.4, 2.8){\includegraphics[height=0.25in]{cozipper.pdf}}
\rput(2.4, 0.6){\includegraphics[height=0.25in]{comult.pdf}}
 \end{pspicture} \right ) 
 &:& \begin{cases}
 1 \to - [1 \otimes (h - X) + X \otimes 1 - h 1 \otimes 1]\\ X \to -[  X \otimes (h - X) + a 1 \otimes 1]
 \end{cases}\end{eqnarray*}
 \begin{eqnarray*}
  \mathcal{T}\left (\psset{xunit=.22cm,yunit=.22cm}
   \begin{pspicture}(5,4)
 i \rput(4.1, -0.7){\includegraphics[height=0.25in]{cozipper.pdf}}
 \rput(1.7, -0.7){\includegraphics[height=0.25in]{identity_web.pdf}} 
 \rput(2.9, 1.6){\includegraphics[height=0.25in]{singcomult.pdf}}
 \end{pspicture} \right ) =
 \mathcal{T}\left (\psset{xunit=.22cm,yunit=.22cm}
   \begin{pspicture}(5,4)
 i \rput(1.7, -0.7){\includegraphics[height=0.25in]{cozipper.pdf}}
 \rput(4.1, -0.7){\includegraphics[height=0.25in]{identity_web.pdf}} 
 \rput(2.9, 1.6){\includegraphics[height=0.25in]{singcomult.pdf}}
 \end{pspicture} \right )
: \begin{cases} 1 \to i(1 \otimes X + X \otimes 1 -h 1 \otimes 1) \\ X \to i(X \otimes X + a 1 \otimes 1)
 \end{cases}
 \end{eqnarray*}
 
 \begin{eqnarray*}
  \mathcal{T}\left (\psset{xunit=.22cm,yunit=.22cm}
   \begin{pspicture}(5,4)
 i \rput(4.1, 0.6){\includegraphics[height=0.25in]{identity_web.pdf}}
 \rput(1.7, 0.6){\includegraphics[height=0.25in]{cozipper.pdf}} 
\rput(2.9, 2.9){\includegraphics[height=0.25in]{singcomult.pdf}}
 \rput(4.1, -1.7){\includegraphics[height=0.25in]{identity_web.pdf}}
 \rput(1.7, -1.6){\includegraphics[height=0.25in]{morphism_f.pdf}} 
 \end{pspicture} \right )
 &:& \begin{cases} 1 \to - [1 \otimes X + (h -X) \otimes 1 - h 1 \otimes 1 ] \\ X \to - [(h -X) \otimes X + a 1 \otimes 1]
 \end{cases}\\
  \mathcal{T}\left (\psset{xunit=.22cm,yunit=.22cm}
   \begin{pspicture}(5,4)
 i \rput(1.7, 0.6){\includegraphics[height=0.25in]{identity_web.pdf}}
 \rput(4.1, 0.6){\includegraphics[height=0.25in]{cozipper.pdf}} 
\rput(2.9, 2.9){\includegraphics[height=0.25in]{singcomult.pdf}}
 \rput(1.7, -1.7){\includegraphics[height=0.25in]{identity_web.pdf}}
 \rput(4.1, -1.6){\includegraphics[height=0.25in]{morphism_f.pdf}} 
 \end{pspicture} \right )
 &:& \begin{cases} 1 \to - [1 \otimes (h - X) + X \otimes 1 - h 1 \otimes 1 ] \\ X \to - [X \otimes (h - X) + a 1 \otimes 1].
 \end{cases}
 \end{eqnarray*}
  
We add minus signs to some maps to make each square face of $\mathcal{T}(\overline{\mathcal{U}})$ anti-commutes; an intrinsic way to do this can be found in~\cite[Section 2.7]{BN1}.  

Finally, we form the total complex $\mathcal{C}(D)$ of the anti-commutative cube $\mathcal{T}(\overline{\mathcal{U}}),$ in such a way that its first non-zero term $\mathcal{T}(\Gamma_{\emptyset}) \{2n_{+} - n_{-} \}$ is placed in cohomological degree $-n_{+}.$  The complex $\mathcal{C}(D)$ is non-zero in cohomological degrees between $-n_{+}$ and $n_{-};$ the cochain object $\mathcal{C}^{r - n_{+}}(D)$ is the direct sum of all $R$-modules decorating the vertices of the cube $\mathcal{T}(\overline{U})$ with height  $r.$ The complex $\mathcal{C}(D)$ is well defined up to isomorphisms; specifically, it is independent of the ordering of crossings in $D$, and of the numbering of components of any resolution.

Let $\mbox{Kom}(\textbf{R-Mod})$ be the category of complexes over \textbf{R-Mod}, and denote by $K_{\textbf{R}}: = \mbox{Kom}_{/h}(\textbf{R-Mod})$ its homotopy subcategory. Two chain complexes in $\mbox{Kom}(\textbf{R-Mod})$ are homotopy equivalent if they are isomorphic in $K_{\textbf{R}}.$  
 
 \begin{theorem}
 If $D$ and $D'$ are oriented link diagrams that are related by a Reidemeister move, then the complexes $\mathcal{C}(D)$ and $\mathcal{C}(D')$ are homotopy equivalent. That is, $\mathcal{C}(D)$ and $\mathcal{C}(D')$ are isomorphic in the category $K_{\textbf{R}}.$
 \end{theorem}
 
 \begin{proof}
 \textit{Reidemeister\, I}.
Consider diagrams $D_1$ and $D'_1$ that differ only in a circular region as shown below.
$$D_1= \raisebox{-15pt}{\includegraphics[height=0.5in]{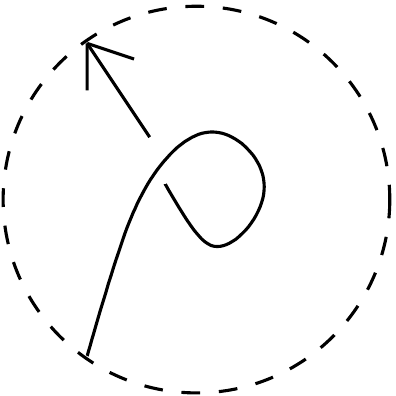}}\qquad
D'_1=\raisebox{-15pt}{\includegraphics[height=0.5in]{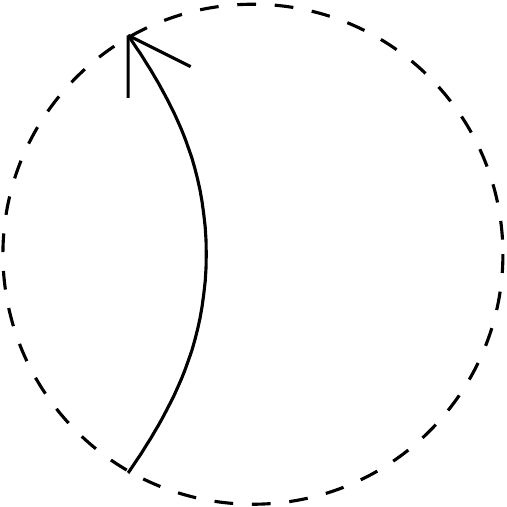}}$$

The chain complex $\mathcal{C}(D_1)$ associated to the diagram $D_1$ has the form

\[\mathcal{C}(D_1): \quad 0 \longrightarrow \left[ \mathcal{T} \left( \raisebox{-8pt}{\includegraphics[height=0.3 in]{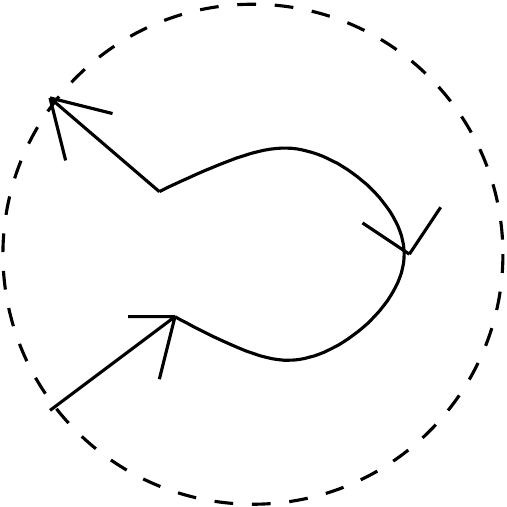}}\right) \{2\} \right] \stackrel{d}{\longrightarrow} \underline{ \left[\mathcal{T} \left (\raisebox{-8pt} {\includegraphics[height=0.3in]{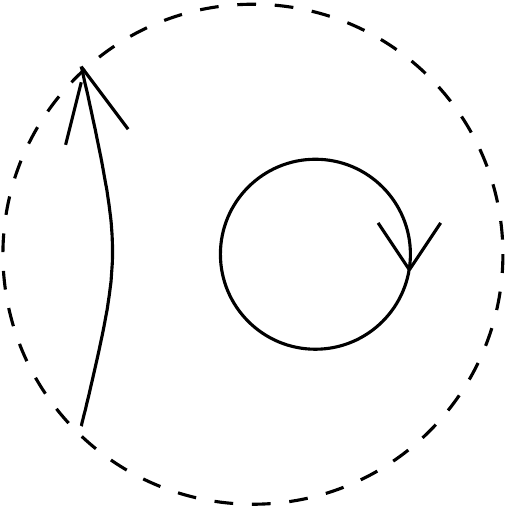}} \right ) \{1\} \right]} \longrightarrow 0, \]
where the underlined object is at cohomological degree 0.
Depending on the shape of $D_1$ (and $D'_1$) outside the circular region, the differential $d$ is the tensor product of 
\[ \mathcal{T} \left(- \psset{xunit=.22cm,yunit=.22cm} \begin{pspicture}(5,4)
\rput(3.6, -1.6){\includegraphics[height=0.25in]{identity_circle.pdf}}
 \rput(1.2, -1.6){\includegraphics[height=0.25in]{morphism_f.pdf}} 
 \rput(2.4, 2.8){\includegraphics[height=0.25in]{cozipper.pdf}}
\rput(2.4, 0.6){\includegraphics[height=0.25in]{comult.pdf}}
 \end{pspicture} \right)\quad \mbox{or} \quad  \mathcal{T} \left (i \psset{xunit=.2cm,yunit=.1cm}  \begin{pspicture}(5,8)
  \rput(3.8, -3.2){\includegraphics[height=0.25in]{cozipper.pdf}}
 \rput(1.2, -3.2){\includegraphics[height=0.25in]{identity_web.pdf}} 
 \rput(2.5, 6.9){\includegraphics[height=0.25in]{identity_web.pdf}}
 \rput(2.5, 1.9){\includegraphics[height=0.25in]{singcomult.pdf}}
 \end{pspicture} \right)\]
  with a finite number of identity maps $\id_{\mathcal{A}_C}$ and $\id_{\mathcal{A}_W}$
 (since diagrams $D_1$ and $D'_1$ are identical outside the circular region). For simplicity, we will omit these identity maps.
 The complex $\mathcal{C}(D_1)$ is isomorphic in $K_{\textbf{R}}$ to the complex
 
\[ 0 \longrightarrow \left[ \mathcal{T} \left( \raisebox{-8pt}{\includegraphics[height=0.3 in]{reid1-3.pdf}}\right) \{2\} \right ] \stackrel{
\left(\begin{array}{c}(\id_{\mathcal{A_C}} \otimes\, \alpha_{C,1}) \circ d \\ (\id_{\mathcal{A}_C} \otimes \,\alpha_{C,2}) \circ d \end{array} \right) }
{\longrightarrow} \underline{ \left[
\begin{array}{c}
\mathcal{T} \left (\raisebox{-8pt} {\includegraphics[height=0.3in]{reid1-1.pdf}} \right ) \{2\} \\
 \mathcal{T} \left (\raisebox{-8pt} {\includegraphics[height=0.3in]{reid1-1.pdf}} \right ) \{0\}
 \end{array} \right]}\longrightarrow 0, \]
where the maps $\alpha_{C,1}$ and $\alpha_{C,2}$ are the components of the isomorphism $\alpha_C$ given in Proposition~\ref{prop:isom_circle}.

This complex decomposes into the direct sum of the complexes
\begin{align*}
0 & \longrightarrow \left[ \mathcal{T} \left( \raisebox{-8pt}{\includegraphics[height=0.3 in]{reid1-3.pdf}}\right) \{2\} \right] \stackrel{(\id_{\mathcal{A}_C}\otimes\, \alpha_{C,1}) \circ d}{\longrightarrow}  \underline{\left[ \mathcal{T} \left (\raisebox{-8pt} {\includegraphics[height=0.3in]{reid1-1.pdf}} \right ) \{2\}\right]}
 \longrightarrow 0 \\
\mathcal{C}(D'_1) : \quad 0 & \longrightarrow  \underline{ \left[\mathcal{T} \left( \raisebox{-8pt} {\includegraphics[height=0.3in]{reid1-1.pdf}}\right) \right]} \longrightarrow 0.
\end{align*}
The morphism $(\id_{\mathcal{A}_C} \otimes \,\alpha_{C,1}) \circ d$ is equal to 
\[ - \mathcal{T} \left(\, \left(\raisebox{-8pt}{\includegraphics[height=0.25 in]{identity_circle.pdf}}   \,\,\raisebox{1pt}{\includegraphics[height=0.12 in]{counit.pdf}}\right)\, \circ \psset{xunit=.22cm,yunit=.22cm} \begin{pspicture}(5,4)
\rput(3.6, -1.6){\includegraphics[height=0.25in]{identity_circle.pdf}}
 \rput(1.2, -1.6){\includegraphics[height=0.25in]{morphism_f.pdf}} 
 \rput(2.4, 2.8){\includegraphics[height=0.25in]{cozipper.pdf}}
\rput(2.4, 0.6){\includegraphics[height=0.25in]{comult.pdf}}
 \end{pspicture} \right) = -  \mathcal{T} \left( \,
 \psset{xunit=.22cm,yunit=.22cm} \begin{pspicture}(2,3)
\rput(1, 1.5) {\includegraphics[height=0.25in]{cozipper.pdf}}
 \rput(1, -0.7){\includegraphics[height=0.25in]{morphism_f.pdf}} 
 \end{pspicture}
\, \right)\,\, \mbox{or} \,\, i \, \mathcal{T} \left(\,\left( \raisebox{-8pt}{\includegraphics[height=0.23 in]{identity_web.pdf}} \,\, \raisebox{-1pt}{\includegraphics[height=0.12 in]{counit.pdf}}\right)\, \circ \psset{xunit=.2cm,yunit=.1cm}  \begin{pspicture}(5,9)
  \rput(3.8, -3.2){\includegraphics[height=0.25in]{cozipper.pdf}}
 \rput(1.2, -3.2){\includegraphics[height=0.25in]{identity_web.pdf}} 
 \rput(2.5, 6.9){\includegraphics[height=0.25in]{identity_web.pdf}}
 \rput(2.5, 1.9){\includegraphics[height=0.25in]{singcomult.pdf}}
 \end{pspicture} \right) = i\, \mathcal{T}\left( \psset{xunit=.22cm,yunit=.22cm}
 \begin{pspicture}(2,2)
  \rput(1, 0.5){\includegraphics[height=0.25in]{identity_web.pdf}}
   \end{pspicture} 
   \right),\] and, in either case, it is an isomorphism. Thus the first complex in the direct sum above is acyclic, and $\mathcal{C}(D_1)$ and $\mathcal{C}(D'_1)$ are isomorphic in $K_{\textbf{R}}.$
   
The invariance under the type I move containing a negative crossing can be verified similarly.

\textit{Reidemeister\, IIa}.
Consider diagrams $D_{2a}$ and $D'_{2a}$ that differ in a circular region, as represented below.
$$D_{2a}=\raisebox{-18pt}{\includegraphics[height=0.5in]{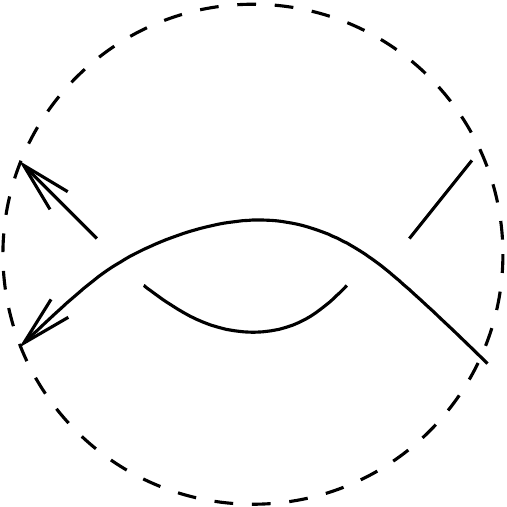}}\qquad
D'_{2a}=\raisebox{-18pt}{\includegraphics[height=0.5in]{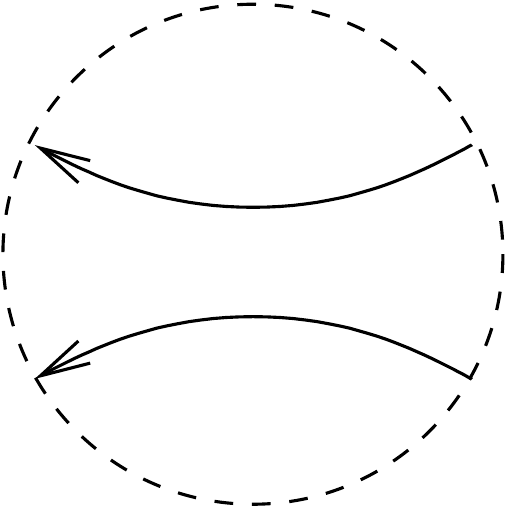}}\ $$
The chain complex $\mathcal{C}(D_{2a})$ associated to the diagram $D_{2a}$ has the form:
\[ 0 \rightarrow \left [ \mathcal{T} \left (\raisebox{-8pt} {\includegraphics[height=0.35in]{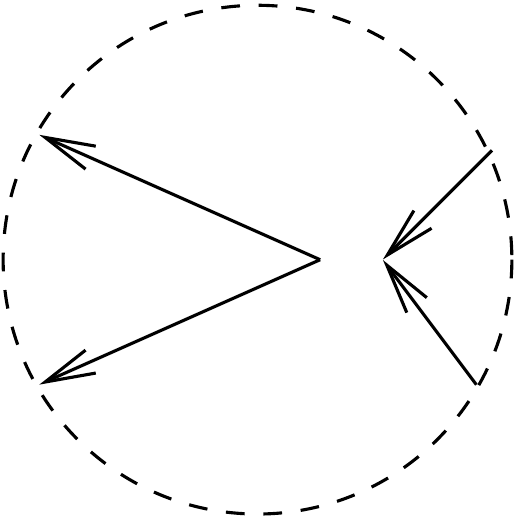}} \right )\{1\} \right] \stackrel{\left( \begin{array}{c} d^{-1}_1 \\ d^{-1}_2 \end{array} \right)}{\longrightarrow} \underline{\left[
\begin{array}{c}
\mathcal{T} \left (\raisebox{-8pt} {\includegraphics[height=0.35in]{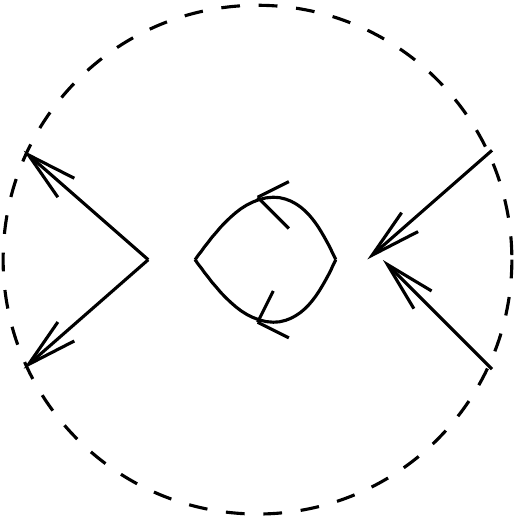}}\right ) \\
 \mathcal{T} \left (\raisebox{-8pt} {\includegraphics[height=0.35in]{twoarcs.pdf}}\right )
 \end{array} \right] }
\stackrel{\left(\begin{array}{c}d^0_1\\ d^0_2
\end{array}\right)^t} {\longrightarrow} \left [\mathcal{T} \left( \raisebox{-8pt} {\includegraphics[height=0.35in]{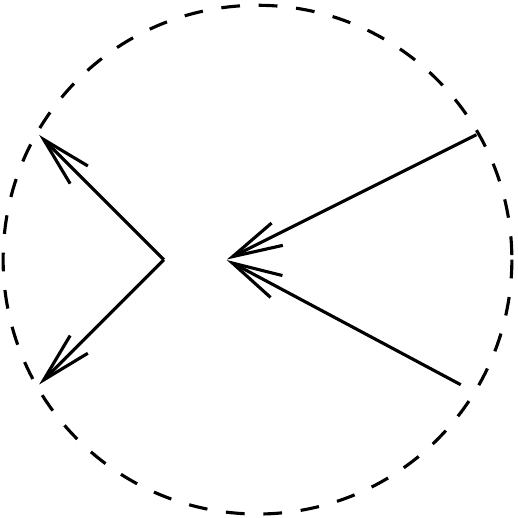}}\right )\{-1\} \right ] \rightarrow 0.\]
Using the result of Proposition~\ref{prop:isom_web}, this complex is isomorphic to the complex:
 
 \[0 \longrightarrow \left [ \mathcal{T} \left (\raisebox{-8pt} {\includegraphics[height=0.35in]{reid2a-1.pdf}} \right )\{1\} \right] \stackrel{\left( \begin{array}{c}(\id_{\mathcal{A}_W}\otimes\,\alpha_{W,1}) \circ d^{-1}_1 \\(\id_{\mathcal{A}_W}\otimes\,\alpha_{W,2}) \circ d^{-1}_1 \\ d^{-1}_2 \end{array} \right)}{\longrightarrow} \underline{\left[
\begin{array}{c}
\mathcal{T} \left (\raisebox{-8pt} {\includegraphics[height=0.35in]{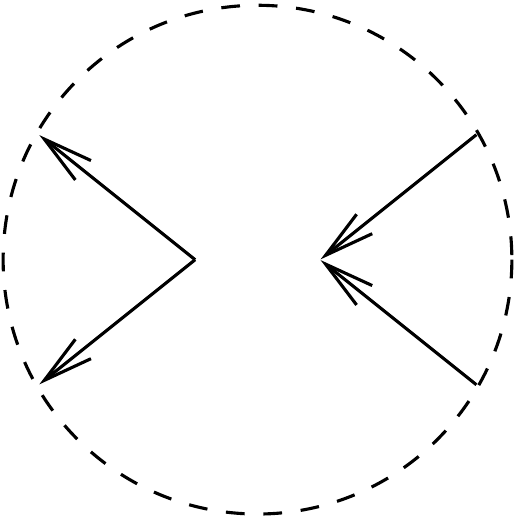}}\right )\{1\} \\
\mathcal{T} \left (\raisebox{-8pt} {\includegraphics[height=0.35in]{reid2a-4.pdf}}\right ) \{-1\}\\
 \mathcal{T} \left (\raisebox{-8pt} {\includegraphics[height=0.35in]{twoarcs.pdf}}\right )\{0\}
 \end{array} \right] }\longrightarrow\]
\[\stackrel{\left( \begin{array}{c} d^0_1 \circ (\id_{\mathcal{A}_W} \otimes\,\beta_{W,1})\\ d^0_1 \circ(\id_{\mathcal{A}_W} \otimes \, \beta_{W,2})\\ d^02 \end{array} \right)^t}{\longrightarrow} \left [\mathcal{T} \left( \raisebox{-8pt} {\includegraphics[height=0.35in]{reid2a-3.pdf}}\right )\{-1\} \right ] \longrightarrow 0,\]
 which decomposes into the following three complexes:
 \begin{align*}
 0& \longrightarrow \left [ \mathcal{T} \left (\raisebox{-8pt} {\includegraphics[height=0.35in]{reid2a-1.pdf}} \right )\{1\} \right]\stackrel{(\id_{\mathcal{A}_W}\otimes\,\alpha_{W,1}) \circ d^{-1}_1} {\longrightarrow}  \underline{\left[ \mathcal{T} \left (\raisebox{-8pt} {\includegraphics[height=0.35in]{reid2a-4.pdf}}\right )\{1\} \right]} \longrightarrow 0 \\
 0 & \longrightarrow  \underline{ \left[\mathcal{T} \left (\raisebox{-8pt} {\includegraphics[height=0.35in]{reid2a-4.pdf}}\right )\{-1\}  \right]} \stackrel{d^{0}_1\circ (\id_{\mathcal{A}_W} \otimes \beta_{W,2})} {\longrightarrow} \left [\mathcal{T} \left( \raisebox{-8pt} {\includegraphics[height=0.35in]{reid2a-3.pdf}}\right )\{-1\} \right ] \longrightarrow 0\\
\mathcal{C}(D'_{2a}): \quad 0& \longrightarrow  \underline{ \left[\mathcal{T} \left (\raisebox{-8pt} {\includegraphics[height=0.35in]{twoarcs.pdf}}\right )\right]} \longrightarrow 0,
 \end{align*}
 where $\alpha_{W,1}, \alpha_{W,2}$ and $\beta_{W,1}, \beta_{W,2}$ are the components of the isomorphisms $\alpha_W$ and $\beta_W$ appearing in Proposition~\ref{prop:isom_web}.
We have that:
 
\[d^{-1}_1 = \mathcal{T}\left( \raisebox{-5pt}{\includegraphics[height=0.25in]{singcomult.pdf}}\right)\quad  \mbox{and} \quad
 d^0_1 = \mathcal{T} \left( \,
 \raisebox{-6pt}{\includegraphics[height=0.25in]{singmult.pdf}}\,
 \right),\]
 \[ (\id_{\mathcal{A}_W}\otimes\,\alpha_{W,1}) \circ d^{-1}_1 = 
 \mathcal{T} \left(\left( \raisebox{-8pt}{\includegraphics[height=0.23 in]{identity_web.pdf}} \,\, \raisebox{-1pt}{\includegraphics[height=0.12 in]{singcounit.pdf}}\right)\, \circ
\raisebox{-8pt}{\includegraphics[height=0.25in]{singcomult.pdf}} 
   \right) = \mathcal{T}\left( 
  \raisebox{-5pt}{\includegraphics[height=0.25in]{identity_web.pdf}}
   \right),\]
    \[d^{0}_1\circ (\id_{\mathcal{A}_W} \otimes \beta_{W,2}) = \mathcal{T} \left(  \raisebox{-5pt}{\includegraphics[height=0.25in]{singmult.pdf}}
\circ \left(i\, \raisebox{-2pt}{\includegraphics[height=0.23 in]{identity_web.pdf}} \,\, \raisebox{-2pt}{\includegraphics[height=0.12 in]{singunit.pdf}}\right)
     \right) = i\,
     \mathcal{T}\left(
  \raisebox{-4pt}{\includegraphics[height=0.25in]{identity_web.pdf}}
   \right).\]
Thus $(\id_{\mathcal{A}_W}\otimes\,\alpha_{W,1}) \circ d^{-1}_1$ and $d^0_1 \circ (\id_{\mathcal{A}_W}\otimes\, \beta_{W,2})$ are isomorphisms,  and consequently, the first two complexes above are acyclic. (Note that all differentials above are tensored with a finite number of identity maps $\id_{\mathcal{A}_C}$ and $\id_{\mathcal{A}_W}$.) Moreover, the last complex corresponds to the diagram $D'_{2a}.$ Therefore, $\mathcal{C}(D_{2a})$ and $\mathcal{C}(D'_{2a})$ are isomorphic in $K_{\textbf{R}}.$

 \textit{Reidemeister\, IIb}. Consider now diagrams $D_{2b}$ and $D'_{2b}$ that differ in a circular region as shown below.
$$D_{2b}=\raisebox{-18pt}{\includegraphics[height=0.5in]{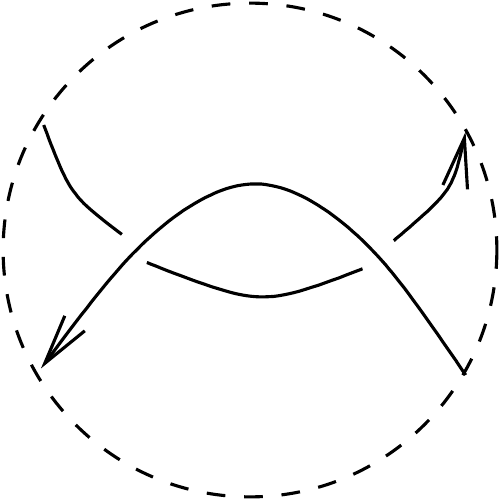}}\qquad
D'_{2b}=\raisebox{-18pt}{\includegraphics[height=0.5in]{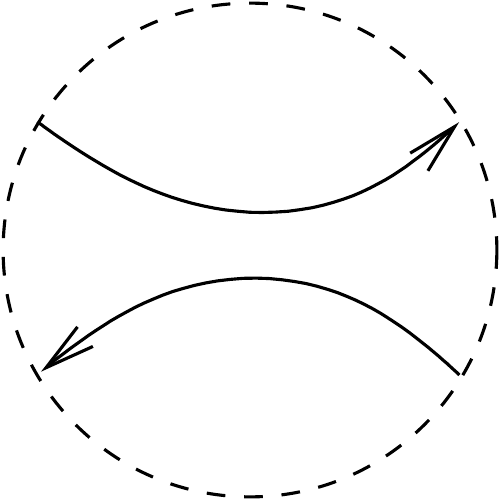}}\ $$
The chain complex associated to $D_{2b}$ has the form:
 \[0 \longrightarrow \left[ \mathcal{T} \left(\raisebox{-8pt} {\includegraphics[height=0.35in]{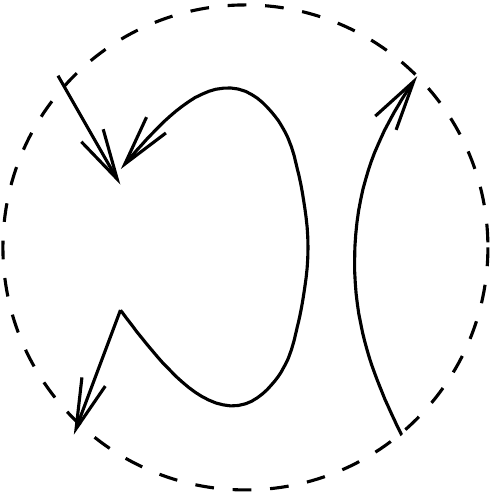}}\right)\{1\} \right] \longrightarrow \underline{ \left[ \begin{array}{c} \mathcal{T} \left(\raisebox{-8pt} {\includegraphics[height=0.35in]{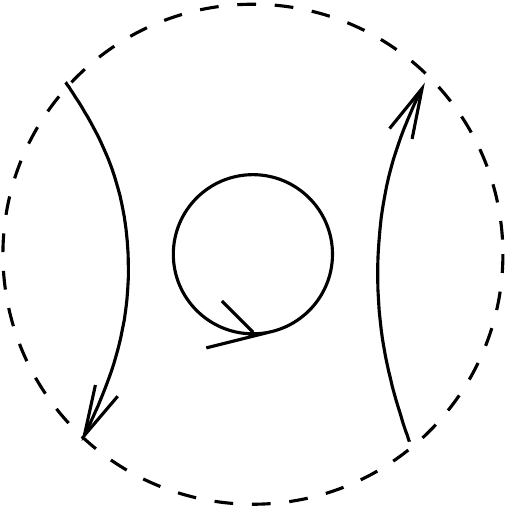}}\right ) \\ \mathcal{T} \left(\raisebox{-8pt} {\includegraphics[height=0.35in]{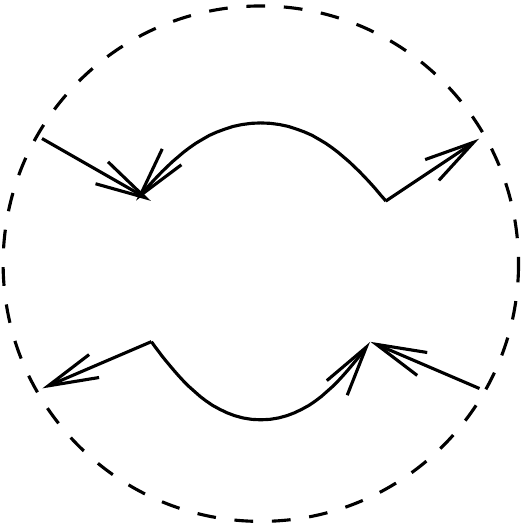}}\right) \end{array}\right]} \longrightarrow \left[ \mathcal{T} \left(\raisebox{-8pt} {\includegraphics[height=0.35in]{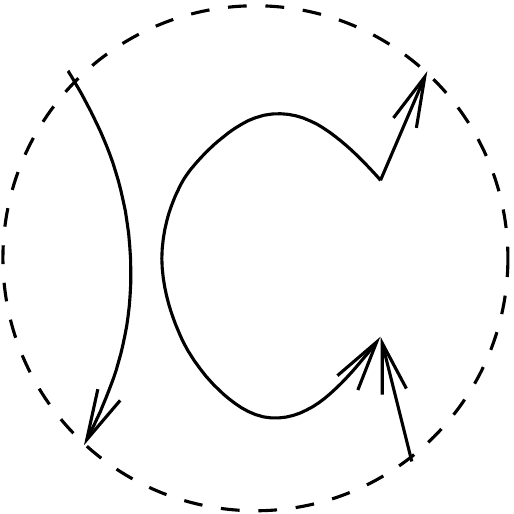}}\right)\{-1\}\right] \longrightarrow 0.\]
 As in the proof of invariance under Reidemeister IIa move, the complex $\mathcal{C}(D_{2b})$ is isomorphic (via Proposition~\ref{prop:isom_circle}) to the complex which is the direct sum of the following three complexes:

\begin{align*}
0 & \longrightarrow \left[ \mathcal{T} \left(\raisebox{-8pt} {\includegraphics[height=0.35in]{reid2b-1.pdf}}\right)\{1\} \right] \longrightarrow \underline{ \left[\mathcal{T} \left( \raisebox{-8pt} {\includegraphics[height=0.35in]{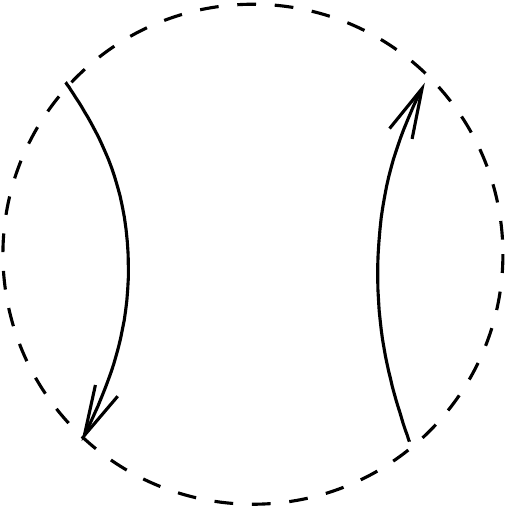}}\right) \{1\} \right]} \longrightarrow 0\\
0 &\longrightarrow \underline{ \left[\mathcal{T} \left( \raisebox{-8pt} {\includegraphics[height=0.35in]{reid2b-5.pdf}}\right) \{-1\} \right]} \longrightarrow \left[ \mathcal{T} \left(\raisebox{-8pt} {\includegraphics[height=0.35in]{reid2b-4.pdf}}\right)\{-1\}\right] \longrightarrow 0\\
0 & \longrightarrow  \underline{ \left[\mathcal{T} \left( \raisebox{-8pt} {\includegraphics[height=0.35in]{reid2b-3.pdf}}\right) \{-1\} \right]} \longrightarrow 0.
\end{align*}
The first two complexes are acyclic, and the last one is isomorphic to $\mathcal{C}(D'_{2b}).$ Therefore, $\mathcal{C}(D_{2b})$ and $\mathcal{C}(D'_{2b})$ are isomorphic in the category $K_{\textbf{R}}.$

\textit{Reidemeister\, III}. We consider diagrams $D_{3}$ and $D'_{3}$ that differ in a circular region as depicted below.
$$D_{3}=\raisebox{-17pt}{\includegraphics[height=0.5in]{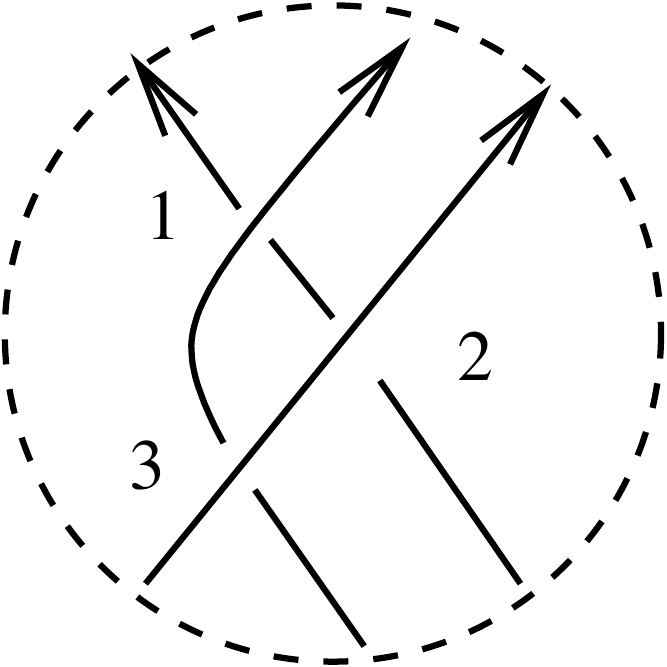}}\qquad
D'_{3}=\raisebox{-17pt}{\includegraphics[height=0.5in]{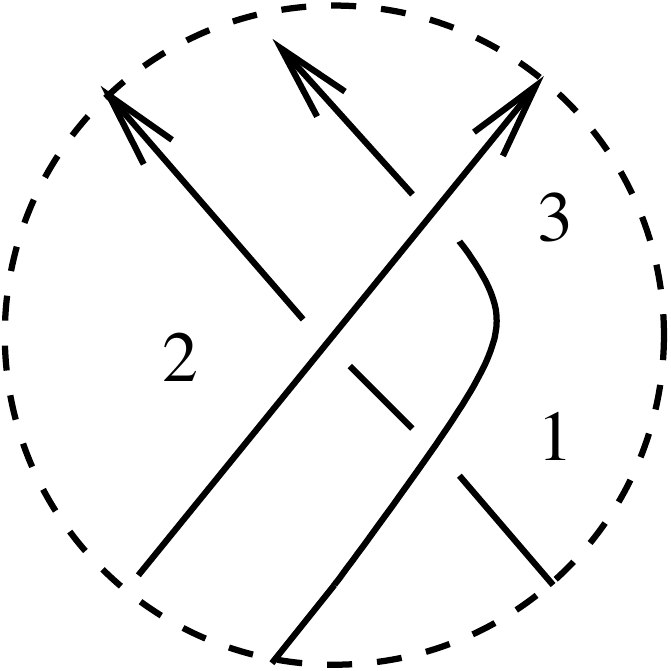}}\ $$

The cube of (simplified) resolutions corresponding to the diagram $D_3$ is given in Figure~\ref{fig:reidIII-left}, and that corresponding to $D'_3$ in Figure~\ref{fig:reidIII-right}. Note that the resolutions $\overline{\Gamma}_{000}$ and $\overline{\Gamma'}_{000}$ are obtained by simplifying (or removing adjacent pairs of vertices of the same type in) the original resolutions 

\[ \Gamma_{000} = \raisebox{-20pt}{\includegraphics[height=0.6in]{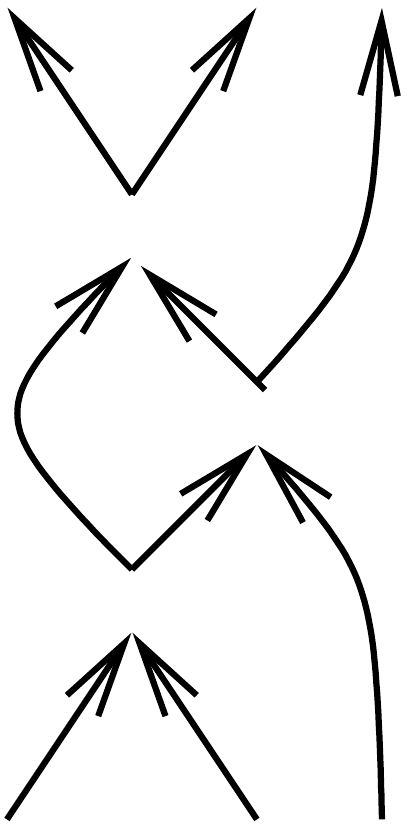}} \quad \text{and} \quad\Gamma'_{000} =  \raisebox{-20pt}{\includegraphics[height=0.6in]{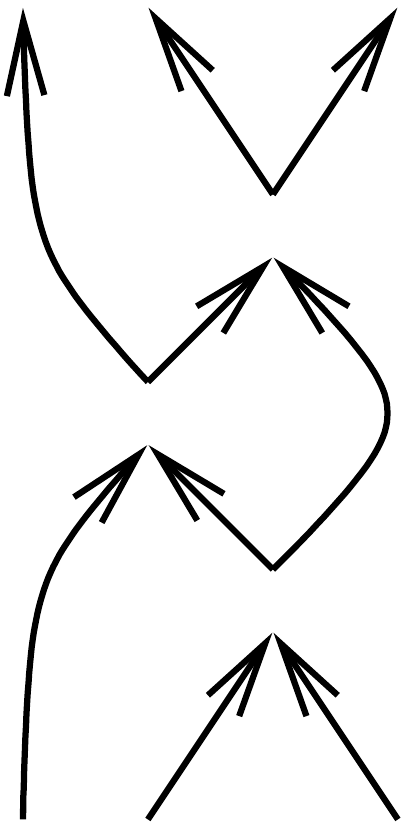}} \]
corresponding to $D_3$ and $D'_3,$ respectively. The reader will notice that we drew a circle at the tail of those edges of the cubes that received an additional minus sign (to make each square face anti-commutes).
\begin{figure}[ht]
\raisebox{-8pt}{\includegraphics[height=3.4in]{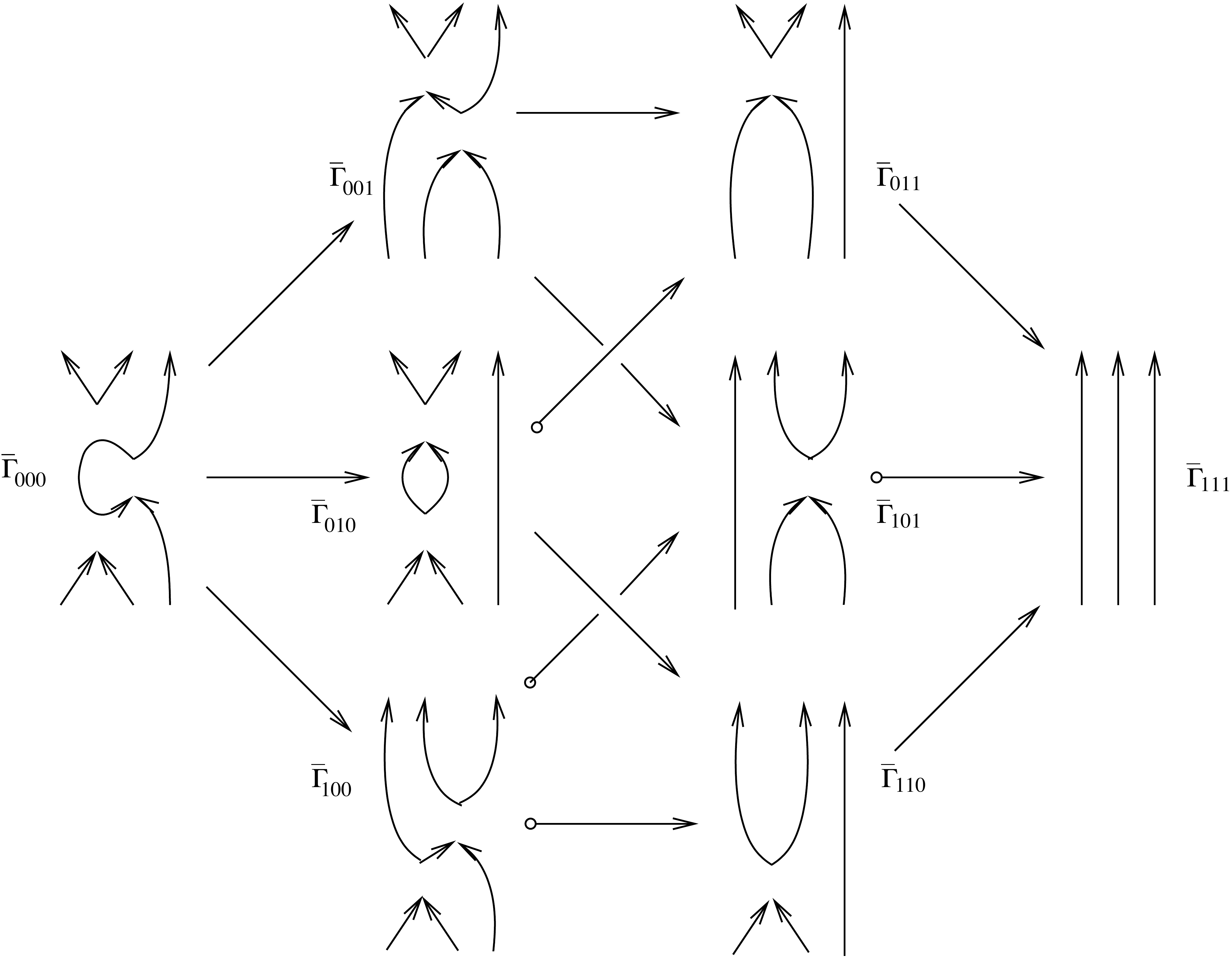}}
\caption{The cube of resolutions of $D_3$}\label{fig:reidIII-left}
\end{figure}

\begin{figure}[ht]
\raisebox{-8pt}{\includegraphics[height=3.4in]{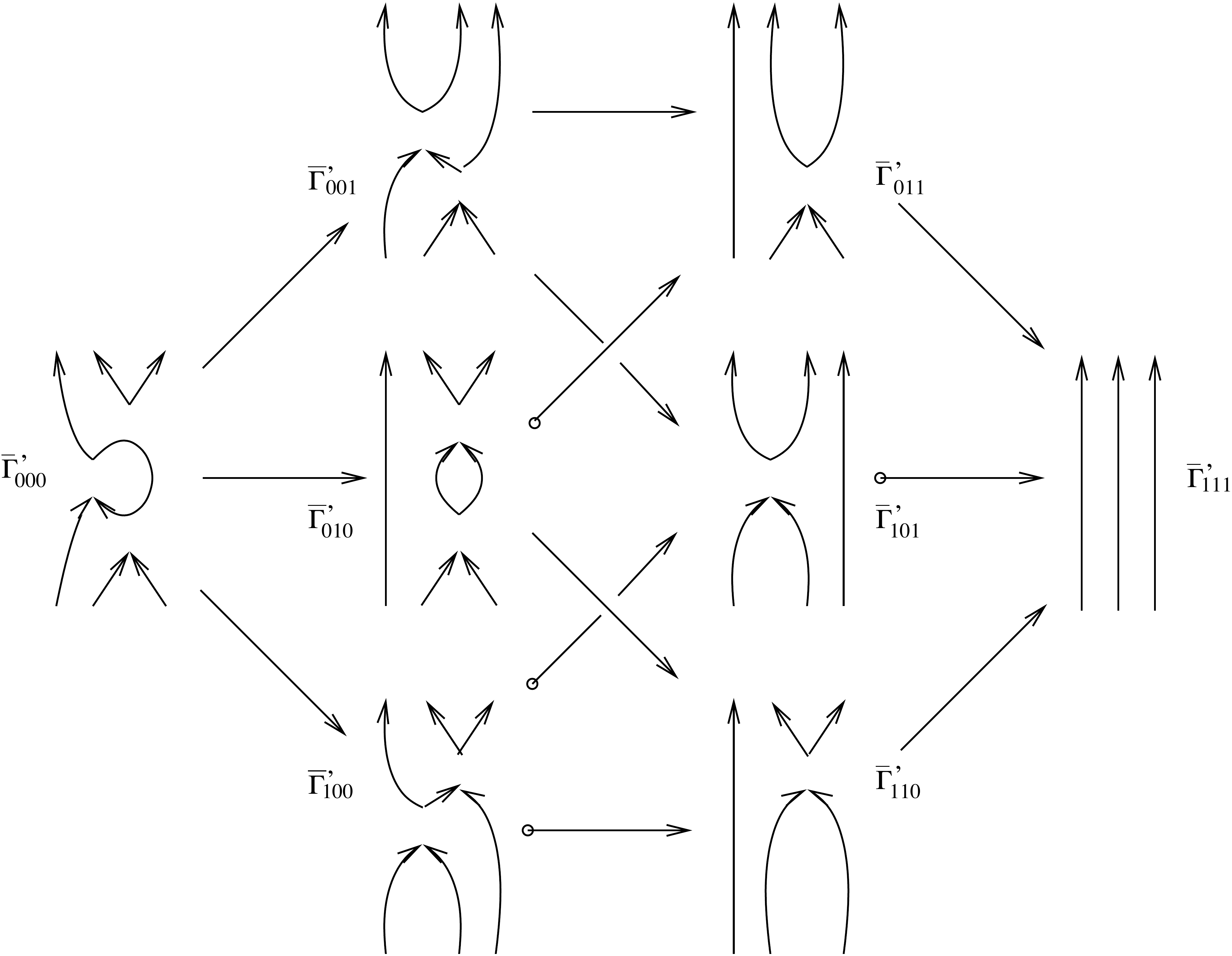}}
\caption{The cube of resolutions of $D'_3$}\label{fig:reidIII-right}
\end{figure}

The complex $C(D_3)$ has the form
\[0 \longrightarrow \left [ \mathcal{T}(\overline{\Gamma}_{000}) \{6\} \right ] \longrightarrow \left [ \begin{array}{c} \mathcal{T}(\overline{\Gamma}_{001})\{5\} \\ \mathcal{T}(\overline{\Gamma}_{010})\{5\} \\ \mathcal{T}(\overline{\Gamma}_{100})\{5\} \end{array}  \right ]\longrightarrow \left [ \begin{array}{c} \mathcal{T}(\overline{\Gamma}_{011})\{4\} \\ \mathcal{T}(\overline{\Gamma}_{101})\{4\} \\ \mathcal{T}(\overline{\Gamma}_{110})\{4\} \end{array}  \right] \longrightarrow \underline {\left [\mathcal{T}(\overline{\Gamma}_{111})\{3\}\right ]} \longrightarrow 0.\]

Using the isomorphism given in Proposition~\ref{prop:isom_web}, we have that
\[ \mathcal{T}(\overline{\Gamma}_{010}) \{5\} \cong \mathcal{T}(\overline{\Gamma}_{000}) \{6\} \oplus \mathcal{T}(\overline{\Gamma}_{110}) \{4\}. \]

Therefore the complex $C(D_3)$ is isomorphic to the complex which is the direct sum  of the  contractible complexes

\[0 \longrightarrow \mathcal{T}(\overline{\Gamma}_{000}) \{6\} \stackrel{\cong}{\longrightarrow} \mathcal{T}(\overline{\Gamma}_{000}) \{6\} \longrightarrow 0 \]
\[0 \longrightarrow \mathcal{T}(\overline{\Gamma}_{110}) \{4\} \stackrel{\cong}{\longrightarrow} \mathcal{T}(\overline{\Gamma}_{110}) \{4\} \longrightarrow 0 \]
and the complex 
\[\mathcal{C}: \quad 0 \longrightarrow \left [ \begin{array}{c} \mathcal{T}(\overline{\Gamma}_{001})\{5\} \\ \mathcal{T}(\overline{\Gamma}_{100})\{5\} \end{array}\right ]  \stackrel{d^{-2}}{\longrightarrow} \left [ \begin{array}{c} \mathcal{T}(\overline{\Gamma}_{011})\{4\} \\ \mathcal{T}(\overline{\Gamma}_{101})\{4\} \end{array}\right] \stackrel{d^{-1}}{\longrightarrow} \underline{ \left [\mathcal{T}(\overline{\Gamma}_{111})\{3\} \right] } \longrightarrow 0.  \]
In other words, $C(D_3)$ and $\mathcal{C}$ are isomorphic in $K_{\textbf{R}}.$

The complex $C(D'_3)$ has a similar form as $C(D_3)$, with the only difference that the resolutions $\overline{\Gamma}_{ijk}$ are replaced by $\overline{\Gamma}'_{ijk}$. Using the fact that, 
\[ \mathcal{T}(\overline{\Gamma}'_{010}) \{5\} \cong \mathcal{T}(\overline{\Gamma}'_{000}) \{6\} \oplus \mathcal{T}(\overline{\Gamma}'_{110}) \{4\},\]
we have that the complex $C(D'_3)$ is isomorphic to the complex 
\[\mathcal{C'}: \quad 0 \longrightarrow \left [ \begin{array}{c} \mathcal{T}(\overline{\Gamma}'_{001})\{5\} \\ \mathcal{T}(\overline{\Gamma}'_{100})\{5\} \end{array}\right ] \stackrel{d'^{-2}}{ \longrightarrow} \left [ \begin{array}{c} \mathcal{T}(\overline{\Gamma}'_{011})\{4\} \\ \mathcal{T}(\overline{\Gamma}'_{101})\{4\} \end{array}\right] \stackrel{d'^{-1}}{\longrightarrow} \underline{ \left[ \mathcal{T}(\overline{\Gamma}'_{111})\{3\} \right ]} \longrightarrow 0, \]
after stripping off the contractible direct summands 
\[0 \longrightarrow \mathcal{T}(\overline{\Gamma}'_{000}) \{6\} \stackrel{\cong}{\longrightarrow} \mathcal{T}(\overline{\Gamma}'_{000}) \{6\} \longrightarrow 0, \]
\[0 \longrightarrow \mathcal{T}(\overline{\Gamma}'_{110}) \{4\} \stackrel{\cong}{\longrightarrow} \mathcal{T}(\overline{\Gamma}'_{110}) \{4\} \longrightarrow 0. \]

Thus, $C(D')$ is isomorphic to $\mathcal{C'}$ in $K_{\textbf{R}}.$
 It remains to show that complexes $\mathcal{C}$ and $\mathcal{C'}$ are isomorphic. We give the resolutions contained in these complexes in Figure~\ref{fig:complex C} and Figure~\ref{fig:complex C'}.

\begin{figure}[ht]
\raisebox{-8pt}{\includegraphics[height=2in]{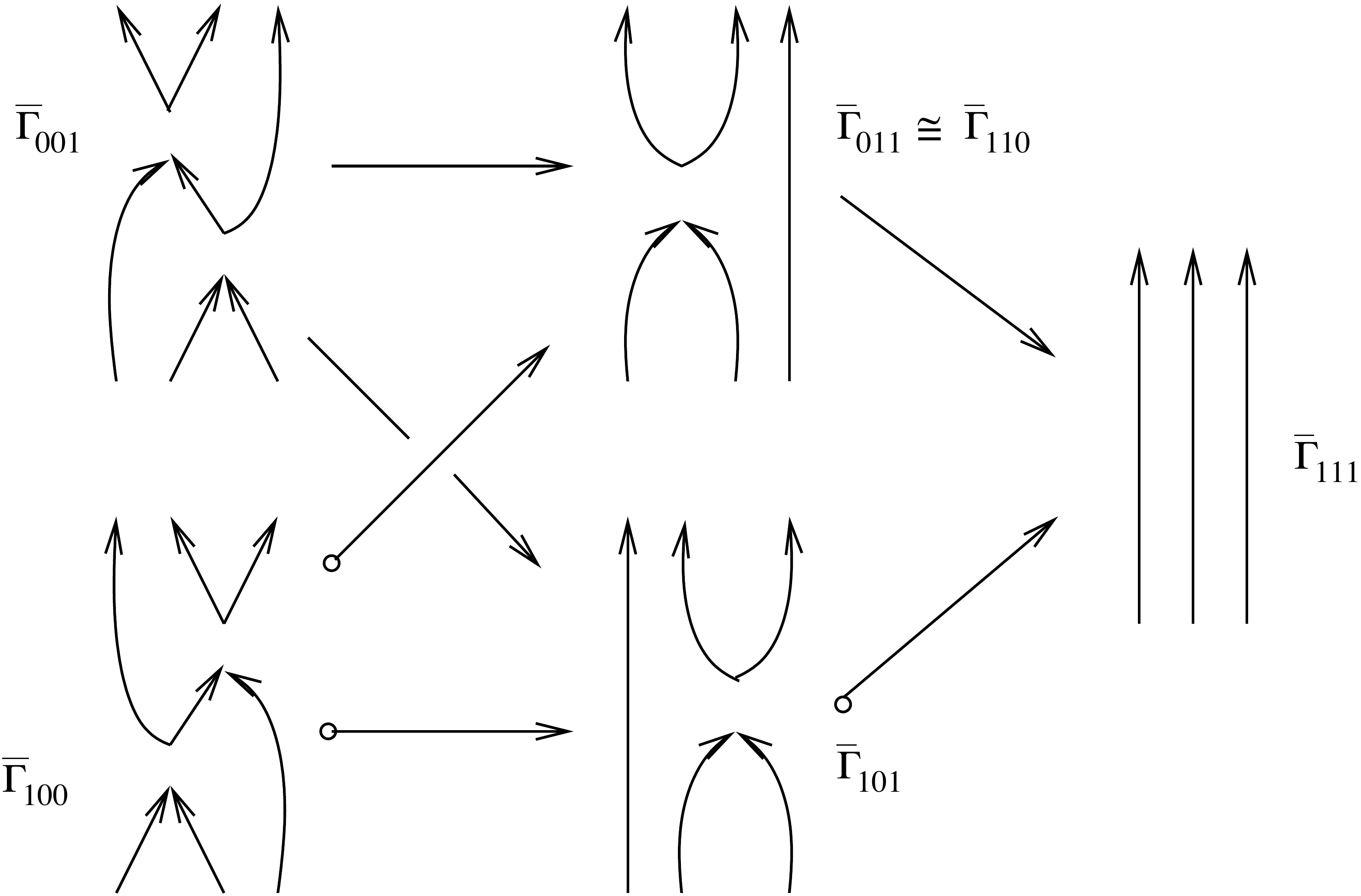}}
\caption{Resolutions forming the complex $C$}\label{fig:complex C}
\end{figure}

\begin{figure}[ht]
\raisebox{-8pt}{\includegraphics[height=2in]{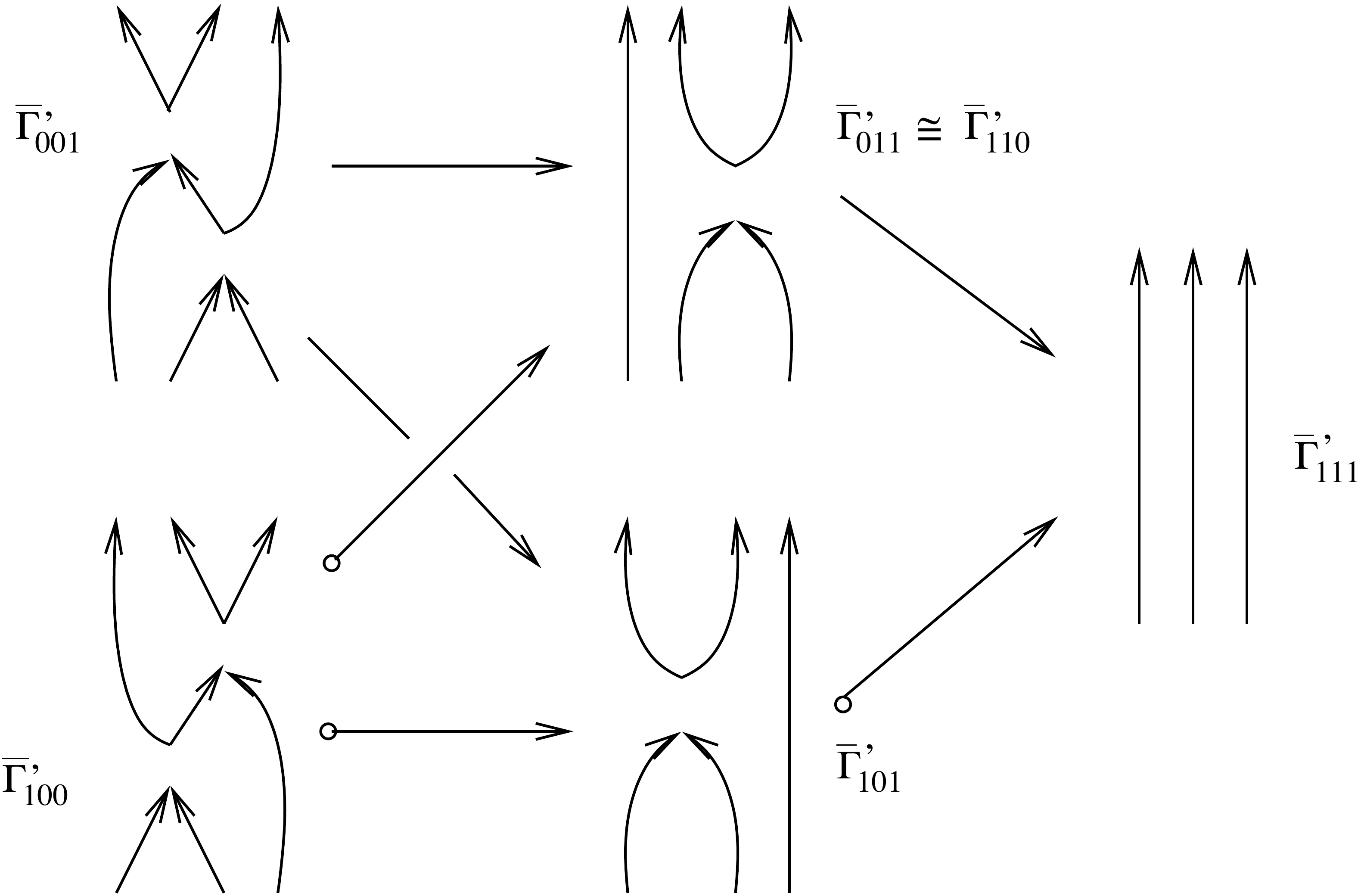}}
\caption{Resolutions forming the complex $C'$}\label{fig:complex C'}
\end{figure}

Denote by $d^{-2} = \left ( \begin{array}{cc} f_1 & f_3 \\ f_2 & f_4\end{array}\right)$ and $d^{-1} = (f_5, f_6)$ the differentials in $\mathcal{C}.$ Similarly, let $d'^{-2} = \left ( \begin{array}{cc} g_1 & g_3 \\ g_2 & g_4\end{array}\right)$ and $d'^{-1} = (g_5, g_6)$ be the differentials in $\mathcal{C'}.$ Then $f_1 = g_2, f_2 = g_1, f_3 = f_4$ and $f_4 = g_3.$ Moreover, $f_5 = - g_6$ and $f_6 = - g_5.$
 
It is an easy exercise to check that the maps $F \co \mathcal{C} \to \mathcal{C'}$ and $G \co \mathcal{C'} \to \mathcal{C}$ with components 
\[ F^{-2} = G^{-2} =  \left ( \begin{array}{cc} 1 & 0 \\ 0 & 1 \end{array} \right ), \quad F^{-1} = G^{-1} = \left ( \begin{array}{cc} 0 & 1 \\ 1 & 0 \end{array} \right ), \quad F^{0} = G^0 = - \id \]
are chain maps realizing the isomorphism (in $K_{\textbf{R}}$) between complexes $\mathcal{C}$ and $\mathcal{C'}.$ The invariance under the considered version of the type III move follows. 

Since the other oriented versions of the Reidemeister III move can be obtained from the type III move described above and the type II moves, the proof of invariance under the Reidemester moves is complete.
\end{proof}

\begin{corollary} Let $\mathcal{H}(D) = \oplus_{i,j \in \mathbb{Z}}\mathcal{H}^{i,j}(D)$ be the cohomology groups of $\mathcal{C}(D)$. Then the isomorphisms classes of the groups $\mathcal{H}^{i,j}(D)$ are invariants of $L$.
\end{corollary}
The following statement follows at once from construction.
\begin{proposition}
The graded Euler characteristic of $\mathcal{H}(L)$ is the quantum $sl(2)$ polynomial of $L$: 
\[P_2(L) = \sum_{i,j \in \mathbb{Z}}(-1)^iq^j\rk(\mathcal{H}^{i,j}(L)). \]
\end{proposition}

\noindent \textbf{Final conclusion and question.} Our construction together with Proposition~\ref{prop:functor T and local relations} imply that the link cohomology provided in this section is isomorphic to the universal dot-free $sl(2)$ foam cohomology. In particular, we gave a description of the universal $sl(2)$ foam cohomology using a 2-dimensional TQFT characterized by an identical twin Frobenius algebra.

 Is there a more natural $sl(2)$ TQFT (co)homology construction that takes as inputs webs with arbitrarily many pairs of bivalent vertices---thus a construction that doesn't require first a vertex-reduction (resolution-simplification)? And if so, what is the associated algebra structure of that TQFT?

\end{document}